\documentclass[a4,11pt]{article} \usepackage{amssymb}
 \usepackage{amsthm}\usepackage{amsmath}\usepackage{latexsym}\usepackage{srcltx} 

\newtheorem{teo}{Theorem}[section]
\newtheorem{oss}[teo]{Remark}
\newtheorem{Prop}[teo]{Proposition}
\newtheorem{war}[teo]{Warning}
\newtheorem{lemma}[teo]{Lemma}

\newtheorem{Defi}[teo]{Definition}
\newtheorem{es}[teo]{Example}
\newtheorem{corollario}[teo]{Corollary}
\newtheorem{no}[teo]{Notation}

\newtheorem{Problem}[teo]{Problem}

\newcommand{\cc}{_{^{_\HH}}}
\newcommand{\vv}{_{^{_\VV}}}
\newcommand{\vs}{_{^{_{\VS}}}}

\newcommand{\res}{\mathop{\hbox{\vrule height 7pt width .5pt depth 0pt
\vrule height .5pt width 6pt depth 0pt\,}}\nolimits}

\def \op{^\perp}

\def \DH{{h}}

\newcommand{\LL}{\mathop{\hbox{\vrule height .5pt width 6pt depth
0pt \vrule height 7pt width .5pt depth 0pt\,}}\nolimits}
\newcommand{\rr}{_{^{_\mathit{R}}}}

\newcommand{\ck}{_{^{_{\HH_k}}}}
\newcommand{\ci}{_{^{_{\HH_i}}}}
\newcommand{\ctr}{_{^{_{\HH_3}}}}
\newcommand{\cd}{_{^{_{\HH_2}}}}

\newcommand{\ciss}{_{^{_{\HH_i\!S}}}}

\newcommand{\cj}{_{^{_{\HH_j}}}}

\def\tst{_{^{_{\TT S_t}}}}

\newcommand{\ct}{_{^{_{\HH_t}}}}

\def \cont{{\mathbf{C}}}
\def \cin{{\mathbf{C}^{\infty}}}

\def\dim {\mathrm{dim}}
\def\dc {d_{CC}}

\def\ss{_{^{_{\HS}}}}
\def\ts{_{^{_{\TS}}}}

\def\g{\mathit{g}\cc}

\def\dg{\textit{grad}\cc}
\def\qq{\textit{grad}\ss}

\def \nis {\sigma^{n-2}\cc}

\def \per {\sigma^{n-1}\cc}
\def \perh {\sigma^{2r}\cc}

\def\WW{\widetilde{X}}
\def\SC{{C^{\gg}}}

\def\UU{\mathcal{U}}

\def \op{^\bot}

\def \ot{^\top}

\def \TB{\tau^{_{\TT\!S}}}

\def \nn{\nu\cc}

\def \XH{\mathfrak{X}(\HH)}
\def \XX{\mathfrak{X}}

\def \MS{\mathcal{H}\cc}

\def \P{{\mathcal{P}}}

\def \PH{\P\cc}

\def \Om{\Omega}

\def \Lie{\mathcal{L}}

\def \ee{\mathrm{e}}

\def \R{\mathbb{R}}
\def \Rn{\mathbb{R}^{\DN}}

\def \div{\mathit{div}}

\def \GG{\mathbb{G}}
\def \gg{\mathfrak{g}}

\def \pert{\left(\per\right)_t}

\def\divh{\div\cc}

\def\grr{{{\mathtt{gr}}}}

\def\tsc{\nabla^{^{_{\TT\!{S}}}}}
\def\gs{\nabla^{_{\HS}}}
\def\gc{\nabla^{_{\HH}}}

\def\UU{\mathcal{U}}

\def\UU{\mathcal{U}}

\def \nn{\nu_{_{\!\HH}}}
\def \XG{\mathfrak{X}(\GG)}

\def \Lie{\mathcal{L}}

\def \ee{\mathrm{e}}

\def \Om{\Omega}
\def \Rn{\mathbb{R}^{n}}
\def \R{\mathbb{R}}
\def \N{\mathbb{N}}

\def \cji {c_{j\,i}(x)}
\def \C { C(x):=[\cji]_{j,i},\,\, {j=1,\ldots,m \,,\, i=1,\ldots,n}}

\def \Qdim {Q:=\sum_{i=1}^{k}i\,\DH_i}

\def \X {X=(X_{1}, \ldots, X_{m_1})}

\def \X0 {X_{1}(0)\!=\!\partial_{x_{1}}, \ldots, X_{m_1}(0)\!=\!\partial_{x_{m_1}}}

\def \HG {\HH\GG}
\def \HS {\HH\!{S}}
\def \VS {\VV\!S}

\def \TG {\mathit{T}\GG}
\def \HH {\mathit{H}}

\def \VV {\mathit{V}}
\def \TT {\mathit{T}}
\def \TS {\mathit{T}\!S}

\def \grad{\textit{grad}}
\def \C0H{\mathbf{C}_{0}^{\infty}(U,\HG)}
\def \C00{\mathbf{C}_{0}^{\infty}(U)}
\def \C01{\mathbf{C}_{0}^{1}(U)}

\def \L1{d\,\mathcal{L}^n}

\def \Ar{\mathcal{H}^{n-1}_{Eu}}

\def \Vol{d\,\mathcal{V}ol^n}

\def \exp{\textsl{exp\,}}

\def \esp{\textsl{exp\,}}
\def \exsp{{\textit{exp}\,}}

\def \Eu {_{\rm Eu}}

\def \Om{\Omega}
\def \Rn{\mathbb{R}^{n}}
\def \R{\mathbb{R}}
\def \N{\mathbb{N}}

\def \cji {c_{j\,i}(x)}
\def \C { C(x):=[\cji]_{j,i},\,\, {j=1,\ldots,m \,,\, i=1,\ldots,n}}

\def \GG{\mathbb{G}}
\def \gg{\mathfrak{g}}

\def \X {X=(X_{1}, \ldots, X_{m_1})}

\def \X0 {X_{1}(0)\!=\!\partial_{x_{1}}, \ldots, X_{m_1}(0)\!=\!\partial_{x_{m_1}}}

\def \HG {\mathit{H}}

\def \C0H{\mathbf{C}_{0}^{\infty}(\Om,\HG)}
\def \C00{\mathbf{C}_{0}^{\infty}(\Om)}
\def \C01{\mathbf{C}_{0}^{1}(\Om)}

\def \exp{\textsl{exp\,}}

\def \esp{\textsl{exp\,}}

\def\GG{\mathbb{G}}

\title{Isoperimetric and Sobolev inequalities
 on hypersurfaces in sub-Riemannian Carnot groups}
\author{Francescopaolo Montefalcone
\thanks{F. M. was partially supported by CIRM, Fondazione Bruno Kessler, University of Trento,
Italy and by  Fondazione CaRiPaRo Project ``Nonlinear Partial Differential Equations: models, analysis, and control-theoretic problems".}}

\begin{document}

\maketitle

\begin{abstract}
The geometric setting of this paper\footnote{\bf Warning: \rm In this last version, we have corrected some mistakes and imprecisions. In particular, we have  improved the estimate of the quantity $\mathcal B_2(t)$; see Section \ref{perdindirindina}. This allows us to state a more precise formulation of the main results.  We refer the reader to Section \ref{mike0} for more detailed comments on this final version and, in particular, to Warning \ref{cpzzo}.} is that of smooth sub-manifolds immersed in a sub-Riemannian
$k$-step Carnot group $\GG$ of homogeneous dimension $Q$. Our main
result is an  isoperimetric-type inequality  for the $\HH$-perimeter measure $\per$ in the case of a
compact hypersurface $S$ of class $\cont^2$  with (or without)
boundary $\partial S$; see Theorem \ref{ghaioio}. This result generalizes an inequality 
involving the mean curvature of the hypersurface, proven by
Michael and Simon \cite{MS} and Allard \cite{Allard},
independently. Finally, we prove some related Sobolev-type inequalities; see Section \ref{sobineqg}.
\end{abstract}

\tableofcontents

\date{}

\normalsize

\section{Introduction}

In the last  decades considerable efforts have been made to extend
to the general setting of metric spaces the methods of Analysis
and Geometric Measure Theory. This philosophy,  in a sense already
contained in Federer's treatise \cite{FE}, has been pursued, among
other authors, by Ambrosio \cite{A2}, Ambrosio and Kirchheim
\cite{AK1}, Capogna, Danielli and Garofalo \cite{CDG}, Cheeger
\cite{Che}, Cheeger and Kleiner \cite{Cheeger1}, David and Semmes
\cite{DaSe}, De Giorgi \cite{DG}, Gromov \cite{Gr1}, Franchi,
Gallot and Wheeden \cite{FGW}, Franchi and Lanconelli
\cite{FLanc},  Franchi, Serapioni and Serra Cassano \cite{FSSC3,
FSSC5}, Garofalo and Nhieu \cite{GN}, Heinonen and Koskela
\cite{HaKo},   Koranyi and Riemann \cite{KR}, Pansu \cite{P1, P2},
but the list is far from being complete.

In this respect, {\it sub-Riemannian} or {\it
Carnot-Carath\'eodory} geometries have become  a subject of great
interest also because of their connections with many different
areas of Mathematics and Physics, such as PDE's, Calculus of
Variations, Control Theory, Mechanics and Theoretical Computer
Science. For references, comments and other perspectives, we refer
the reader to Montgomery's book \cite{Montgomery} and the surveys
by Gromov, \cite{Gr1}, and Vershik and Gershkovich, \cite{Ver}. We
also mention, specifically for sub-Riemannian geometry,
 \cite{Stric} and \cite{P4}. More recently, the
 so-called Visual Geometry has also received new impulses from this
 field; see \cite{SCM}, \cite{CMS} and references therein.

The setting of the sub-Riemannian geometry is that of a smooth
manifold $N$, endowed with a smooth non-integrable distribution
$\HH\subset\TT N$ of $\DH$-planes, or {\it horizontal subbundle}
($\DH\leq\dim N$), where a metric $g\cc$ is defined. The
manifold $N$ is said to be a {\it Carnot-Carath\'eodory space} or
{\it CC-space}  when one introduces the so-called {\it CC-metric}
$\dc$; see Definition \ref{dccar}. With respect to such a metric,
the only paths on $N$ which have finite length are
tangent to $\HH$ and therefore called {\it
horizontal}. Roughly speaking, in connecting two points we are
only allowed to follow horizontal paths joining them.

A $k$-{\it{step Carnot group}}
$(\GG,\bullet)$ is an $n$-dimensional, connected, simply
connected, nilpotent and stratified Lie group (with respect to the
group multiplication $\bullet$) whose Lie algebra $\gg\cong\TT_0\GG$
satisfies the following:\[ {\mathfrak{g}}={\HH}_1\oplus...\oplus {\HH}_k,\quad
 [{\HH}_1,{\HH}_{i-1}]={\HH}_{i}\qquad\forall\,\,i=2,...,k,\quad
 {\HH}_{k+1}=\{0\}.\]We also set $\HH:=\HH_1$ and $\VV:={\HH}_2\oplus...\oplus {\HH}_k$. In the sequel, we shall refer to $\HH$ and $\VV$ as the \it horizontal space \rm and the \it vertical space\rm, respectively. Note that they also have a natural bundle structure, in which the basis is
the group $\GG$.  Let
$\underline{X\cc}:=\{X_1,...,X_{\DH}\}$ be a frame of left-invariant vector
fields for the horizontal layer $\HH$. This frame can be completed to a global
left-invariant frame $\underline{X}:=\{X_1,...,X_n\}$ for $\gg$.
In fact, the standard basis $\{\ee_i:i=1,...,n\}$ of $\Rn$ can be
relabeled to be {\it graded} or {\it adapted to the
stratification}. Every Carnot group $\GG$ is endowed with
a one-parameter group of positive dilations (adapted to the grading of $\gg$)
making it a {\it homogeneous group} of homogeneous dimension
$\Qdim\,(\DH_i=\dim \HH_i)$, in the sense of Stein's definition; see \cite{Stein}.
The number $Q$ coincides with the {\it Hausdorff dimension} of
$(\GG,\dc)$ as a metric space with respect to the CC-distance. Carnot groups are of special
interest for many reasons and, in particular, because they
constitute a wide class of examples of sub-Riemannian geometries.
Note that, by a well-know result due to Mitchell \cite{Mi} (see
also \cite{Montgomery}), the {\it
Gromov-Hausdorff tangent cone} at any regular point of a
sub-Riemannian manifold turns out to be a suitable Carnot group. This fact motivates
the interest towards Carnot groups  which play for
sub-Riemannian geometries  an analogous role to that of
 Euclidean spaces in Riemannian geometry. The initial development of Analysis in this setting was motivated
by some works published in the first eighties. Among others, we
cite the paper by Fefferman and Phong \cite{FePh} about the
so-called ``sub-elliptic estimates'' and that of Franchi and
Lanconelli \cite{FLanc}, where a H\"{o}lder regularity theorem was
proven for a class of degenerate elliptic operators in divergence
form. Meanwhile, the beginning of Geometric Measure Theory was
perhaps an intrinsic isoperimetric inequality proven by Pansu in
his thesis \cite{P1}, for the {\it Heisenberg group}
$\mathbb{H}^1$. For further results about isoperimetric inequalities
on Lie groups and Carnot-Carath\'eodory spaces, see also
\cite{Varo}, \cite{Gr1}, \cite{P4}, \cite{GN}, \cite{CDG},
\cite{FGW}, \cite{HaKo}. For results on these topics, and for more
detailed bibliographic references, we refer the reader to
\cite{A2},  \cite{CDG}, \cite{FSSC3, FSSC5}, \cite{DGN3},
\cite{G}, \cite{GN}, \cite{Mag, Mag2}, \cite{Montea, Monteb},
\cite{HP}. We also quote \cite{CCM}, \cite{vari}, \cite{G}, \cite{Pauls},
\cite{RR}, for some results about minimal and constant
mean-curvature hypersurfaces immersed in Heisenberg groups.

In this paper we are concerned with
hypersurfaces immersed  in Carnot groups, endowed with the
so-called {\it $\HH$-perimeter measure} $\per$; see Definition
\ref{sh}. We  first study some technical
tools. In particular, we extend to hypersurfaces with
non-empty characteristic sets,  the 1st-variation of $\per$,  proved in \cite{Monte, Monteb}
for the non-characteristic case; see Section \ref{prvar0}. We then  discuss a blow-up theorem, which
also holds for characteristic points and a horizontal Coarea
Formula for smooth functions on hypersurfaces; see Section
\ref{blow-up} and Section \ref{COAR}. In Section
\ref{mike}, these results will be used  to investigate the validity in this context of a  monotonicity inequality for the $\HH$-perimeter and of a related isoperimetric inequality. These results were proved by Michael and Simon in \cite{MS}
for a general setting including Riemannian geometries and,
independently,
 by Allard in \cite{Allard} for varifolds; see below for a more precise statement.
In Section \ref{sobineqg}, we shall deduce some related Sobolev-type
inequalities, following a classical pattern by Federer-Fleming
\cite{FedererFleming} and Mazja \cite{MAZ}.  We here observe that similar results in this
direction have been obtained by Danielli, Garofalo and Nhieu
in \cite{DGN3}, where a monotonicity estimate for the
$\HH$-perimeter has been proven for graphical strips in the
Heisenberg group $\mathbb{H}^1$.

Now we would like to make a short comment about the Isoperimetric
Inequality for compact hypersurfaces immersed in the Euclidean
space $\Rn$.

 \begin{teo}[Euclidean Isoperimetric Inequality for $S\subset\Rn$]\label{w33w}Let
$S\subset\Rn\,(n>2)$ be a compact hypersurface of class $\cont^2$ 
with -or without- piecewise  $\cont^1$ boundary. Then
\[\left(\sigma^{n-1}\rr(S)\right)^{\frac{n-2}{n-1}}\leq C_{Isop}\left(\int_S|\mathcal{H}\rr|\,\sigma^{n-1}\rr+\sigma^{n-2}\rr(\partial
S)\right)\]where $C_{Isop}>0$ is a dimensional constant.\end{teo}

In the above statement, $\mathcal{H}\rr$ is the mean curvature and $\sigma^{n-1}\rr$ and $\sigma^{n-2}\rr$ denote, respectively, the Riemannian
measures on $S$ and $\partial S$.
 The first step in the proof  is a linear isoperimetric inequality. More
 precisely, one proves that
\[\sigma^{n-1}\rr(S)\leq r\left(\int_S|\mathcal{H}\rr|\,\sigma^{n-1}\rr+\sigma^{n-2}\rr(\partial
S)\right),\]where $r$ is the radius of a Euclidean ball $B(x, r)$
containing $S$. Starting from this linear inequality and using Coarea Formula,
one gets the so-called {\it monotonicity inequality}, that is, 
\[-\frac{d}{dt}\frac{\sigma^{n-1}\rr(S_t)}{t^{n-1}}\leq
\frac{1}{t^{n-1}}\left(\int_{S_t}|\mathcal{H}\rr|\,\sigma^{n-1}\rr
+ \sigma^{n-2}\rr(\partial S\cap B(x,t))\right)
\]for every $x\in {\rm Int}\,S$, for $\mathcal{L}^1$-a.e. $t>0$, where $S_t=S\cap B(x, t)$. (Note
that every interior point of a $\cont^2$ hypersurface $S$
is a {\it density-point}, that is, $\lim_{t\searrow
0^+}\frac{\sigma\rr^{n-1}(S_t)}{t^{n-1}}=\omega_{n-1}$, where
$\omega_{n-1}$ denotes the  measure of the unit ball in
$\R^{n-1}$).

By applying the monotonicity inequality along with a contradiction argument,  one
obtains a calculus lemma
 which, together with a standard Vitali-type
covering theorem, allows to achieve the proof of Theorem \ref{w33w}. We also remark that the monotonicity inequality  is equivalent to an asymptotic
exponential estimate, that is, \[\sigma^{n-1}\rr(S_t)\geq
\omega_{n-1}\, t^{n-1} e^{-\mathcal{H}^0 t}\]for $t\rightarrow
0^+$, where $x\in {\rm Int}\,S$ and $\mathcal{H}^0$ is any positive constant
such that $|\mathcal{H}\rr|\leq\mathcal{H}^0$. In case of
minimal hypersurfaces (that is,  $\mathcal H\rr=0$), this implies  that
$\sigma^{n-1}\rr(S_t)\geq \omega_{n-1}\, t^{n-1}$ as  $t\rightarrow
0^+$.

We now give a quick overview of the paper.

Section \ref{0prelcar}  introduces Carnot groups, immersed hypersurfaces and submanifolds. In particular,
we  describe some geometric structures and basic facts about
stratified  Lie groups, Riemannian and
sub-Riemannian geometries, intrinsic measures and connections.

If $S\subset\GG$ is a hypersurface of class $\mathbf{C}^1$, then
$x\in S$ is a {\it characteristic point} if $\HH_x\subset\TT_x S$.
If $S$ is non-characteristic,  the {\it unit $\HH$-normal} along
$S$ is given by $\nn: =\frac{\PH\nu}{|\PH\nu|}$, where $\nu$ is
the Riemannian unit normal of $S$ and $\PH:\gg\longrightarrow\HH$ is the orthogonal projection operator onto $\HH$. By means of the {\it
contraction operator} $\LL$ on differential
forms\footnote{\label{contraction}Recall that $\LL:
\Om^{k}(\GG)\rightarrow\Om^{k-1}(\GG)$
is defined, for $X\in\XX(\TG)$ and
$\alpha\in\Om^k(\GG)$, by
\begin{eqnarray*}(X \LL \alpha)
(Y_1,...,Y_{k-1}):=\alpha(X,Y_1,...,Y_{k-1})\end{eqnarray*} This operator extends, in a simple way, to $p$-vectors; see
\cite{Helgason}, \cite{FE}.}, we can define a differential $(n-1)$-form $\per\in\Om^{n-1}(S)$ as
\[\per:=(\nn
\LL \Vol)|_S,\]where
$\Vol:=\bigwedge_{i=1}^n\omega_i\in
\Om^n(\GG)$ denotes the Riemannian (left-invariant)
volume form on $\GG$ (obtained by wedging together the elements of the \textquotedblleft dual\textquotedblright
basis $\underline{\omega}=\{\omega_1,...,\omega_n\}$ of
$\gg$, where $\omega_i=X_i^\ast\in\Om^1(\GG)$, for every $i=1,...,n$). Notice that this  $(n-1)$-form is $(Q-1)$-homogeneous. By integrating the $(n-1)$-form $\per$ along $S$ we obtain the so-called $\HH$-perimeter measure. Note that the characteristic set $C_S$ of $S$ can
be seen as the set of all points at which the horizontal
projection of the unit normal vanishes, that is,  $C_S=\{x\in S:
|\PH\nu|=0\}$.

Analogously, we can define a $(Q-2)$-homogeneous measure $\nis$ on
any $(n-2)$-dimensional smooth submanifold $N$ of $\GG$. To this aim,  let
$\nn=\nn^1\wedge\nn^2$  be a horizontal  unit  normal $2$-vector to $N$; see Definition
\ref{dens}.  Then, we  obtain a $(Q-2)$-homogeneous
measure by integrating the  differential $(n-2)$-form $\nis:=(\nn
\LL \Vol)|_N$. The measures $\per$ and $\nis$ turn out to be
equivalent (up to bounded densities  called {\it metric
factors}; see \cite{Mag, Mag2}), respectively, to the $(Q-1)$-dimensional and $(Q-2)$-dimensional
spherical Hausdorff measures $\mathcal{S}_\varrho^{Q-1}$ and
$\mathcal{S}_\varrho^{Q-2}$ associated with a
homogeneous distance $\varrho$ on $\GG$; see, for instance,  Section \ref{blow-up}.

\begin{oss}The stratification of  $\gg$ induces a  natural  decomposition
of the tangent space of any smooth hypersurface $S\subset \GG$. More precisely,  
we intersect  $\TT_x S\subset\TT_x\GG$ with
$\TT_x^i\GG=\oplus_{j=1}^i(\HH_j)_x$. Setting
$\TT^iS:=\TS\cap\TT^i\GG,$ $n'_i:=\dim\TT^iS$,
$\HH_iS:=\TT^iS\setminus \TT^{i-1}S$ and $\HS=\HH_1S$, yields
$$\TS:=\oplus_{i=1}^k\HH_iS$$ and $\sum_{i=1}^kn'_i=n-1.$
Henceforth, we shall set $\VS:=\oplus_{i=2}^k\HH_iS$. 
\end{oss}

Section \ref{Preliminaries} contains some technical preliminaries.
In Section \ref{COAR} we state a smooth Coarea Formula for the $\HS$-gradient. More precisely, let $S\subset\GG$ be a
 compact hypersurface  of class $\cont^2$  and let
$\varphi\in\mathbf{C}^1(S)$. Then
\[\int_{S}\psi(x)|\qq\varphi(x)|\,\per(x)=\int_{\R}ds\int_{\varphi^{-1}[s]\cap
S}\psi(y)\,\nis(y)\]for every $\psi\in L^1(S; \per)$.

In Section \ref{prvar0} we   discuss the 1st variation formula of $\per$; see Theorem \ref{1vg}. This result, proved in \cite{Monte, Monteb} for non-characteristic hypersurfaces, is  generalized to the case of non-empty
characteristic sets. Roughly speaking, we shall show that the \textquotedblleft infinitesimal\textquotedblright 1st variation of $\per$  is given by
$$\Lie_W\per=\left(-\MS\langle W,\nu\rangle +\div\ts \left( W\ot|\PH\nu|-\langle W,\nu\rangle\nn\ot \right)\right)\,\sigma\rr^{n-1}.$$(Here $\Lie_W\per$ denotes the Lie derivative of $\per$ with respect to the initial velocity $W$  of the  variation, $\MS=-\div\cc\nn$ is the so-called \it horizontal mean curvature \rm of $S$; moreover, the symbols $W\op,\,W\ot$ denote the normal and tangential components of $W$, respectively). If $\MS$ is $L^1(S; \sigma\rr^{n-1})$, then the function $\Lie_W\per$ turns out to be integrable on $S$ and the integral of $\Lie_W\per$ on $S$ gives the 1st variation of $\per$. Note  that the 
third term in the previous formula depends on the normal component of $W$.  We stress that this term was omitted in \cite{Monteb}. Using a generalized divergence-type formula, the divergence term  can be integrated on the boundary. 
It is worth observing that a central point of this paper is to make an appropriate choice of the variation vector field  in the 1st variation formula \eqref{fva2formula}; see Theorem \ref{1vg}.

In Section \ref{blow-up}  we state a blow-up theorem for the
horizontal perimeter $\per$. In other words, we  study the density of $\per$ at $x\in S$, or the
limit\[\lim_{r\rightarrow 0^+}\frac{\per (S\cap
B_{\varrho}(x,r))}{r^{Q-1}},\]where $B_{\varrho}(x,r)$ is a homogeneous
$\varrho$-ball of center $x\in S$ and radius $r$.  We first discuss the  blow-up procedure at
non-characteristic points of a $\cont^1$  hypersurface $S$; see,
for instance, \cite{FSSC3, FSSC5}, \cite{balogh}, \cite{Mag,
Mag2}. Then, under  more regularity assumptions
on $S$, we tract the characteristic case; see
Theorem \ref{BUP}. A similar result can be found in \cite{Mag8} for  submanifolds of
$2$-step groups.

Section \ref{mike} is devoted to our main results, that are a (global) monotonicity inequality  for the $\HH$-perimeter and an isoperimetric-type inequality for compact
hypersurfaces with (or without) boundary, depending on the
horizontal mean curvature $\MS$. This
extends to  Carnot groups an inequality proved by Michael and
Simon \cite{MS} and Allard \cite{Allard}, independently. 

\begin{no}Set  ${\bf r}(S):=\sup_{x\in {\rm Int}(S\setminus C_S)} r_0(x),$ where $r_0(x)=2\left(\frac{\per({S})}{k_\varrho(\nn(x))}\right)^{{1}/{Q-1}}$ and $k_\varrho(\nn(x))$ denotes the \rm metric factor \it at $x$; see Section \ref{dere} and Section \ref{blow-up}. 
\end{no}

\begin{teo}[Isoperimetric-type Inequality]\label{ghaioiuuo0} Let
$S\subset\GG$ be a compact hypersurface of class $\cont^2$  with
boundary $\partial S$  (piecewise) $\cont^1$ and assume that the horizontal mean curvature $\MS$ of $S$ is integrable, that is,  $\MS\in L^1(S; \sigma\rr^{n-1})$.There exists  $C_{Isop}>0$ only
dependent on $\GG$ and on the homogeneous metric $\varrho$ such that   \begin{eqnarray}\label{2gha}\left(\per({S})\right)^{\frac{Q-2}{Q-1}}\leq
C_{Isop}\left(\int_S
|\MS|\,\per +\nis(\partial S)+ \sum_{i=2}^k  \left( {\bf r}(S)\right)^{i-1} \int_{\partial
S }|\P\ciss\eta|\,\sigma\rr^{n-2} \right),\end{eqnarray}where $\eta$ denotes the outward-pointing unit normal along $\partial S$ and $\P\ciss$ denotes the orthogonal projection onto $\HH_i S$. 
In particular, if $\partial S=\emptyset$, it follows that \begin{eqnarray}\label{2gha}\left(\per({S})\right)^{\frac{Q-2}{Q-1}}\leq
C_{Isop} \int_S
|\MS|\,\per.\end{eqnarray}
\end{teo}

In order to better understand this result, we  refer the reader to Section \ref{mike0}; see, in particular, Example \ref{zazaz1} and Warning \ref{cpzzo}. We also formulate an open problem which is intimately connected with the previous result; see Problem \ref{pope}.

The proof of this result is heavily inspired from the classical
one, for which we refer the reader to the book by Burago and
Zalgaller \cite{BuZa}. A similar strategy was useful in
proving isoperimetric and Sobolev inequalities in abstract metric
setting such as weighted Riemannian manifolds and graphs; see
\cite{CGY}.

The starting point is a linear isoperimetric inequality (see
Proposition \ref{correctdimin}) that is used to obtain a {\it
monotonicity formula} for the $\HH$-perimeter; see Theorem
\ref{rmonin}. This formula is one of our main results. We remark that, exactly as in the Euclidean/Riemannian case, the
monotonicity inequality is an ordinary differential inequality,
concerning the first derivative of  the density-quotient
$$\frac{\per (S_t)}{t^{Q-1}},$$ where $S_t:=S\cap B_{\varrho}(x,t))$ and $x\in {\rm Int}\,(S\setminus C_S)$; see Section
\ref{wlineq}.  We observe that, in the case of smooth hypersurfaces without boundary, we shall show that for every $x\in {\rm Int}(S\setminus C_S)$ the
following ordinary differential inequality
$$ -\frac{d}{dt}\frac{\per(S_t)}{t^{Q-1}}\leq
\frac{1}{t^{Q-1}} \int_{S_t}|\MS|\,\per$$
holds for $\mathcal{L}^1$-a.e. $t>0$. Hence, if $\MS=0$ it follows that $\frac{d}{dt}\frac{\per(S_t)}{t^{Q-1}}\geq 0$ for 
 $\mathcal{L}^1$-a.e. $t>0$.

In Section \ref{perdindirindina} we will discuss some  explicit 
estimates  and then, in
 Section \ref{isopineq1}, we will  prove the Isoperimetric Inequality.

In Section \ref{asintper}
we  give some straightforward applications of the monotonicity
formula. More precisely, let $S\subset\GG$ be a  
hypersurface  of class $\cont^2$, let $x\in
{\rm Int} (S\setminus C_S)$ and, without loss of generality, assume that, near $x$, the horizontal mean curvature $\MS$
is bounded by a positive constant $\MS^0$. Then, we will show that
\[\per({S}_t)\geq \kappa_{\varrho}(\nn(x))\,t^{Q-1}
e^{-t\,\MS^0}\]as long as $t\searrow 0^+$, where
$\kappa_{\varrho}(\nn(x))$ denotes the   metric factor  at
$x$; see Corollary \ref{asynt}.
We also consider the case where $x\in C_S$;
see Corollary \ref{2asynt} and Corollary \ref{hasynt}.

Finally, in  Section \ref{sobineqg} we 
 discuss some related  inequalities which can 
be deduced by  the Isoperimetric Inequality, following a
 classical argument by Federer-Fleming
\cite{FedererFleming} and Mazja \cite{MAZ}. The main result is a Sobolev-type inequality for compact hypersurfaces without boundary.\begin{teo}Let $\GG$ be a $k$-step Carnot group endowed with a homogeneous metric $\varrho$ as in
Definition \ref{2iponhomnor1}. Let
$S\subset\GG$ be a compact hypersurface of class $\cont^2$  without
boundary. Let $\MS$ be the horizontal mean curvature of
$S$ and assume that $\MS\in L^1(S; \sigma\rr^{n-1})$.
Then\[\left(\int_S|\psi|^{\frac{Q-1}{Q-2}}\,\per\right)^{\frac{Q-2}{Q-1}}\leq
C_{Isop}\left( \int_{S}\left(|\psi|\,|\MS|
+|\qq\psi|\right)\,\per+ \sum_{i=2}^k  \left( {\bf r}(S)\right)^{i-1}   \int_S |\grad\ciss\psi|\,\sigma\rr^{n-1}\right) \]for every
$\psi\in\cont^1(S)$, where $C_{Isop}$ is the constant
appearing in Theorem \ref{ghaioiuuo0}.

\end{teo} 
\section{Carnot
groups, submanifolds and measures}\label{0prelcar}
\subsection{Sub-Riemannian Geometry of Carnot groups}\label{prelcar}
In this section we introduce basic definitions and main
features  of Carnot groups.
References for this large subject can be found in
\cite{CDG}, \cite{GN}, \cite{Gr1}, \cite{Mag}, \cite{Mi},
\cite{Montgomery}, \cite{P1, P2, P4}, \cite{Stric}. Let $N$  be a
${\cin}$-smooth connected $n$-dimensional manifold and let
$\HH\subset \TT N$ be an $\DH$-dimensional smooth subbundle of
$\TT N$. For any $x\in N$, let $\TT^{k}_x$ denote the vector
subspace of $\TT_x N$ spanned by a local basis of smooth vector
fields $X_{1}(x),...,X_{\DH}(x)$ for $\HH$ around $x$, together
with all commutators of these vector fields of order $\leq k$. The
subbundle $\HH$ is called {\it generic} if, for all $x\in N$,
$\dim \TT^{k}_x$ is independent of the point $x$ and {\it
horizontal} if $\TT^{k}_x = \TT N$, for some $k\in \N$. The pair
$(N,\HH)$ is a {\it $k$-step CC-space} if is generic and
horizontal and if $k=\inf\{r: \TT^{r}_x = \TT N \}$. In this case
\begin{equation*}0=\TT^{0}\subset
\HH=\TT^{\,1}\subset\TT^{2}\subset...\subset \TT^{k}=\TT
N\end{equation*} is a strictly increasing filtration of {\it
subbundles} of constant dimensions $n_i:=\dim
\TT^{i}\,\,i=1,...,{k}.$ Setting $(\HH_i)_x:=\TT^i_x\setminus
\TT^{i-1}_x,$ then $\grr(\TT_x N)=\oplus_{i=1}^k (\HH_k)_x$ is the
associated {\it graded Lie algebra}  at $x\in N$,  with respect to
the Lie product $[\cdot,\cdot]$. We set $\DH_i:=\dim
{\HH}_{i}=n_i-n_{i-1}\,(n_0=\DH_0=0)$ and, for simplicity,
$\DH:=\DH_1=\dim\HH$. The $k$-vector
$\overline{\DH}=(\DH,\DH_2,...,\DH_{k})$ is the {\it growth
vector} of $\HH$.

\begin{Defi}[Graded frame] We say that $\underline{X}=\{X_1,...,X_n\}$
is a { graded frame} for $N$  if $\{{X}_{i_j}(x): n_{j-1}<i_j\leq
n_j\}$ is a basis for ${\HH_j}_x$ for  any 
$j=1,...,k$ and for any $x\in N$.\end{Defi}

\begin{no}Let $E\subset\TT N$ be any smooth subbundle of $\TT N$. Throughout this paper, we  denote by $\XX^r(E)$ the space of  sections of class $\cont^r$ of $E$ $(r\geq 0)$. When $r=+\infty$, we  simply set $\XX(E)$. Furthermore, the space of differential  $p$-forms on $N$, is denoted as $\Om^p(N)$.\end{no}

\begin{Defi}\label{dccar} A { sub-Riemannian metric} $g\cc=\langle\cdot,\cdot\rangle\cc$ on $N$ is a
symmetric positive bilinear form on $\HH$. If $(N,\HH)$ is a
{CC}-space, the { {CC}-distance} $\dc(x,y)$ between $x, y\in N$ is
defined by
$$\dc(x,y):=\inf \int\sqrt{\langle
 \dot{\gamma},\dot{\gamma}\rangle\cc} dt,$$
where the infimum is taken over all piecewise-smooth horizontal
paths $\gamma$ joining $x$ to $y$.
\end{Defi}

In fact, Chow's Theorem implies that $\dc$ is metric on $N$
and that any two points can be joined with at least one horizontal
curve. The topology induced on $N$ by the {CC}-metric is equivalent
to the standard manifold topology; see \cite{Gr1},
\cite{Montgomery}.

This is the setting of sub-Riemannian geometry. An important class of these geometries is represented by {\it Carnot groups} which,
for many reasons, play  in sub-Riemannian geometry  an analogous
role to that of Euclidean spaces in Riemannian geometry.  For the geometry of Lie groups
we refer the reader to Helgason's book \cite{Helgason} and
Milnor's paper \cite{3}, while for  sub-Riemannian
geometry, to Gromov, \cite{Gr1}, Pansu, \cite{P1, P4}, and
Montgomery, \cite{Montgomery}.

A $k$-{\it{step Carnot group}}  $(\GG,\bullet)$ is an
$n$-dimensional, connected, simply connected, nilpotent and
stratified Lie group (with respect to the multiplication
$\bullet$). Let $0$ be the identity on $\GG$. The Lie algebra  $\gg\cong\TT_0\GG$ of $\GG$ is an $n$-dimensional vector space such that:\begin{equation*} {\mathfrak{g}}={\HH}_1\oplus...\oplus
{\HH}_k,\quad
 [{\HH}_1,{\HH}_{i-1}]={\HH}_{i}\quad\forall\,\,i=2,...,k,\,\,\,
 {\HH}_{k+1}=\{0\}.\end{equation*}

The first layer
  ${\HH}_1$ of the stratification of $\gg$ is called
{\it horizontal}  and denoted by $\HH$. Let
${\VV}:={\HH}_2\oplus...\oplus {\HH}_k$ be the {\it
vertical subspace} of $\gg$. We set
$\DH_i=\dim{{\HH}_i},$
 $n_i:=\DH_1+...+\DH_i$, for every $i=1,...,k$ ($\DH_1=\DH,\,n_k=n$). We assume that $\HH$
is generated by a frame $\underline{X\cc}:=\{X_1,...,X_{\DH}\}$ of
left-invariant vector fields. This frame can always be completed to a global,
 graded, left-invariant frame  $\underline{X}:=\{X_i:
 i=1,...,n\}$ for  $\gg$, in a way that
 ${\HH}_l={\mathrm{span}}_\R\big\{X_i:  n_{l-1}< i \leq
 n_{l}\big\}$ for $l=1,...,k$. In fact, the standard basis $\{\ee_i:i=1,...,n\}$ of $\Rn\cong\TT_0\GG$ can be
relabeled to be {\it adapted to the stratification}.
Note that each left-invariant vector field of the frame $\underline{X}$ is given by
${X_i}(x)={L_x}_\ast\ee_i\,(i=1,...,n)$, where ${L_x}_\ast$
is the differential of the left-translation by $x$.

\begin{no}\label{1notlne0}We  denote by $\P\ci:\gg\longrightarrow\HH_i$ the orthogonal projection map from $\gg$ onto $\HH_i$ for any $i=1,...,k$. In particular, we set $\P\cc:=\P{_{^{_{\HH_1}}}}$.  Analogously, we   denote by $\P\vv:\gg\longrightarrow\VV$ the orthogonal projection map from $\gg$ onto $\VV$.  
\end{no}

\begin{no}\label{1notazione}We  set $I\cc:=\{1,...,h\}$,
$I\cd:=\{n_1+1,...,n_2\}$,..., $I\vv:=\{h +1,...,n\}$. Unless
otherwise specified, Latin letters $i, j, k,...$  are used for
indices belonging to $I\cc$ and Greek letters
 $\alpha, \beta, \gamma,...$ for indices belonging to $I\vv$. The
  function \textquotedblleft order\textquotedblright\, $\mathrm{ord}:\{1,...,n\}\longrightarrow\{1,...,k\}$
is defined by $\mathrm{ord}(a):= i $, whenever  $n_{i-1}<a\leq
n_{i}$,  $i=1,...,k$.
\end{no}

We use  exponential coordinates of 1st
kind so that $\GG$ will be identified with its Lie algebra $\gg$,
via the (Lie group) exponential map $\exp:\gg\longrightarrow\GG$; see \cite{Vara}.
The {\it Baker-Campbell-Hausdorff
formula} gives the group law $\bullet$ of the group
$\GG$,  starting from a corresponding operation on the Lie algebra $\gg$.
In fact, one has
$$\exp(X)\bullet\exp(Y)=\exp(X\star Y)$$ for any $X,\,Y \in\gg$, where
${\star}:\gg \times \gg\longrightarrow \gg$ is the
 {\it Baker-Campbell-Hausdorff product} defined by \begin{eqnarray}\label{CBHf}X\star Y= X +
Y+ \frac{1}{2}[X,Y] + \frac{1}{12} [X,[X,Y]] -
 \frac{1}{12} [Y,[X,Y]] + \mbox{ brackets of length} \geq 3.\end{eqnarray}

In
exponential coordinates,
 the group law $\bullet$ on $\GG$ is polynomial and explicitly computable; see
\cite{Corvin}. Note that $0=\exp(0,...,0)$ and the inverse of each point
 $x=\exp(x_1,...,x_{n})\in\GG$ is given by ${x}^{-1}=\exp(-{x}_1,...,-{x}_{n})$.

When $\HH$ is endowed with a metric
$g\cc=\langle\cdot,\cdot\rangle\cc$, we say that $\GG$
 has a {\it sub-Riemannian structure}.  It is always possible to define a left-invariant Riemannian metric
  $g =\langle\cdot,\cdot\rangle$ on $\gg$ such that $\underline{X}$ is {\it
orthonormal} and $g_{|\HH}=\g$. Note that, if we fix a Euclidean metric on
$\Rn\cong \TT_0\GG$  such that $\{\ee_i: i=1,...,n\}$  is an orthonormal
basis, this metric extends to each $\TT_x\GG$ ($x\in\GG$)
by left-translations. Since Chow's Theorem trivially holds true for Carnot groups, the {\it
Carnot-Carath\'eodory distance} $\dc$ associated with $g\cc$ can
be defined and the pair $(\GG,\dc)$ turns out to be a complete metric
space where every couple of points can be joined by  at least one $\dc$-geodesic.
Carnot groups are {\it homogeneous groups}, in the sense that they
admit a 1-parameter group of automorphisms
$\delta_t:\GG\longrightarrow\GG$ $(t\geq 0)$ defined by
\begin{eqnarray*}\delta_t x
:=\exp\left(\sum_{j=1}^k\sum_{i_j\in I\cj}t^j\,x_{i_j}\ee_{i_j}\right),\end{eqnarray*}
for any $x=\exp\left(\sum_{j,i_j}x_{i_j}\ee_{i_j}\right)\in\GG.$ By definition, the
{\it homogeneous dimension} of $\GG$ is the positive integer  $\Qdim,$ 
coinciding with the {\it Hausdorff dimension} of $(\GG,\dc)$ as a
metric space;
 see \cite{Mi}, \cite{Gr1}.
\begin{Defi}\label{hometr}A continuous distance
$\varrho:\GG\times\GG\longrightarrow\R_+$ is called {\rm homogeneous}
if, and only if, the following hold:
\begin{itemize}\item[{\rm(i)}]$\varrho(x,y)=\varrho(z\bullet x,z\bullet
y)$ for every $x,\,y,\,z\in\GG$;
\item[{\rm(ii)}]$\varrho(\delta_tx,\delta_ty)=t\varrho(x,y)$ for
all $t\geq 0$. \end{itemize}\end{Defi}

The CC-distance $\dc$ is an example of homogeneous distance.
Another  example can be found in \cite{FSSC5}.\\

Any
Carnot group admits a smooth, subadditive, homogeneous
norm (see \cite{HeSi}), that is, there exists a continuous function
$\|\cdot\|_\varrho:\GG\times\GG\longrightarrow\R_+\cup \{0\}$  which is  smooth on
$\GG\setminus\{0\}$ and such
that:\begin{itemize}\item[{\rm(i)}]$\|x\bullet
y\|_\varrho\leq\|x\|_\varrho+\|y\|_\varrho$;
\item[{\rm(ii)}]$\|\delta_tx\|_\varrho=t\|x\|_\varrho\quad (t>
0)$;\item[{\rm(iii)}]$\|x\|_\varrho=0\Leftrightarrow x=0$;
\item[{\rm(iv)}]$\|x\|_\varrho=\|x^{-1}\|_\varrho$.
\end{itemize}

\begin{Defi}\label{2iponhomnor1}Let  $\varrho:\GG\times \GG\longrightarrow\R_+$
be a homogeneous distance such that:
\begin{itemize}\item[{\rm(i)}]$\varrho$ is (piecewise)
$\cont^1$;\item[{\rm(ii)}]$|\grad\cc\varrho|\leq 1$ at each
regular point of $\varrho$;
\item[{\rm(iii)}]${|x\cc|}\leq{\varrho(x)}$ for every $x\in\GG$, where
$\varrho(x)=\varrho(0,x)=\|x\|_\varrho$. Furthermore, we shall
assume that there exist constants ${\bf c}_i\in\R_+$ such that
$|x\ci|\leq {\bf c}_i
\varrho^i(x)$ for any $i=2,...,k.$\end{itemize}
\end{Defi}
\begin{es}\label{distusata}
A smooth homogeneous norm $\varrho$  on
$\GG\setminus \{0\}$, can be defined by setting
\begin{equation}\label{distusata0}
\|x\|_\varrho:=\left(|x\cc|^{\lambda}+C_2|x\cd|^{\lambda/2}+C_3|x\ctr|^{\lambda/3}+...+C_k|x\ck|^{\lambda/k}
\right)^{1/\lambda}, 
\end{equation}where $\lambda$ is any positive number evenly
divisible by $i=1,...,k$ and $|x\ci|$ denotes the
Euclidean norm of the projection $x\ci$ of $x$ onto the i-th layer
$\HH_i$ of  $\gg$.
\end{es}
\begin{es}\label{Kor}Let us consider the case of Heisenberg groups
$\mathbb{H}^r$; see Example \ref{epocase}. It can be shown that the  CC-distance $\dc$ satisfies all the   assumptions of Definition \ref{2iponhomnor1}. Another  example is the so-called {Koranyi
norm}, defined by
$$\|y\|_\varrho:=\varrho(y)=\sqrt[4]{|y\cc|^4+16t^2}\qquad(y=\exp(y\cc,t)\in\mathbb{H}^r),$$ is homogeneous and
  $\cin$-smooth  outside  
$0\in\mathbb{H}^r$ and  satisfies \rm (ii) \it and \rm (iii) \it of Definition
\ref{2iponhomnor1}.
\end{es}

Since we have fixed a Riemannian metric on $\gg$, we may define the left-invariant
co-frame $\underline{\omega}:=\{\omega_i:i=1,...,n\}$ dual to
$\underline{X}$. In fact, the {\it left-invariant 1-forms}
\footnote{That is, $L_p ^{\ast}\omega_I=\omega_I$ for every
$p\in\GG.$} $\omega_i$ are uniquely determined by the condition:
$$\omega_i(X_j)=\langle X_i,X_j\rangle=\delta_i^j\qquad \mbox{for every}\,\,i,
j=1,...,n,$$ where $\delta_i^j$ denotes \textquotedblleft Kronecker delta\textquotedblright.
Recall that the {\it structural constants} of the Lie algebra
$\gg$ associated with the left invariant frame $\underline{X}$ are
defined by
$$\SC^r_{ij}:=\langle [X_i,X_j],
 X_r\rangle\qquad\mbox{for every}\,\, i, j, r=1,...,n. $$
\noindent {They satisfy the following:
\begin{itemize}\item [{\rm (i)}]\, $\SC^r_{ij} +\SC^r_{ji}=0$\,\,\, (skew-symmetry); \item
[{\rm(ii)}]\, $\sum_{j=1}^{n} \SC^i_{jl}\SC^{j}_{rm} +
\SC^i_{jm}\SC^{j}_{lr} + \SC^i_{jr}\SC^{j}_{ml}=0$\,\,\, (Jacobi's
identity).\end{itemize} \noindent The stratification  of
the Lie algebra implies the following structural  property:
\[ X_i\in {\HH}_{l},\, X_j \in
{\HH}_{m}\Longrightarrow [X_i,X_j]\in {\HH}_{l+m}.\]

\begin{Defi}[Matrices of structural constants]\label{nota}We set\begin{itemize}\item[{\rm(i)}]
$C^\alpha\cc:=[\SC^\alpha_{ij}]_{i,j\in
I\cc}\in\mathcal{M}_{h\times h}(\R)\,\,\,\qquad\forall\,\,\alpha\in I\cd$;
\item[{\rm(ii)}] $ C^\alpha:=[\SC^\alpha_{ij}]_{i, j=1,...,n}\in
\mathcal{M}_{n\times n}(\R)\qquad\forall\,\,\alpha\in
I\vv.$\end{itemize}The  linear operators associated with these matrices are denoted in the same way.
\end{Defi}

\begin{Defi}\label{parzconn}
Let $\nabla$ be the  (unique)  left-invariant
Levi-Civita connection on $\GG$ associated with the metric $g$. If
$X, Y\in\XH:=\cin(\GG,\HH)$, we set  $\gc_X Y:=\PH(\nabla_X
Y).$
\end{Defi}

\begin{oss} The operation $\gc$ is  called {\it horizontal $\HH$-connection}; see
\cite{Monteb} and references therein. By using  the properties of the structural constants of
the Levi-Civita connection, one can show that $\gc$ is {\rm flat},  that is, 
$\gc_{X_i}X_j=0$ for every $i,j\in I\cc$.    $\gc$ turns out to be
{\rm compatible with the sub-Riemannian metric} $g\cc$, that is, 
$X\langle Y, Z \rangle=\langle \gc_X Y, Z \rangle
+ \langle  Y, \gc_X Z \rangle$ for all $X, Y, Z\in
\XH$. Moreover, $\gc$ is  {\rm torsion-free}, that is, 
$\gc_X Y - \gc_Y X-\PH[X,Y]=0$ for all $X, Y\in \XH$. All these properties easily follow  from the very definition of $\gc$
together with the corresponding properties of the Levi-Civita connection
$\nabla$ on $\GG$. Finally, we recall a fundamental property of $\nabla$, that is, 
\begin{equation*}\nabla_{X_i} X_j =
\frac{1}{2}\sum_{r=1}^n\left( \SC_{ij}^r  - \SC_{jr}^i +
\SC_{ri}^j\right) X_r\qquad \forall\,\,i,\,j=1,...,n.\end{equation*}
\end{oss}

\begin{Defi}
If $\psi\in\cin({\GG})$ we define the  horizontal gradient of
$\psi$  as the  unique horizontal vector field $\dg \psi$ such
that
$\langle\dg \psi,X \rangle= d \psi (X) = X
\psi$ for all $X\in \XH.$  The  horizontal
divergence of $X\in\XH$, $\divh X$, is defined, at each point
$x\in \GG$, by
$$\divh X(x):= \mathrm{Trace}\big(Y\longrightarrow \gc_{Y} X
\big)(x)\quad(Y\in \HH_x).$$
\end{Defi}

\begin{es}[Heisenberg groups $\mathbb{H}^r$] \label{epocase}Let $\mathfrak{h}_r:=\TT_0\mathbb{H}^r=\R^{2r + 1}$
denote the Lie algebra of  $\mathbb{H}^r$. The only non-trivial algebraic
rules are given by $[\ee_{i},\ee_{i+1}]=\ee_{2r + 1}$ for every $i=2k+1$ where
$k=0,...,r-1$. We have
$\mathfrak{h}_r=\HH\oplus \R\ee_{2r+1},$ where $\HH={\rm
span}_{\R}\{\ee_i:i=1,...,2r\}$ and the second layer turns out to be the 1-dimensional center of $\mathfrak{h}_r$. The Baker-Campbell-Hausdorff formula determines the group
law $\bullet$. For every
$x=\exp\left(\sum_{i=1}^{2r+1}x_iX_i\right),\,
y=\exp\left(\sum_{i=1}^{2r+1}y_i X_i\right)\in \mathbb{H}^r$ one has\begin{center}$x\bullet y =\exp \left(x_1 + y_1,x_2+y_2,...,x_{2r} +
y_{2r}, x_{2r+1} + y_{2r+1} + \frac{1}{2}\sum_{k=1}^{r} (x_{2k-1}
y_{2k}- x_{2k} y_{2k-1})\right).$\end{center} The matrix of structural constants is given by
$$C\cc^{2r+1}:=\left|
\begin{array}{cccccc}
  0 & 1 & 0 & 0 & \cdot &\\
  -1 & 0 & 0 & 0 & \cdot &\\0 & 0 & 0 & 1 & \cdot &\\0 & 0 & -1 & 0 & \cdot &\\\cdot & \cdot  & \cdot
    & \cdot  & \cdot &\!\!\!\!
\end{array}%
\right|.$$\end{es}

\subsection{Hypersurfaces, homogeneous measures and geometric structures}\label{dere}
 Hereafter, $\mathcal{H}^m_{\varrho}$ and $\mathcal{S}^m_{\varrho}$ will denote
 the Hausdorff measure and the spherical
Hausdorff measure, respectively, associated with a homogeneous distance
$\varrho$ on $\GG$\footnote{We recall
that:\begin{itemize}\item[{\rm(i)}]
$\mathcal{H}_{{\varrho}}^{m}({S})=\lim_{\delta\to 0^+}\mathcal
{H}_{{\varrho},\delta}^{m}({S})$,  where
$${\mathcal{H}}_{{\varrho},\delta}^{m}({S})=
\inf\left\{\sum_i\left(\mathrm{diam}_\varrho ({C}_i)\right)^{m}:\;{S}
\subset\bigcup_i
{C}_i;\;\mathrm{diam}_\varrho({C}_i)<\delta\right\}$$ and the
infimum is taken with respect to any non-empty family of closed
subsets $\{{C}_i\}_i\subset\GG$;\item[{\rm(ii)}]
$\mathcal{S}_{{\varrho}}^{m}({S})=\lim_{\delta\to 0^+}\mathcal
{S}_{{\varrho},\delta}^{m}({S})$,  where
$${\mathcal{S}}_{{\varrho},\delta}^{m}({S})=
\inf\left\{\sum_i\left(\mathrm{diam}_\varrho({B}_i)\right)^{m}:\;{S}
\subset\bigcup_i {B}_i;\;\mathrm{diam}_\varrho
({B}_i)<\delta\right\}$$ and the infimum is taken with respect to
closed $\varrho$-balls ${B}_i$.\end{itemize}}

The Riemannian left-invariant volume form on $\GG$ is defined by
$\sigma\rr^n:=\bigwedge_{i=1}^n\omega_i\in
\Om^n(\GG)$. Hereafter we will set $\Vol:=\sigma\rr^n$. This  is the Haar
measure of $\GG$ and equals (the push-forward of) the
$n$-dimensional Lebesgue measure $\mathcal{L}^n$ on $\Rn\cong\TT_0\GG$.

\begin{Defi}\label{caratt}
Let $S\subset\GG$ be a  hypersurface  of class $\cont^1$. We
say that $x\in S$ is a {\rm characteristic point} if
$\dim\,\HH_x = \dim (\HH_x \cap \TT_x S)$ or, equivalently, if
$\HH_x\subset\TT_x S$. The {\rm characteristic set} $C_S$ of $S$ is the set of all characteristic points, that is,  $ C_S:=\{x\in S : \dim\,\HH_x = \dim (\HH_x \cap
\TT_x S)\}.$
\end{Defi}
Note that a hypersurface $S\subset\GG$ oriented by the outward-pointing normal
vector $\nu$ turns out to be {\it non-characteristic} if, and only if, the
horizontal space $\HH$ is {\it transversal} to $S$. We have to remark that the $(Q-1)$-dimensional CC-Hausdorff measure of $C_S$  vanishes, that is, 
$\mathcal{H}_{CC}^{Q-1}(C_S)=0$;  see \cite{Mag}. The
$(n-1)$-dimensional {\it Riemannian measure} along $S$ can be defined
by setting $\sigma^{n-1}\rr:=(\nu\LL\Vol)|_{S},$  where  $\LL$ denotes
the \it contraction operator \rm (or, \it interior product\rm) on differential
 forms; see footnote \ref{contraction}. Just as in \cite{Monte, Monteb} (see also
\cite{vari}, \cite{HP}, \cite{RR}), since we are studying \textquotedblleft smooth\textquotedblright
           hypersurfaces, instead of the variational definition \`a la De Giorgi (see, for instance,
\cite{FSSC3, FSSC5}, \cite{GN}, \cite{Monte} and bibliographies
therein) we  define an $(n-1)$-differential form which, by
integration on smooth boundaries, yields the usual $\HH$-perimeter
measure.

\begin{Defi}[$\per$-measure]\label{sh}
Let $S\subset\GG$ be a $\mathbf{C}^1$  non-characteristic
hypersurface and  $\nu$ the outward-pointing unit normal vector. We call {\rm
unit $\HH$-normal} along $S$ the normalized projection of $\nu$
onto $\HH$, that is,  $\nn: =\frac{\PH\nu}{|\PH\nu|}$. The
{\rm $\HH$-perimeter} along ${S}$ is the homogeneous measure associated
with the $(n-1)$-differential form $\per$ on $S$ given by  $\per:=(\nn \LL
\Vol)|_S.$ \end{Defi}If we allow $S$ to have characteristic
points, one extends $\per$ by setting $\per\res C_{S}= 0$. It
turns out that $\per = |\PH \nu |\,\sigma^{n-1}\rr $ and that $C_S=\{x\in S : |\PH\nu|=0\}$. We also remark that
$ \per(S\cap B)=k_{\varrho}(\nn)\,\mathcal{S}_{\varrho}^{Q-1}({S}\cap B)$ for all $B\in \mathcal{B}or(\GG)$, where the  bounded density-function $k_{\varrho}(\nn)$, called {\it metric factor}, only
depends on $\nn$ and on the (fixed) homogeneous metric $\varrho$ on
$\GG$; see \cite{Mag}; see also
Section \ref{blow-up}.

\begin{Defi}\label{carca}Setting $\mathit{H}_x S:=\HH_x\cap\TT_x S$ for every $x\in
S\setminus C_S$, yields $\HH_x=\mathit{H}_x
S\oplus{\rm span}_\R\{\nn(x)\}$ and this uniquely defines the subbundles
$\HS$ and $\nn S$, called, respectively, {\rm horizontal tangent
bundle} and {\rm horizontal normal bundle}. Note that $\dim \HH_x S=\dim \HH_x-1=2n-1$ at each non-characteristic point $x\in S\setminus C_S$.
\end{Defi}

The horizontal tangent space  is well defined even at the characteristic set $C_S$. More precisely, in this case $\HH_x S=\HH_x$ for every $x\in C_S$. However $\dim\HH_xS=\dim\HH_x=\DH$ for any $x\in C_S$.

For the sake of simplicity, unless otherwise mentioned, \it we  assume that $S\subset\GG$ is a $\cont^2$  non-characteristic hypersurface. \rm We first remark
that, if $\tsc$ is the connection induced  on $\TT S$ from the
Levi-Civita connection $\nabla$ on $\TG$\footnote{$\tsc$ is the Levi-Civita connection on $S$; see \cite{Ch1}.},
then $\tsc$ induces a \textquotedblleft partial connection\textquotedblright $\gs$ on
$\HS\subset\TT{S}$, that is defined by\footnote{The map
$\P\ss:\TT{S}\longrightarrow\HS$ denotes the orthogonal projection
onto $\HS$.}
$$\gs_XY:=\P\ss(\tsc_XY)\qquad\,\forall\,\,X,Y\in\XX^1(\HS):=\cont^1(S, \HS).$$
Note that the orthogonal decomposition
 $\HH=\HS\oplus\nn S$ enable us to define $\nabla^{_{\HS}}$ in analogy with the
 definition of ``connection on submanifolds''; see
\cite{Ch1}. More precisely, one has$$\gs_XY=\gc_X Y-\langle\gc_X
Y,\nn\rangle\,\nn\qquad\forall\,\,X,Y\in\XX^1(\HS).$$

\begin{Defi}Let $S\subset\GG$ be a $\cont^2$  non-characteristic hypersurface and  $\nu$ the outward-pointing unit normal vector. 
The $\HS$-{gradient}  $\qq\psi$ of $\psi\in \cont^1({S})$ is the unique
horizontal tangent vector field such that
$\langle\qq\psi,X \rangle= d \psi (X) = X
\psi$ for all $X\in \XX^1(\HS)$. We denote by $\div\ss$
the divergence operator on $\HS$, that is,  if $X\in\XX^1(\HS)$ and $x\in
{S}$, then
$$\div\ss X (x) := \mathrm{Trace}\big(Y\longrightarrow
\gs_Y X \big)(x)\quad\,(Y\in \HH_xS).$$The \rm horizontal 2{nd}
fundamental form \it of ${S}$ is the continuous map given by
${B\cc}(X,Y):=\langle\gc_X Y, \nn\rangle$ for every $X, Y\in\XX^1(\HS)$. 
The {\rm horizontal mean curvature} $\MS$ is  the trace of the linear operator ${{B}\cc}$,
that is,   $\MS:={\rm
Tr}B\cc=-\div\cc\nn.$ We set
\begin{itemize}\item[{\rm(i)}]$\varpi_\alpha:=\frac{\langle X_\alpha, \nu\rangle}{|\PH\nu|}\qquad (\nu_\alpha:=\langle X_\alpha, \nu\rangle)\qquad \forall\,\,\alpha\in
I\vv$;\item[{\rm(ii)}]$\varpi:=\frac{\P\vv\nu}{|\PH\nu|}=\sum_{\alpha\in I\vv}\varpi_\alpha
X_\alpha$;\item[{\rm(iii)}] $C\cc:=\sum_{\alpha\in {I\cd}}
\varpi_\alpha\,C^\alpha\cc$.\end{itemize}
\end{Defi}

Note that $\frac{\nu}{|\PH\nu|}=\nn+\varpi$. The horizontal 2nd
fundamental form ${B\cc}(X,Y)$ is a (continuous) bilinear form of
$X$ and $Y$. However, in general, $B\cc$ {\it is not symmetric} and so it can be written as a sum of two matrices, one
symmetric and the other skew-symmetric, that is,  $B\cc= S\cc + A\cc.$
It turns out that  $A\cc=\frac{1}{2}\,C\cc\big|_{\HS}$; see
 \cite{Monteb}.

\begin{Defi}\label{movadafr}Let $S\subset\GG$ be a $\cont^2$  hypersurface with boundary $\partial S$.  We  call {\rm adapted frame to $S$} any graded orthonormal frame
 $\underline{\tau}:=\{\tau_1,...,\tau_n\}$ for $\gg$  such that:
\begin{itemize}
 \item $ \tau_1(x)=\nn(x) \quad\forall\,\,x\in \overline{S}\setminus C_S,$\item $\HH_x S=\mathrm{span}\{ \tau_2(x),...,\tau_{\DH}(x)\}\quad\forall\,\,x\in \overline{S}\setminus C_S;$\item $\tau_\alpha:=
X_\alpha.$
\end{itemize}
Furthermore set
$\TB_\alpha:=\tau_\alpha -
\varpi_\alpha\tau_1$ for every $\alpha\in I\vv$. Note that
$\VV_xS= \mathrm{span}_{\R}\{\TB_\alpha(x): \alpha\in I\vv\},$
where $\VV_x S$ is the orthogonal complement of $\HH_x S$ in
$\TT_x S$, that is,   $\TT_x S=\HH_x S\oplus\VV_x S.$\end{Defi}
It is worth remarking that
  $$\underline{\tau}=\{\underbrace{\tau_1}_{=\nn},
 \underbrace{\tau_2,...,\tau_{\DH}}_{\mbox{\tiny{o.n. basis of}}\,\HS},\underbrace{\tau_{\DH+1},...,\tau_n}_{\mbox{\tiny o.n. basis of}\,\VV}\}.$$

\begin{oss}[Induced stratification on $\TS$; see \cite{Gr1}]\label{indbun}The stratification of  $\gg$ induces a  natural  decomposition
of the tangent space of any smooth submanifold of $\GG$. Let us
analyze the case of a hypersurface $S\subset\GG$. To this aim, at
each point $x\in S$, we intersect  $\TT_x S\subset\TT_x\GG$ with
$\TT_x^i\GG=\oplus_{j=1}^i(\HH_j)_x$. Setting
$\TT^iS:=\TS\cap\TT^i\GG,$ $n'_i:=\dim\TT^iS$,
$\HH_iS:=\TT^iS\setminus \TT^{i-1}S$ and $\HS=\HH_1S$, yields
$\TS:=\oplus_{i=1}^k\HH_iS$ and $\sum_{i=1}^kn'_i=n-1.$
Henceforth, we shall set $\VS:=\oplus_{i=2}^k\HH_iS$. It turns out
that the Hausdorff dimension of any smooth hypersurface $S$ is
$Q-1=\sum_{i=1}^k i\,n'_i$; see \cite{Gr1}, \cite{P4},
\cite{FSSC5}, \cite{Mag, Mag3}. Furthermore, if the horizontal
tangent bundle $\HS$ is generic and horizontal, then the couple
$(S, \HS)$ turns out to be a $k$-step CC-space; see Section
\ref{prelcar}.\end{oss}

\begin{es}Let  $S\subset\mathbb{H}^n$ be a smooth hypersurface. If $n=1$,
the horizontal tangent bundle $\HS$  is  $1$-dimensional at each non-characteristic point. On the constrary, if $n>1$, then $\HS$ turns out to be generic and horizontal along any non-characteristic domain
$\UU\subseteq S$.
\end{es}

\begin{Defi}\label{iuoi}Let $N\subset\GG$ be a  $(n-2)$-dimensional submanifold of class $\cont^1$.
At each point $x\in N$, the horizontal tangent space  is given by
$\HH_x N:=\HH_x\cap\TT_x N$. We say that $N$  is  {\rm non-characteristic}  at $x\in N$  if there exist two linearly
independent vectors $\nn^1,\,\nn^2\in\HH_x$ which are transversal
to $N$ at $x$. Without loss of generality, we can always assume that
$\nn^1,\,\nn^2$ are orthonormal and such that $|\nn^1 \wedge\nn^2
|=1$. If this condition holds for every $x\in N$, we say
that $N$ is  {\rm non-characteristic}. In
this case, we define in the obvious way the associated vector
bundles $\HH N(\subset \TT N)$ and $\nn N$, called, respectively,
{\rm horizontal tangent bundle} and {\rm horizontal normal
bundle}. Note that $\HH_x:=\HH_x N\oplus{\nn}_xN$, where ${\nn}_xN\cong {\rm
span}_\R\{\nn^1(x)\wedge\nn^2(x)\}.$
\end{Defi}

\begin{Defi}[see \cite{Mag}]\label{carsetgen} Let $N\subset \GG$ be a  
$(n-2)$-dimensional submanifold  of class $\cont^1$. The {\rm characteristic set}
$C_N$ is the set of all characteristic points of $N$.
Equivalently, one has $C_N:=\{x\in N : \dim\,\HH_x -\dim (\HH_x
\cap \TT_x N)\leq 1\}.$
\end{Defi}
Let $N\subset\GG$ be a  submanifold  of class $\cont^1$; then the $(Q-2)$-dimensional Hausdorff measure (associated
with a homogeneous metric $\varrho$ on $\GG$) of $C_N$ is $0$,
that is,  $\mathcal{H}_{\varrho}^{Q-2}(C_N)=0$; see \cite{Mag}.

\begin{Defi}[$\nis$-measure]\label{dens}
Let $N\subset\GG$ be a  $(n-2)$-dimensional
non-characteristic submanifold  of class $\mathbf{C}^1$; let $\nn^1, \nn^2\in\XX^0(\nn N):=\cont(N,\nn N)$ be as in
Definition \ref{iuoi} and set $\nn:=\nn^1\wedge \nn^2$. Equivalently, we are assuming that $\nn$ is a horizontal unit  normal
$2$-vector field along $N$. Then, we define a $(Q-2)$-homogeneous measure
$\nis$ on $N$ by setting
$\nis:=(\nn \LL \Vol)|_N$. 
\end{Defi}
The measure $\nis$ is  the
contraction\footnote{For the most general
 definition of  $\LL$, see \cite{FE}, Ch.1.} of the top-dimensional volume form $\Vol$ by
the\footnote{It is unique, up to the sign.} horizontal unit  normal
$2$-vector $\nn=\nn^1\wedge \nn^2$ which spans $\nn N$. Hence, $\nis$
 can be represented in terms of the
$(n-2)$-dimensional Riemannian measure $\sigma\rr^{n-2}$. In fact, let $\nu N$ denote the  normal bundle of $N$
and let $\nu_1\wedge \nu_2\in \XX^0(\nu N)$ be a unit normal $2$-vector
field orienting $N$. By standard Linear Algebra, we get
that
$\nn=\frac{\PH\nu_1\wedge \PH\nu_2}{|\PH\nu_1\wedge
\PH\nu_2|}$.  Moreover, it turns out that\[\nis = |\PH
\nu_1\wedge\PH \nu_2 |\,\sigma^{n-2}\rr. \]Note  that if $C_N\neq
\emptyset$, then $C_N=\{x\in N : |\PH \nu_1\wedge\PH \nu_2 |=0\}$
and  $\nis$ can be extended up to $C_N$ just by setting $\nis\res
C_{N}= 0$. By construction,  $\nis$ is $(Q-2)$-homogeneous with
respect to Carnot dilations $\{\delta_t\}_{t>0}$, that is, 
 $\delta_t^{\ast}\nis=t^{Q-2}\nis$. Furthermore,  $\nis$ is
equivalent, up to a bounded density-function called metric-factor,
to the $(Q-2)$-dimensional Hausdorff
 measure associated with a
homogeneous distance $\varrho$ on $\GG$; see \cite{Mag3}. For the sake of completeness,  we recall some
results obtained by Magnani and Vittone \cite{Mag3} and Magnani
\cite{Mag8}.  
\begin{teo}[Blow-up for $(n-2)$-dimensional submanifolds; see
\cite{Mag3}]\label{Blowupfor}
 Let $N\subset\GG$ be a $(n-2)$-dimensional
submanifold of class $\cont^{1, 1}$, let $x\in N$ be a
non-characteristic point and let $\delta^x_t:\GG\longrightarrow\GG$ be the Carnot homothety centered at $x$. Then
\begin{equation}\label{bkaza} 
\delta^x_{\frac{1}{r}}N\cap B_\varrho(0,1)\longrightarrow
\mathcal{I}^2(\nn(x))\cap B_\varrho(0,1)
\end{equation}as long as $r\rightarrow
0^+$, where  $\mathcal{I}^2(\nn(x))$ denotes the
$(n-2)$-dimensional subgroup of $\GG$ given by
$$\mathcal{I}^2(\nn(x)):=\left\{y\in\GG : y=\exp (Y)\,\, \mbox{for any} \,\,Y\in\gg \,\,\mbox{such that} \,\,Y\wedge
\nn(x)=0\right\}$$ and  $\nn=\nn^1\wedge\nn^2$ is a unit horizontal
normal $2$-vector along $N$.  If $\nu=\nu_1\wedge\nu_2$
is a unit normal $2$-vector field orienting $N$, then
\[\lim_{r\rightarrow 0^+}\frac{\sigma\rr^{n-2}(N\cap B_\varrho(x,
r))}{r^{Q-2}}=\frac{\kappa(\nn(x))}{|\PH \nu
(x)|},\]where $\kappa(\nn(x)):=\nis\left(
\mathcal{I}^2(\nn(x))\cap B_\varrho(0,1)\right)$ is a  positive and bounded
density-function, called {\rm metric factor} and $\PH$ is the orthogonal projection operator extended to horizontal $2$-vectors.
\end{teo}

It is worth observing that the convergence in \eqref{bkaza} is understood with respect to the
Hausdorff distance of sets.  We finally recall a recent
result about the 
size of  \it horizontal tangencies \rm (that is,   characteristic sets, in our therminology) to non-involutive distributions; see  \cite{Bal3}. The following theorem can be regarded as a  generalized version of Derridj's Theorem; see  Theorem 4.5 in \cite{Bal3}. 
\begin{teo} \label{baloghteo}Let $\GG$ be a $k$-step Carnot group.

\begin{itemize}\item[{\rm (i)}]If $S\subset \GG$ is a hypersurface
of class $\cont^2$, then the Euclidean-Hausdorff dimension of the
characteristic set $C_S$ of $S$ satisfies the inequality
$\dim_{\rm Eu-Hau}(C_S)\leq n-2.$\item[{\rm (ii)}]Let
$\VV=\HH^\perp\subset\TG$ be such that $\dim \VV\geq 2$. If
$N\subset \GG$ is a $(n-2)$-dimensional submanifold of class
$\cont^2$, then the Euclidean-Hausdorff dimension of the
characteristic set $C_N$ of $N$ satisfies the inequality
$\dim_{\rm Eu-Hau}(C_N)\leq n-3.$\end{itemize}
\end{teo}

It is worth observing, however, that more precise results can be obtained only with a further analysis of the algebraic structure of the Lie algebra of the given group.

\begin{oss}[What happens if $\dim \VV=1?$] \label{11baloghteo}First note that  if $\dim \VV=1$, then $\GG$ is a Carnot group of step $2$. In addition, the codimension of $\HH$ is $1$ so that $n=\dim\,\gg=h+1$. The most important example of such a Carnot group is, of course, the Heisenberg group $\mathbb H^r\,(r\geq 1)$. In this case, by using the results in \cite{Bal3},
we infer that: \begin{itemize}
 \item if
$N\subset \mathbb H^r$ is a $(n-2)$-dimensional submanifold of class
$\cont^2$, then the Euclidean-Hausdorff dimension of the
characteristic set $C_N$ of $N$ satisfies the inequality
$\dim_{\rm Eu-Hau}(C_N)\leq r$, where $n=2r+1$.
\end{itemize}Hence $\dim_{\rm Eu-Hau}(C_N)\leq n-3$ if, and only if, $r>1$. Furthermore, it is not difficult to show that the same assertion holds for the direct product $\mathbb H^r\times \R^m$ of the Heisenberg group $\mathbb H^r$ with a Euclidean group $\R^m$.  
\end{oss}

\begin{es}[A \textquotedblleft bad\textquotedblright example with $\dim \VV=1$] We here construct an elementary example of a $4$-dimensional $2$-step Carnot group $\GG$ in which there may exist  smooth $2$-dimensional horizontal submanifolds. In order to do this, let $\underline{X}=\{X_1, X_2, X_3, X_4\}$ be a basis of $\gg$ and assume that $[X_1, X_2]=[X_2, X_3]=X_4$ are the only non-trivial commuting relations. Obviously, we have $\HH=\mathrm{span}_\R\{X_1, X_2, X_3\}$, $\HH_2=\mathrm{span}_\R\{X_4\}$. In fact, since $[X_1, X_3]=0$, by applying
Frobenious' Theorem  it follows that the $2$-plane $\{(x_1,...,x_4)\in\GG: x_2=x_4=0\}$ is an integrable horizontal plane. 
This example can be realized, for instance, by the following  vector fields: $X_1=\partial_{x_1}-\frac{x_2}{2}\partial_{x_4}$, $X_2=\partial_{x_2}+\frac{(x_1+x_3)}{2}\partial_{x_4}$, $X_3=\partial_{x_3}-\frac{x_2}{2}\partial_{x_4}$, $X_4=\partial_{x_4}$. \end{es}

\section{Preliminary tools}\label{Preliminaries}

\subsection{Coarea Formula for the
$\HS$-gradient}\label{COAR}

\begin{teo}\label{TCOAR}Let $S\subset\GG$ be a compact hypersurface  of class $\mathbf{C}^2$  and
let $\varphi\in\cont^1(S)$. Then
\begin{equation}\label{1coar}\int_{S}\psi(x)|\qq\varphi(x)|\,\per(x)=\int_{\R}ds \int_{\varphi^{-1}[s]\cap S}\psi(y)\,\nis(y)
\end{equation}
for every $\psi\in L^1(S,\per)$.\end{teo}

\begin{proof}
This result can be  deduced by the Riemannian Coarea
Formula. Indeed, we have
\begin{equation*}\int_{S}\phi(x)|
\grad\ts\varphi(x)|\,\sigma^{n-1}\rr(x)=\int_{\R}ds
\int_{\varphi^{-1}[s]\cap
S}\phi(y)\,\sigma^{n-2}\rr(y)\end{equation*} for every $\phi\in
L^1(S,\sigma^{n-1}\rr)$; see
 \cite{BuZa}, \cite{FE}.  Choosing
 $\phi=\psi\frac{|\grad\ss\varphi|}{|\grad\ts\varphi|}|\P\cc\nu|,$ for some $\psi\in
 L^1(S,\per)$, yields
\begin{eqnarray*}\int_{S}\phi|
\grad\ts\varphi|\,\sigma^{n-1}\rr=\int_{S}\psi\frac{|\grad\ss\varphi|}{|\grad\ts\varphi|}
|\grad\ts\varphi|\underbrace{|\P\cc\nu|\,\sigma^{n-1}\rr}_{=\per}=\int_{S}\psi|\qq\varphi|\,\per.
\end{eqnarray*}Since
$\eta=\frac{\grad\ts\varphi}{|\grad\ts\varphi|}$ along
$\varphi^{-1}[s]$, it follows that
$|\P\ss\eta|=\frac{|\grad\ss\varphi|}{|\grad\ts\varphi|}$. Therefore
\begin{eqnarray*}\int_{\R}ds
\int_{\varphi^{-1}[s]\cap S}\phi(y)\,\sigma^{n-2}\rr&=&\int_{\R}ds
\int_{\varphi^{-1}[s]\cap
S}\psi\frac{|\grad\ss\varphi|}{|\grad\ts\varphi|}|\P\cc\nu|\,\sigma^{n-2}\rr\\&=&\int_{\R}ds\int_{\varphi^{-1}[s]\cap
S}\psi\underbrace{|\P\ss\eta||\P\cc\nu|\,\sigma^{n-2}\rr}_{=\nis}\\&=&\int_{\R}ds\int_{\varphi^{-1}[s]\cap
S}\psi\,\nis.\end{eqnarray*}
\end{proof}

\subsection{First variation of $\per$}\label{prvar0}

 Below we shall discuss a general integral formula, the so-called first variation formula for the $\HH$-perimeter, which is the key-tool of this paper.

\begin{oss}[The measure $\nis$ along $\partial S$]\label{measonfr}
Let $S\subset\GG$ be a hypersurface of class $\cont^2$   with
(piecewise) $\cont^1$ boundary  $\partial{S}$. Let 
 $\eta\in\XX^1(\TS)$ be the outward-pointing unit normal vector along $\partial S$ and denote by $\sigma^{n-2}\rr$
the  Riemannian measure on $\partial{S}$, given by  
$\sigma^{n-2}\rr=(\eta\LL\sigma^{n-1}\rr)|_{\partial{S}}$.
We  recall that $(X\LL\per)|_{\partial{S}}=\langle X, \eta\rangle
  |\PH\nu|\, \sigma^{n-2}\rr $ for every  $X\in\XX^1(\TS)$. The
characteristic set $C_{\partial{S}}$ of ${\partial{S}}$ turns out to be given
by $C_{\partial{S}}=\{p\in{\partial{S}}: |\PH\nu||\P\ss\eta|=0\}$.
Furthermore, by applying Definition \ref{dens}, one has
 $${\nis} =
 \left(\frac{\P\ss\eta}{|\P\ss\eta|}\LL\per\right)\bigg|_{\partial{S}},$$
or, equivalently ${\nis} =
|\PH\nu||\P\ss\eta|\,\sigma^{n-2}\rr $. The {\rm unit  horizontal normal} along $\partial{S}$ is given
by
$\eta\ss:=\frac{\P\ss\eta}{|\P\ss\eta|}$. Note that $(X\LL\per)|_{\partial{S}}=\langle X, \eta\ss\rangle\,
  {\nis} $ for every  $X\in\XX^1(\HS)$.
 \end{oss}
\begin{Defi} \label{leibniz}Let $S\subset\GG$ be a hypersurface of class $\cont^2$. Let $\imath:S\rightarrow\GG$ be the inclusion of $S$ in $\GG$
and let
 $\vartheta: ]-\epsilon,\epsilon[\times S
\rightarrow \GG$ be a ${\cont}^2$-smooth map. We say that $\vartheta$ is a  {\rm
 variation} of $\imath$ if, and only if:
\begin{itemize}
\item[{\rm(i)}] every
$\vartheta_t:=\vartheta(t,\cdot):S\rightarrow\GG$ is an
immersion;\item[{\rm(ii)}] $\vartheta_0=\imath$.
\end{itemize}
The {\rm variation vector} of $\vartheta$  is defined by
$X:=\frac{\partial \vartheta}{\partial
t}\big|_{t=0}$ and we also set $\WW=\frac{\partial \vartheta}{\partial
t}$.
\end{Defi}

\begin{no}Let $S\subset\GG$ be a  hypersurface of class $\cont^2$. Let $X\in\XG$ and let $\nu$ be the  outward-pointing unit normal vector along $S$. Hereafter, we shall denote by $X\op$ and $X\ot$ the standard decomposition of $X$ into its normal and tangential components, that is,  $X\op=\langle X,\nu\rangle\nu$ and $X\ot=X-X\op$. 
\end{no}

By definition, the   1st  variation formula of $\per$ along $S$ is given by
\begin{equation}\label{nome}I_S(\per):=\frac{d}{dt}\left(\int_{S}\vartheta_t^\ast\pert\right) \Bigg|_{t=0},\end{equation}
where $\vartheta_t^\ast$ denotes the pull-back by $\vartheta_t$ and $\pert$ denotes the $\HH$-perimeter along $S_t:=\vartheta_t(S)$.

A natural question arises: is it possible to bring the -time- derivatives inside the integral sign?
Clearly,  if we assume that $\overline{S}$ is non-characteristic, then the answer is affirmative.
In the general case, we can argue as follows.  We first note that
$$ \int_{S}\vartheta_t^\ast\pert=\int_{S}|\P\ct\nu^t|\,\mathcal{J}ac\,\vartheta_t\,\sigma\rr^{n-1},$$where $\mathcal{J}ac\,\vartheta_t $ denotes the usual Jacobian of the map $\vartheta_t$; see \cite{Simon}, Ch. 2, $\S$ 8, pp. 46-48. Indeed, by  definition, we have   $\pert=|\P\ct\nu^t|(\sigma\rr^{n-1})_t$ and hence the previous formula follows from the well-known  Area formula of Federer; see \cite{FE} or \cite{Simon}. Let us set $ f:]-\epsilon, \epsilon[\times S\longrightarrow\R$, \begin{equation}\label{faz} 
f(t, x):=|\P\ct\nu^t(x)|\,\mathcal{J}ac\,\vartheta_t(x).
\end{equation}In this case, we also set $C_{S}:=\left\lbrace x\in S: |\P\ct\nu^t(x)|=0 \right\rbrace$.  With this notation,  our original question can be solved by applying  to  $f$ the  Theorem of Differentiation under the integral; see, for instance, \cite{Jost}, Corollary 1.2.2, p.124. More precisely, let us compute
\begin{eqnarray}\label{ujh}\frac{d f}{dt}&=&\frac{d\,|\P\ct\nu^t|}{dt}\,\mathcal{J}ac\,\vartheta_t + |\P\ct\nu^t|\frac{d\,\mathcal{J}ac\,\vartheta_t}{dt}\\\nonumber &=&\left\langle\WW,\grad\,|\P\ct\nu^t|\right\rangle\,\mathcal{J}ac\,\vartheta_t + |\P\ct\nu^t|\frac{d\,\mathcal{J}ac\,\vartheta_t }{dt}\\\nonumber &=&\left( \left\langle\WW\op,\grad\,|\P\ct\nu^t|\right\rangle+\left\langle\WW\ot,\grad\,|\P\ct\nu^t|\right\rangle + |\P\ct\nu^t|\div\tst\WW\right) \mathcal{J}ac\,\vartheta_t\\\nonumber &=&\left( \left\langle\WW\op,\grad\,|\P\ct\nu^t|\right\rangle+   \div\tst\left(\WW|\P\ct\nu^t|\right)\right) \mathcal{J}ac\,\vartheta_t,
\end{eqnarray}where we have used the very definition of tangential divergence and the well-known calculation of $\frac{d\,\mathcal{J}ac\,\vartheta_t}{dt}$, which can be found in Chavel's book \cite{Ch2}; see Ch.2, p.34. Now since  $|\P\ct\nu^t|$ is a Lipschitz continuous function, it follows that $\frac{d f}{dt}$ is bounded on $S\setminus C_{S}$ and so lies to $L^1_{loc}(S; \sigma\rr^{n-1})$. This shows that: \it we can pass the time-derivative  through the integral sign. \rm

At this point the 1st variation formula  follows from the calculation of the Lie derivative of $\per$ with respect to the initial velocity $X$ of the flow $\vartheta_t$.

\begin{oss}Let $M$ be  a smooth manifold, let $\omega\in\Om^k(M)$ be a differential  $k$-form on $M$ and let $X\in\XX(\TT M)$ be a differentiable vector field on $M$, with associated flow $\phi_t:M\longrightarrow M$. We recall that the Lie derivative of $\omega$  with respect to $X$, is defined by $\Lie_X\omega:=\frac{d}{dt}\phi_t^\ast\omega\big|_{t=0},$ where $\phi_t^\ast\omega$ denotes the pull-back of $\omega$ by $\phi_t$. In other words, the Lie derivative of $\omega$ under the flow generated by $X$ can be seen as the \textquotedblleft infinitesimal 1st variation\textquotedblright  of  $\omega$  with respect to $X$. Then, Cartan's identity says that \[\Lie_X\omega= (X\LL d\omega) +d(X\LL\omega).\]This formula is a very useful tool in proving variational formulas. For the case of  Riemannian volume forms, we refer the reader to Spivak's book \cite{Spiv}; see Ch. 9, pp. 411-426 and 513-535. 
\end{oss}

The Lie derivative of the differential $(n-1)$-form $\per$ with respect to $X$ can be calculated elementarily as follows.
 We have$$ X\LL d \per=X\LL d(\nn\LL\Vol)= X\LL\left(\div\,\nn \Vol\right)=\langle X,\nu\rangle\,\div\,\nn\,\sigma\rr^{n-1}.$$Note that $\div\,\nn =\div\cc\nn=-\MS$. In fact, one has $$\div\,\nn =\sum_{i=1}^n\langle\nabla_{X_i}\nn,
{X_i}\rangle=\sum_{i=1}^\DH X_i({\nn}_i)=\div\cc\nn=-\MS.$$

Now the second term in Cartan's identity can be computed using the following:\begin{lemma}\label{fondam}If $X\in\XX^1(\TG)$, then   $(X\LL\per)|_S=\left( \left(X\ot|\PH\nu|-\langle X,\nu\rangle\nn\ot\right)\LL\sigma\rr^{n-1}\right)\big|_S$ and at each non-characteristic point of $S$, we have $$d(X\LL\per)|_S=\div\ts \left( X\ot|\PH\nu|-\langle X,\nu\rangle\nn\ot \right) \,\sigma\rr^{n-1}.$$
\end{lemma}
\begin{proof}We have
\begin{eqnarray*}d(X\LL\per)|_S &=& (X\LL\nn\LL\Vol)|_S\\&=&d\left( \left( X\ot + X\op\right)\LL \left( \nn\ot + \nn\op\right)\LL\Vol\right)\big|_S\\&=&d\left( X\ot \LL  \nn\op \LL\Vol\right)\big|_S+
d\left( \nn\ot\LL X\op \LL\Vol\right)\big|_S\\&=&d \left(X\ot\LL\per\right)\big|_S+d \left(\nn\ot\LL\langle X,\nu\rangle\sigma\rr^{n-1}\right)\big|_S\\&=&\div\ts \left( X\ot|\PH\nu|-\langle X,\nu\rangle\nn\ot \right) \,\sigma\rr^{n-1}.\end{eqnarray*}
\end{proof}
 \begin{oss}\label{mistake} The previous calculation corrects a mistake in \cite{Monteb}, where the normal component of the vector field $X$
 was omitted and this caused the loss of some divergence-type terms in some of the variational formulas proved there. 
\end{oss}

 Thus, we can  conclude that
\begin{equation}\label{9}
 \Lie_X\per=\left(-\MS\langle X,\nu\rangle +\div\ts \left( X\ot|\PH\nu|-\langle X,\nu\rangle\nn\ot \right)\right)\,\sigma\rr^{n-1},
\end{equation}at each non-characteristic point of $S$.
Furthermore, if $\MS\in L^1(S; \sigma\rr^{n-1})$,  we can integrate this formula over all of $S$.
Indeed, in this case, all terms in the formula above turn out to be in $L^1(S; \sigma\rr^{n-1})$; see also \cite{MonteStab}.

\begin{teo}[1st variation of $\per$] \label{1vg}Let $S\subset\GG$ be a compact
hypersurface  of class ${\cont}^2$  with -or without- boundary $\partial S$ and let
 $\vartheta: ]-\epsilon,\epsilon[\times S
\rightarrow \GG$ be a ${\cont}^2$-smooth variation of $S$. Let $X=\frac{d\,\vartheta_t}{dt}\big|_{t=0}$ be the variation vector field and denote by $X\op$ and $X\ot$  the normal and  tangential components of $X$ along $S$, respectively. If $\MS\in L^1(S; \sigma\rr^{n-1})$, then \begin{eqnarray}\label{fva}I_S(X,\per)=
\int_{S}\left(-\MS\langle X\op,\nu\rangle +\div\ts\left( X\ot|\PH\nu|-\langle X\op,\nu\rangle\nn\ot \right)\right)\,\sigma\rr^{n-1}\\\label{fva2formula}=\int_{S}-\MS\frac{\langle X\op,\nu\rangle}{|\PH\nu|}\,\per+ \int_{\partial S}\left\langle \left( X\ot-\frac{\langle X\op,\nu\rangle}{|\PH\nu|}\nn\ot \right),\frac{\eta }{|\P\ss\eta|}\right\rangle \,\underbrace{|\PH\nu||\P\ss\eta|\,\sigma\rr^{n-2}}_{=\nis}.
\end{eqnarray}
\end{teo}

Note that the second equality follows by applying the following generalized Stokes' formula to the differential $(n-2)$-form $\alpha:=(X\LL \per)|_S\in\Om^{n-2}(S)$.
\begin{Prop}\label{ST} Let $M$ be an oriented $k$-dimensional manifold of class $\cont^2$ with boundary
$\partial M$. Then $\int_M d\alpha=\int_{\partial M}\alpha$ for every compactly supported $(k-1)$-form $\alpha$ such that  $\alpha\in L^\infty(M)$, $d\alpha\in L^1(M)$ -or  $d\alpha\in L^\infty(M)$- and $\imath_M^\ast\alpha\in L^\infty(\partial M)$, where $\imath_M:\partial M\longrightarrow\overline{M}$ is the natural inclusion. 
\end{Prop}

\begin{oss} 
The previous result can be deduced by applying a standard procedure\footnote{See, for instance, Federer's book \cite{FE}, paragraph 3.2.46, p. 280; see also \cite{Pfeffer}, Remark 5.3.2, p. 197.} from a divergence-type theorem proved by  Anzellotti; see, more precisely, Theorem 1.9 in \cite{Anze}. More recent and  more general results can  be found in the paper by Chen, Torres and Ziemer  \cite{variZ}. See also \cite{Taylor}, formula (G.38), Appendix G. 
\end{oss}

\subsection{Blow-up of the horizontal perimeter
$\per$ up to $C_S$}\label{blow-up}Let $S\subset\GG$ be a smooth
hypersurface. In this section we shall study the following density-limit:
\begin{equation}\label{limit} \lim_{r\rightarrow
0^+}\frac{\per (S\cap
B_{\varrho}(x,r))}{r^{Q-1}},\end{equation}where $B_{\varrho}(x,r)$
is the $\varrho$-ball of center $x\in{\rm Int}\,S$ and radius $r$.

 It is worth observing  that the point
$x$ is {\it  not necessarily non-characteristic}.  For
a similar analysis, we refer the reader to \cite{Mag, Mag3,
Mag4} and to \cite{Mag8}, for  the characteristic
case in the setting of $2$-step Carnot groups; see also
\cite{balogh, BTW}, \cite{FSSC3, FSSC5}.

We  stress that the second part of the following Theorem \ref{BUP} (which, to the best of our knowledge, is
new) will be used  only in Section \ref{perdindirindina} and Section \ref{asintper} in order to prove some monotonicity estimates for the $\HH$-perimeter of the intersection $S_t$ of a smooth hypersurface $S$ with a homogeneous $\varrho$-ball $B_\varrho(x, t)$ centered at an interior characteristic point $x\in S\cap C_S$.

\begin{teo}\label{BUP}Let $\GG$ be  a $k$-step Carnot group.

\begin{itemize}

\item [{\rm Case (i)}]\,\,Let $S$ be a 
hypersurface of class $\cont^1$ and let $x\in {\rm Int}(S\setminus C_S)$;
 then \begin{eqnarray}\label{BUP1}\per(S\cap B_\varrho(x,r))\sim
\kappa_\varrho(\nn(x))\, r^{Q-1}\qquad\mbox{for}\quad r\rightarrow
0^+,\end{eqnarray}where the density-function $\kappa_\varrho(\nn(x))$ is
called {\rm metric factor}. It turns out that
$$\kappa_\varrho(\nn(x))=\per\left(\mathcal{I}(\nn(x))\cap
B_\varrho(x,1)\right),$$where $\mathcal{I}(\nn(x))$
 denotes the vertical hyperplane\footnote{\label{piedipag}Note that $\mathcal{I}(\nn(x))$
corresponds to an ideal of the Lie algebra $\gg$. We also remark
that the $\HH$-perimeter on a vertical hyperplane equals the
Euclidean-Hausdorff measure $\Ar$ on the hyperplane.} through $x$ and orthogonal
to $\nn(x)$.

\item [{\rm Case (ii)}]\,\, Let $x\in {\rm Int}(S\cap C_S)$ and
let $\alpha\in I\vv$ be such that, locally around $x$, $S$ can
be represented as an
$X_\alpha$-graph of class $\cont^i$, where $i={\rm ord}(\alpha)\in\{2,...,k\}$. In this case, we have
$$S\cap
B_\varrho(x,r)\subset\exp\left\{\left(\zeta_1,...,\zeta_{\alpha-1},
\psi(\zeta),\zeta_{\alpha+1},...,\zeta_n \right)\, : \,
\zeta:=(\zeta_1,...,\zeta_{\alpha-1},
0,\zeta_{\alpha+1},...,\zeta_n )\in \ee_\alpha^\perp\right\},$$
for some function $\psi:\ee_\alpha^{\perp}\cong\R^{n-1}\rightarrow\R$  of class $\cont^i$. Without loss of generality,
we may assume that $x=0\in\GG$ and  $\psi(0)=0$.  If
\begin{equation}\label{0dercond}\frac{\partial^{\scriptsize(l)}
\psi}{\partial\zeta_{j_1}...\partial\zeta_{j_l}}(0)=0\qquad\mbox{whenever}\quad{\rm
ord}(j_1)+...+{\rm ord}(j_l)< i,
\end{equation}then
\begin{eqnarray}\label{BUPcarcase}\per(S\cap B_\varrho(x,r))\sim
\kappa_\varrho(C_S(x))\, r^{Q-1}\qquad\mbox{as long as}\qquad r\rightarrow
0^+,\end{eqnarray}where the function $\kappa_\varrho(C_S(x))$ can
 be computed by integrating the measure $\per$ along a  polynomial
hypersurface, which is the graph of the Taylor's expansion up to
 $i={\rm ord}(\alpha)$ of $\psi$ at
$\zeta=0\in\ee_\alpha^\perp$. More precisely, one has$$\kappa_\varrho(C_S(x))=\per(S_\infty\cap
B_\varrho(x,1)),$$where
$S_\infty=
\{(\zeta_1,...,\zeta_{\alpha-1},\widetilde{\psi}
(\zeta),\zeta_{\alpha+1},...,\zeta_n)\, :\,
\zeta\in\ee_\alpha^\perp\}$ and\begin{eqnarray}\nonumber\widetilde{\psi}(\zeta)=\sum_{\stackrel{{j_1}}{\scriptsize{\rm
ord}(j_1)=i}}\frac{\partial\psi}{\partial\zeta_{j_1}}(0)\,\zeta_{j_1}+\ldots+
\sum_{\stackrel{{j_1,...,j_l}}{\scriptsize{{\rm ord}(j_1)+...+{\rm
ord}(j_l)=i}}}\frac{\partial^{\scriptsize(l)}\psi}{\partial\zeta_{j_1}...
\partial\zeta_{j_l}}(0)\,\zeta_{j_1}...\zeta_{j_l}.\end{eqnarray}
Finally, if \eqref{0dercond} does not hold, then $S_\infty=\emptyset$ and $\kappa_\varrho(C_S(x))=0$.
\end{itemize}
\end{teo}

\begin{oss}[Order of $x\in S$]\label{fraccicarla}The rescaled hypersurfaces $\delta_{\frac{1}{r}}S$
locally converge to a limit-set $S_\infty$, that is, 
$\delta_{\frac{1}{r}}S\longrightarrow S_\infty$ for $r\rightarrow0^+$, where the
convergence is understood with respect the Hausdorff convergence
of sets; see \cite{Mag3, Mag8}. At every $x\in {\rm Int}(S\setminus
C_S)$ the limit-set $S_\infty$ is the vertical
hyperplane  $\mathcal{I}(\nn(x))$. Otherwise, $S_\infty$ is a
polynomial hypersurface.
Assume that $S$ is smooth enough near its characteristic set $C_S$,
say of class $\cont^k$. Then, there  exists a minimum integer $i={\rm
ord}(\alpha)$ such that \eqref{0dercond} holds. The number
${\rm ord}(x)=Q-i$ is called the \rm order \it of the
characteristic point $x\in C_S$. \end{oss}

\begin{proof}[Proof of Theorem \ref{BUP}]
We preliminarily note that the limit \eqref{limit} can be
computed, without loss of generality, at $0\in\GG$,
just by left-translating $S$. We have
$${\per \left(S\cap B_{\varrho}(x,r)\right)}={\per \left(x^{-1}\bullet\left(S\cap
B_{\varrho}(x,r)\right)\right)}={\per  \left(\left(x^{-1}\bullet
S\right)\cap B_{\varrho}(0,r)\right)}$$for any $x\in{\rm Int}\,S$,
where the second equality  follows from the additivity of the
group law $\bullet$.
\begin{no}We shall set:\begin{itemize}
\item[\rm (i)] $S_r(x):=S\cap B_{\varrho}(x,r)$;\item[\rm (ii)]
 $\widetilde{S}:=x^{-1}\bullet S$;
\item[\rm (iii)] $\widetilde{S}_r:=x^{-1}\bullet
S_r(x)=\widetilde{S}\cap B_{\varrho}(0,r).$\end{itemize}\end{no}By
using the homogeneity of $\varrho$ and the invariance of $\per$
under positive Carnot dilations\footnote{This means that
$\delta_t^\ast\per=t^{Q-1}\per$, $t\in \R_+$; see Section
\ref{prelcar}.}, it follows that
$$\per(\widetilde{S}_r)=r^{Q-1}\per\left(\delta_\frac{1}{r}\widetilde{S}\cap B_\varrho(0,1)\right)$$for all $r\geq 0$.
Therefore $\frac{\per(
\widetilde{S}_r)}{r^{Q-1}}=\per\left(\delta_\frac{1}{r}\widetilde{S}\cap
B_\varrho(0,1)\right)$, and it remains to compute
\begin{equation}\label{dens1}
\lim_{r\rightarrow 0^+}\per\left(\delta_{\frac{1}{r}}\widetilde{S}\cap
B_\varrho(0, 1)\right).\end{equation}

We begin by studying the non-characteristic case; see also
\cite{Mag8, Mag4}.\\
\\\noindent{\rm Case (i).\,\,\rm Blow-up for non-characteristic points.}\,\,{\it Let $S\subset\GG$
be a hypersurface of class $\cont^1$ and let $x\in {\rm Int}\,S$ be
non-characteristic}. Locally around  $x$, the hypersurface $S$ is
oriented by unit $\HH$-normal $\nn(x)$,
that is,  $\nn(x)$ is transversal\footnote{We say that $X$ is transversal to $S$ at $x$, in symbols  $X\pitchfork \TT_xS$, if
$\langle X, \nu\rangle\neq 0$  at  $x$, where $\nu$ is a
unit normal vector along $S$.} to $S$ at $x$. Thus, at least locally
around $x$, we may think of $S$ as a
$\cont^1$-graph  with respect to the horizontal direction
$\nn(x)$. Moreover, we can find an orthonormal change of
coordinates on
$\Rn\cong\TT_0\GG$ such that
$$\ee_1=X_1(0)=(L_{x^{-1}})_\ast\nn(x).$$With no loss of generality, by the Implicit Function Theorem we can write
$\widetilde{S}_r=x^{-1}\bullet S_r(x)$, for some (small enough)
$r>0$, as the exponential image in $\GG$ of a
$\cont^1$-graph\footnote{Actually, since the argument is local,
$\psi$ can be defined just on a suitable neighborhood of $0\in
\ee_1^\perp\cong \R^{n-1}$.}. So let
$$\Psi=\{(\psi(\xi),\xi)\,:\, \xi\in\R^{n-1}\}\subset\gg,$$ where
$\psi:\ee_1^{\perp}\cong\R^{n-1}\longrightarrow\R$ is a
$\cont^1$-function satisfying:

\begin{itemize}\item[${\rm(i)}$]$\psi(0)=0$;
\item[${\rm(ii)}$]${\partial\psi}/{\partial\xi_j} (0)=0$ for every
$j=2,...,\DH\,(=\dim\HH)$,
\end{itemize}for $\xi\in\ee_1^\perp\cong\R^{n-1}$. Therefore  $\widetilde{S}_r=\exp\Psi\cap
B_\varrho(0,r),$ for all (small enough) $r>0$. This remark can be
used to compute the limit \eqref{dens1}. So let us fix a positive
$r_0$ satisfying the previous assumptions and let $0\leq r\leq
r_0$. Then
\begin{eqnarray}\label{dens2}\delta_\frac{1}{r}\widetilde{S}\cap
B_\varrho(0, 1)=\exp\left(\widehat{\delta}_\frac{1}{r}\Psi\right)\cap
B_\varrho(0,1),\end{eqnarray}where
$\{\widehat{\delta}_{t}\}_{t\geq 0}$ are the induced dilations on
$\gg$, that is,  $\delta_t=\exp\circ\widehat{\delta}_{t}$ for every
$t\geq 0$. Henceforth, we will consider the restriction of
$\widehat{\delta}_{t}$ to the hyperplane
$\ee_1^\perp\cong\R^{n-1}$. So, with a slight
abuse of notation, instead of
$(\widehat{\delta}_{t})\big|_{\ee_1^{\perp}}(\xi)$ we shall
write $\widehat{\delta}_{t}\xi$. Moreover, we shall assume
$\R^{n-1}=\R^{\DH-1}\oplus\R^{n-\DH}$. Note that the induced
dilations $\{\widehat{\delta}_{t}\}_{t\geq 0}$ make
$\ee_1^\perp\cong\R^{n-1}$  a {\it graded vector space}, whose
grading respects that of $\gg$. We have
$$\widehat{\delta}_\frac{1}{r}\Psi=\widehat{\delta}_\frac{1}{r}\left\{(\psi(\xi),\xi)\, :\, \xi\in\R^{n-1}\right\}
=\left\{\left(\frac{\psi(\xi)}{r},\widehat{\delta}_\frac{1}{r}\xi\right)\,
:\, \xi\in\R^{n-1}\right\}.$$By using the change of variables
$\zeta:=\widehat{\delta}_{{1}/{r}}\xi$, we get that
$$\widehat{\delta}_\frac{1}{r}\Psi=
\left\{\left(\frac{\psi\left(\widehat{\delta}_{r}\zeta\right)}{r},\zeta\right)\,:\,
\zeta\in\R^{n-1}\right\}.$$By hypothesis $\psi\in \cont^1(U_0)$,
where $U_0$ is a suitable open neighborhood of $0\in\R^{n-1}$.
Using a Taylor's expansion of $\psi$ at $0\in\R^{n-1}$ and the
assumptions
 (i) and (ii), yields
\begin{eqnarray*}\psi(\xi)=\psi(0) + \langle\grad_{\R^{n-1}}\psi(0), \xi \rangle_{\R^{n-1}} + {\rm o}(\|\xi\|_{\R^{n-1}})
=\langle\grad_{\R^{n-\DH}}\psi(0), \xi_{\R^{n-\DH}} \rangle_{\R^{n-\DH}}
+{\rm o}(\|\xi\|_{\R^{n-1}}),\end{eqnarray*}as long as $\xi\rightarrow
0\in \R^{n-1}$. Note that $\widehat{\delta}_{r}\zeta\longrightarrow
0\in\R^{n-1}$ for $r\rightarrow 0^+$. By the
previous change of variables, we get that
\[\psi\left(\widehat{\delta}_{r}\zeta\right)
= \left\langle\grad_{\R^{n-\DH}}\psi(0),
\widehat{\delta}_{r}\left(\zeta_{\R^{n-\DH}}\right)
\right\rangle_{\R^{n-\DH}} + {\rm o}\left(r \right)\]for
$r\rightarrow 0^+$. Since $\left\langle\grad_{\R^{n-\DH}}\psi(0),
\widehat{\delta}_{r}\left(\zeta_{\R^{n-\DH}}\right)
\right\rangle_{\R^{n-\DH}}={\rm o}(r)$ for $r\rightarrow 0^+$, we
easily get that  the limit-set (obtained by blowing-up
$\widetilde{S}$ at the non-characteristic point $0$) is given by
\begin{equation}\label{dens3}\Psi_\infty=\lim_{r\rightarrow
0^+}\widehat{\delta}_\frac{1}{r}\Psi=\exp(\ee_1^\perp)=\mathcal{I}(X_1(0)),\end{equation}
where $\mathcal{I}(X_1(0))$ denotes the vertical hyperplane
through the identity $0\in\GG$ and orthogonal to $X_1(0)$. We
have shown that \eqref{dens1} can be computed by means of
\eqref{dens2} and \eqref{dens3}. More precisely
\begin{equation*}\lim_{r\rightarrow 0^+}\per\left(\delta_\frac{1}{r}\widetilde{S}\cap
B_\varrho(0, 1)\right)=\per\left(\mathcal{I}(X_1(0))\cap
B_\varrho(0,1)\right)\end{equation*}By remembering the previous change of
variables, it follows that $S_\infty=\mathcal{I}(\nn(x))$ and that
$$\kappa_\varrho(\nn(x))=\lim_{r\rightarrow 0^+}\frac{\per (S\cap
B_{\varrho}(x, r))}{r^{Q-1}}=\per\left(\mathcal{I}(\nn(x))\cap
B_\varrho(x,1)\right),$$which was to be proven.\\

 \noindent{\rm Case (ii).\,\rm  Blow-up at the characteristic set.}
{\it We are now assuming that $S\subset\GG$ is a  
hypersurface of class $\cont^i$  for some $i\geq 2$ and that $x\in {\rm Int}(S\cap C_S)$}.
Near $x$ the hypersurface $S$ is  oriented by some vertical
vector. Hence, at least locally around $x$, we may think of $S$ as the exponential image of a $\cont^i$-graph
 with respect to
some vertical direction $X_\alpha$ transversal to $S$ at $x$. Note that $X_\alpha$ is a vertical
left-invariant vector field of the fixed left-invariant frame
$\underline{X}=\{X_1,...,X_n\}$  and $\alpha\in
I\vv=\{h+1,..., n\}$ denotes a ``vertical'' index; see Definition
\ref{1notazione}. {\it Furthermore,  we are assuming that}
$\mathrm{ord}(\alpha):= i$,
 for some $i=2,..,k.$ As in the non-characteristic case,
for the sake of simplicity,  we
left-translate $S$ in such a way that $x$
coincides with $0\in\GG$. To this end, it is sufficient to
replace $S$ by $\widetilde{S}=x^{-1}\bullet S$. At the level of
the Lie algebra $\gg$, let us consider the hyperplane
$\ee_\alpha^\perp$ through the origin $0\in\gg\cong\R^{n}$ and
orthogonal to $\ee_\alpha=X_{\alpha}(0)$. Note that
$\ee_\alpha^\perp$ is the natural ``parameter space'' of a $\ee_\alpha$-graph. By the classical Implicit Function
Theorem, we may write
$\widetilde{S}_r=x^{-1}\bullet S_r(x)$ as the exponential image in
$\GG$ of a $\cont^i$-graph. We  have
$$\Psi=\left\{\left(\xi_1,...,\xi_{\alpha-1}\underbrace{,
\psi(\xi),}_{\scriptsize{\alpha-th\,
place}}\xi_{\alpha+1},...,\xi_n \right)\, :\,
\xi:=(\xi_1,...,\xi_{\alpha-1}, 0,\xi_{\alpha+1},...,\xi_n )\in
\ee_\alpha^\perp\cong \R^{n-1}\right\}$$ where
$\psi:\ee_\alpha^{\perp}\cong\R^{n-1}\longrightarrow\R$ is a
$\cont^i$-smooth function satisfying:

\begin{itemize}\item[${\rm(j)}$]$\psi(0)=0$;
\item[${\rm(jj)}$]${\partial\psi}/{\partial\xi_j} (0)=0$ for every
$j=1,...,\DH\,(=\dim\HH)$.
\end{itemize}Thus we get that $\widetilde{S}_r=\exp\Psi\cap
B_\varrho(0,r),$ for every (small enough) $r>0$. Hence, we can use the above remarks to compute \eqref{dens1} and as in the non-characteristic case, we  use
\eqref{dens2}. We have
\begin{eqnarray*}\widehat{\delta}_\frac{1}{r}\Psi&=&
\widehat{\delta}_\frac{1}{r}\left\{\left(\xi_1,...,\xi_{\alpha-1},
\psi(\xi),\xi_{\alpha+1},...,\xi_n \right)\, :\, \xi\in
\ee_\alpha^\perp\right\}\\
&=&\left\{\left(\frac{{\xi_1}}{r},...,\frac{\xi_{\alpha-1}}{r^{{\rm
ord}(\alpha-1)}},
\frac{\psi(\xi)}{r^i},\frac{\xi_{\alpha+1}}{r^{{\rm
ord}(\alpha+1)}},...,\frac{\xi_n}{r^k} \right)\, :\, \xi\in
\ee_\alpha^\perp\right\}.\end{eqnarray*}Setting
$$\zeta:=\widehat{\delta}_\frac{1}{r}
\xi=\left(\frac{{\xi_1}}{r},...,\frac{\xi_{\alpha-1}}{r^{{\rm
ord}(\alpha-1)}}, 0,\frac{\xi_{\alpha+1}}{r^{{\rm
ord}(\alpha+1)}},...,\frac{\xi_n}{r^k} \right),$$where
$\zeta=(\zeta_1,...,\zeta_{\alpha-1},0,\zeta_{\alpha+1},...,\zeta_n)\in\ee_\alpha^\perp$,
yields
$$\widehat{\delta}_\frac{1}{r}\Psi=
\left\{\left(\zeta_1,...,\zeta_{\alpha-1},
\frac{\psi\left(\widehat{\delta}_r\zeta\right)}{r^i},\zeta_{\alpha+1},...,\zeta_n\right)\,:\,
\zeta\in\ee_\alpha^\perp\right\}.$$By hypothesis $\psi\in
\cont^i(U_0)$, where $U_0$ is an open neighborhood of
$0\in\ee_\alpha^{\perp}\cong\R^{n-1}$. Furthermore, one has
$\widehat{\delta}_{r}\zeta\longrightarrow 0$ as long as $r\rightarrow 0^+$. So
we have to study the following limit
\begin{equation}\label{lim}\widetilde{\psi}(\zeta):=\lim_{r\rightarrow
0^+}\frac{\psi\left(\widehat{\delta}_r\zeta\right)}{r^i},\end{equation}
whenever exists. The first remark is that, when this limit equals
$+\infty$, we have
$$\lim_{r\rightarrow
0^+}\frac{\per( \widetilde{S}_r)}{r^{Q-1}}=\lim_{r\rightarrow
0^+}\per\left(\exp\left(\widehat{\delta}_\frac{1}{r}\Psi\right)\cap
B_\varrho(0,1)\right)=0,$$ because $\exp\left(\widehat{\delta}_\frac{1}{r}\Psi\right)\cap
B_\varrho(0,1)\longrightarrow\emptyset$
as long as $r\rightarrow 0^+$.

At this point, making use of a Taylor's expansion of $\psi$ together with
$\rm(j)$ and $\rm(jj)$, yields
\begin{eqnarray*}\psi\left(\widehat{\delta}_{r}\zeta\right)&=
&\psi(0) +\sum_{j_1}r^{{\rm
ord}(j_1)}\frac{\partial\psi}{\partial\zeta_{j_1}}(0)\,\zeta_{j_1}
+\sum_{j_1, j_2}r^{{\rm ord}(j_1)+{\rm
ord}(j_2)}\frac{\partial^{\scriptsize(2)}\psi}{\partial\zeta_{j_1}
\partial\zeta_{j_2}}(0)\,\zeta_{j_1}\zeta_{j_2}\\&&+...+
\sum_{j_1,..., j_i}r^{{\rm ord}(j_1)+...+{\rm
ord}(j_i)}\frac{\partial^{\scriptsize(i)}\psi}{\partial\zeta_{j_1}...
\partial\zeta_{j_i}}(0)\,\zeta_{j_1}\cdot...\cdot\zeta_{j_i}+{\rm
o}\left(r^i\right)\end{eqnarray*}\begin{eqnarray*}&=& \sum_{j_1}r^{{\rm
ord}(j_1)}\frac{\partial\psi}{\partial\zeta_{j_1}}(0)\,\zeta_{j_1}
+\sum_{j_1, j_2}r^{{\rm ord}(j_1)+{\rm
ord}(j_2)}\frac{\partial^{\scriptsize(2)}\psi}{\partial\zeta_{j_1}\partial\zeta_{j_2}}(0)\,\zeta_{j_1}\zeta_{j_2}\\&&+...+
\sum_{j_1,..., j_i}r^{{\rm ord}(j_1)+...+{\rm
ord}(j_l)}\frac{\partial^{\scriptsize(l)}\psi}{\partial\zeta_{j_1}...
\partial\zeta_{j_i}}(0)\,\zeta_{j_1}\cdot...\cdot\zeta_{j_l}+{\rm
o}\left(r^i\right)\end{eqnarray*}as $r\rightarrow 0^+$. Therefore
\begin{eqnarray*}\frac{\psi\left(\widehat{\delta}_{r}\zeta\right)}{r^i}&=
& \sum_{j_1}r^{{\rm
ord}(j_1)-i}\frac{\partial\psi}{\partial\zeta_{j_1}}(0)\,\zeta_{j_1}
+\sum_{j_1, j_2}r^{{\rm ord}(j_1)+{\rm
ord}(j_2)-i}\frac{\partial^{\scriptsize(2)}\psi}{\partial\zeta_{j_1}\partial\zeta_{j_2}}(0)\,
\zeta_{j_1}\zeta_{j_2}\\&&+...+ \sum_{j_1,..., j_l}r^{{\rm
ord}(j_1)+...+{\rm
ord}(j_l)-i}\frac{\partial^{\scriptsize(l)}\psi}{\partial\zeta_{j_1}...
\partial\zeta_{j_l}}(0)\,\zeta_{j_1}\cdot...\cdot\zeta_{j_l}+{\rm
o}\left(1\right)\end{eqnarray*}as $r\rightarrow 0^+$. By applying the
hypothesis
$\frac{\partial^{\scriptsize(l)}
\psi}{\partial\zeta_{j_1}...\partial\zeta_{j_l}}(0)=0$ whenever ${\rm
ord}(j_1)+...+{\rm ord}(j_l)<i$,it follows that
\eqref{lim} exists. Setting
\begin{equation*}\Psi_\infty=\lim_{r\rightarrow
0^+}\widehat{\delta}_\frac{1}{r}\Psi=
\left\{\left(\zeta_1,...,\zeta_{\alpha-1},\widetilde{\psi}
(\zeta),\zeta_{\alpha+1},...,\zeta_n\right)\,:\,
\zeta\in\ee_\alpha^\perp\right\},\end{equation*}where
$\widetilde{\psi}$ is the polynomial function of homogeneous degree
$i={\rm ord}(\alpha)$ given by
\begin{eqnarray*}\widetilde{\psi}(\zeta)=\sum_{\stackrel{{j_1}}{\scriptsize{\rm
ord}(j_1)=i}}\frac{\partial\psi}{\partial\zeta_{j_1}}(0)\,\zeta_{j_1}+\ldots
+ \sum_{\stackrel{{j_1,...,j_l}}{\scriptsize{{\rm
ord}(j_1)+...+{\rm
ord}(j_l)=i}}}\frac{\partial^{\scriptsize(l)}\psi}{\partial\zeta_{j_1}...
\partial\zeta_{j_l}}(0)\,\zeta_{j_1}\cdot...\cdot\zeta_{j_l},\end{eqnarray*}
yields $S_\infty=x\bullet \Psi_\infty$ and the thesis  
follows.
\end{proof}

\begin{oss}\label{boundonmetricfactor}The metric factor is not constant, in general. It
turns out to be constant, for instance,  by assuming that $\varrho$
is symmetric on all layers; see \cite{Mag3}.
Anyway, it is {\rm uniformly bounded} by two positive constants $K_1$
and $K_2$. This can be seen by using the \textquotedblleft ball-box metric\textquotedblright \footnote{By definition one has ${\rm Box}(x,r)=x\bullet {\rm Box}(0,r)$ for every $x\in\GG$, where
$${\rm Box}(0,r)=\left\{y=\exp\left(\sum_{i=1}^k y\ci\right)\in\GG\,:\,
\|y\ci\|_\infty\leq r^i\right\}.$$We stress that $y\ci=\sum_{j_i\in
I\ci}y_{j_i}\ee_{j_i}$ and that $\|y\ci\|_\infty$ denotes the sup-norm on
the $i$-th layer of $\gg$; see \cite{Gr1},
\cite{Montgomery}.} and a homogeneity argument. Let $S$ be as in Theorem \ref{BUP}, Case (i). Let $B_\varrho(x,1)$ the unit $\varrho$-ball centered at $x\in {\rm Int} (S\setminus C_S)$ and let  $r_1, r_2\geq 0$ be such that
$0<r_1\leq 1\leq r_2$. In particular, ${\rm Box}(x,r_1)\subseteq
B_\varrho(x,1)\subseteq{\rm Box}(x,r_2)$. Recall that
$$k_\varrho(\nn(x))=\per(\mathcal{I}(\nn(x))\cap B_\varrho(x,1))=\Ar(\mathcal{I}(\nn(x))\cap B_\varrho(x,1)),$$
where $\mathcal{I}(\nn(x))$ denotes the vertical hyperplane
orthogonal to $\nn(x)$. By homogeneity, one has $\delta_{t}{\rm
Box}(0,1/2)={\rm Box}(0,t/2)$ for every $t\geq 0$ and by an elementary
computation\footnote{The unit box ${\rm Box}(x,1/2)$ is the
left-translated at $x$ of ${\rm Box}(0,1/2)$ and so, by
left-invariance of $\per$, the computation can be done at
$0\in\GG$.  Since ${\rm Box}(0,1/2)$ is the unit hypercube of
$\R^n\cong\gg$, it remains to estimate the
$\per$-measure of the intersection of ${\rm Box}(0,1/2)$ with a
 generic vertical hyperplane through the origin
$0\in\R^n$.   If $\mathcal{I}(X)$
is the vertical hyperplane through  $0\in\Rn$ and
orthogonal to $X\in\HH$, we get that
$$1\leq\Ar({\rm Box}(0,1/2)\cap\mathcal{I}(X))\leq\sqrt{n-1},$$where
 $\sqrt{n-1}$ is  the diameter of any
face of the unit hypercube of $\Rn$. Therefore
\begin{eqnarray*}\big(\delta_{2r_1}{\rm
Box}(0,1/2)\cap\mathcal{I}(X)\big)\subseteq
\big(B_\varrho(0,1)\cap\mathcal{I}(X)\big)\subseteq
\big(\delta_{2r_2}{\rm Box}(0,1/2)\cap\mathcal{I}(X)\big)
\end{eqnarray*}and so\begin{eqnarray*}{(2r_1)}^{Q-1}&\leq&{(2r_1)}^{Q-1}\Ar({\rm Box}(0,1/2)\cap\mathcal{I}(X))
\leq
\Ar(B_\varrho(0,1)\cap\mathcal{I}(X))\\&=&\kappa_\varrho(X)\leq{(2r_2)}^{Q-1}\Ar({\rm
Box}(0,1/2)\cap\mathcal{I}(X))\leq\sqrt{n-1}{(2r_2)}^{Q-1}.\end{eqnarray*}}
we get that
$$(2r_1)^{Q-1}\leq k_\varrho(\nn(x))\leq
\sqrt{n-1}\,(2r_2)^{Q-1}.$$Set
$K_1:=(2r_1)^{Q-1}$,
$K_2:=\sqrt{n-1}\,{(2r_2)}^{Q-1}$. The previous argument shows that one can always
choose two positive constants $K_1,\,K_2$, independent of $S$, such
that 
\begin{eqnarray}\label{emfac}K_1\leq \kappa_\varrho(\nn(x))\leq
K_2\qquad \forall\,\,x\in {\rm Int}({S}\setminus C_{S}).
 \end{eqnarray} \end{oss}

\section{Isoperimetric Inequality on hypersurfaces}\label{mike}

\subsection{Statement of the main result and further remarks}\label{mike0}
 \begin{no} Set  $${\bf r}(S):=\sup_{x\in {\rm Int}(S\setminus C_S)} r_0(x),$$ where $r_0(x)=2\left(\frac{\per({S})}{k_\varrho(\nn(x))}\right)^{{1}/{Q-1}}$; see Lemma \ref{lem} and Notation \ref{klkl}.\end{no}

\begin{teo}[Isoperimetric-type Inequality]\label{ghaioio}Let
$S\subset\GG$ be a compact hypersurface of class $\cont^2$  with
boundary $\partial S$  (piecewise) $\cont^1$ and assume that the horizontal mean curvature $\MS$ of $S$ is integrable, that is,  $\MS\in L^1(S; \sigma\rr^{n-1})$. There exists  $C_{Isop}>0$ only
dependent on $\GG$ and on the homogeneous metric $\varrho$ such that   \begin{eqnarray}\label{2gha}\left(\per({S})\right)^{\frac{Q-2}{Q-1}}\leq
C_{Isop}\left(\int_S
|\MS|\,\per +\nis(\partial S)+ \sum_{i=2}^k  \left( {\bf r}(S)\right)^{i-1} \int_{\partial
S }|\P\ciss\eta|\,\sigma\rr^{n-2} \right).\end{eqnarray}In particular, if $\partial S=\emptyset$, it follows that \begin{eqnarray}\label{2gha}\left(\per({S})\right)^{\frac{Q-2}{Q-1}}\leq
C_{Isop} \int_S
|\MS|\,\per.\end{eqnarray}
\end{teo}
 
In general, we have $C_{Isop}=\max\{C_1, C_2\}$, where  $C_1= {2^{Q}}/{K_1^{\frac{1}{Q-1}}}$ and  $C_2=\max_{i=2}^k i\,{\bf c}_i\,h_i$. Here $K_1$ denotes a (universal) lower bound on the metric factor $k_\varrho(\nn(x))$; see Remark \ref{boundonmetricfactor}.  Furthermore, the constants ${\bf c}_i\,(i=2,...,k)$  has been introduced in Definition \ref{2iponhomnor1}. Note that if $\partial S=\emptyset$,  we can take $C_{Isop}= C_1$. The next example can be helpful in understanding our result. 
\begin{es}[Key example] \label{zazaz1}Let $\mathbb H^1$ be the first Heisenberg group. In particular, let $\{X, Y, T\}$ be the standard left invariant frame for the Lie algebra $\mathfrak h^1=\HH\oplus \mathrm{span}_\R T$ of $\mathbb H^1$. We recall that $X=\partial_x-\frac{y}{2}\partial_t$, $Y=\partial_y+\frac{x}{2}\partial_t$ and $T=\partial_t$, where $(x, y, t)$ are exponential coordinates of the generic point of  $\mathbb H^1$. Let $\mathcal{I}(X(0))=\{(x, y, t)\in\mathbb H^1: x=0\}$. The plane  $\mathcal{I}(X(0))$ is a \textquotedblleft vertical plane\textquotedblright passing through the identity  $0\in\mathbb H^1$. More precisely, $\mathcal{I}(X(0))$ turns out to be a maximal ideal of the Lie algebra $\mathfrak h^1$. It is well known that the horizontal mean curvature of any vertical plane turns out to be zero. Now let us consider a  rectangle $R_{\texttt{h,v}}\subsetneq \mathcal{I}(X(0))$ with sides parallel to the directions $Y$ and $T$, respectively. In other words, we are assuming that $$R_{\texttt{h,v}}=\left\lbrace (x, y, t)\in\mathbb H^1: x=0,\, |y|\leq \texttt{h},\, |t|\leq \texttt{v} \right\rbrace.$$The $\HH$-perimeter of $R_{\texttt{h,v}}$ coincides with the  Euclidean area and hence is obtained by multiplying  (horizontal) base and (vertical) height, that is,   $\sigma\cc^2(R_{\texttt{h,v}})=\texttt{h}\cdot \texttt{v}$. It is not difficult to see that the only non-zero contributions to the homogeneous measure $\sigma\cc^1$ of $\partial R_{\texttt{h,v}}$ come from the vertical sides. In fact  $\sigma\cc^1\left(\partial R_{\texttt{h,v}}\setminus\{x=0,\,|y|<\texttt{h}\}\right)=2\texttt{v}$. But, the \textquotedblleft horizontal sides\textquotedblright \,
have a non-zero Riemannian $1$-dimensional measure and, up to a normalization constant,  their $1$-dimensional intrinsic Hausdorff measure  is given by $\mathcal{H}_{\varrho}^1\left(\partial R_{\texttt{h,v}}\setminus\{x=0,\,|t|<\texttt{v}\}\right)=2\texttt{h}$. Hence, even if we fix the $\sigma\cc^1$-measure of $\partial R_{\texttt{h,v}}$, we can indefinitely increase the $\HH$-perimeter of $ R_{\texttt{h,v}}$ by increasing the size of the horizontal sides.
\end{es}

 Let
$S\subset\GG$ be a compact hypersurface of class $\cont^2$  with
boundary $\partial S$. Notice that the previous example shows that in order to bound the $\HH$-perimeter $\per$ in terms of the $(Q-2)$-homogeneous measure $\nis$ \underline{only}, we need some extra assumptions on the characteristic set $C_{\partial S}$ of the boundary. 
More precisely, we need -at least- to assume that $\sigma\rr^{n-2}(C_{\partial S})=0$. We also stress that, assuming enough regularity on $\partial S$ it can be sufficient for the validity of the last condition (in fact, we have already seen that if $\partial S$ is of class $\cont^2$ and if $\dim \VV\geq 2$, then $\dim_{\rm Eu-Hau}(C_N)\leq n-3$; see Theorem
\ref{baloghteo} and Remark \ref{11baloghteo}. We also recall that the same assertion holds true  for  Heisenberg groups $\mathbb H^r$ with $r>1$).

\begin{war}\label{cpzzo}  The present version of Theorem \ref{ghaioio}    corrects some  previous formulations of it posted on ArXiv. We would like to spend some words on this new version. First of all, we have to observe that the Isoperimetric Inequality stated in Theorem \ref{ghaioio} will be proved by following the classical scheme already discussed in the Introduction. In particular, the starting point is the so-called \rm Monotonicity Inequality, \it see Theorem \ref{rmonin}. More precisely, let
$S\subset\GG$ be a compact hypersurface of class $\cont^2$  with (piecewise) $\cont^1$
boundary $\partial S$. We shall show that for every $x\in {\rm Int}(S\setminus C_S)$ the
following ordinary differential inequality
holds\[-\frac{d}{dt}\frac{\per({S}_t)}{t^{Q-1}}
\leq \frac{\mathcal{A}(t)+{\mathcal{B}}_2(t)}{t^{Q-1}}
\]for $\mathcal{L}^1$-a.e. $t>0$, where $S_t=S\cap B_\varrho(x, t)$ and $B_\varrho(x, t)$ denotes the homogeneous $\varrho$-ball centered at $x$ and of radius $t$; for the very definition of the integrals $\mathcal{A}(t),\,{\mathcal{B}}_2(t)$ we refer the reader to Definition \ref{lsd34} below. The key fact in order to prove this inequality, will be a density type estimate; see Lemma \ref{kr}. The proof of the Isoperimetric Inequality can then be done once we estimate  the integrals $\mathcal{A}(t),\,{\mathcal{B}}_2(t)$. The first term can again be estimated by using a blow-up method and it turns out that $\mathcal{A}(t)\leq \int_{S_t}|\MS|\per$; see Lemma \ref{crux00}. Nevertheless, in order to estimate the integral ${\mathcal{B}}_2(t)$, \underline{we cannot  use local estimates} and/or blow-up results. More precisely, we stress that ${\mathcal{B}}_2(t)=\int_{\partial S\cap  B_\varrho(x, t)} f\,\nis$, for a suitable function $f:\partial S\longrightarrow \R_+$; see  Definition \ref{lsd34}. Below, we shall show that
${\mathcal{B}}_2(t)\leq \nis(\partial
S\cap B_\varrho(x, t)) + \widetilde{{\mathcal{B}}_2}(t)$ where $$\widetilde{{\mathcal{B}}_2}(t)\lesssim \sum_{i=2}^k  t^{i-1}\int_{\partial S\cap B_\varrho(x, t) }|\P\ciss\eta|\,\sigma\rr^{n-2};$$see Lemma  \ref{417}. Note that the  right-hand side turns out to be $(Q-2)$-homogeneous with respect to Carnot dilations but, in general, cannot be expressed in terms of the measure $\nis$ only. It is important to observe that no blow-up method can be profitably used here: the reason  is that the center of the $\varrho$-ball belongs to ${\rm Int}(S\setminus C_S)$. Hence there can be large balls intersecting a small portion of $\partial S$ and, on the contrary, there can be  small balls very close to  $\partial S$. 
\end{war}
Taken all together, the previous remarks suggest that, 
in order to prove a weaker formulation of the Isoperimetric Inequality for the $\HH$-perimeter $\per$,  which  only uses  the homogeneous measure $\nis$ on the boundary, we need -at least- some extra assumptions on the characteristic set $C_{\partial S}$ of the boundary and, in particular, \it is necessary that  $\sigma^{n-2}\rr(C_{\partial S})=0$. \rm
We end this introductory section by formulating an interesting related open question. 
\begin{Problem}\label{pope}Let $\Sigma^{n-1}$ denote the class of all compact hypersurfaces $S\subset\GG$  with (piecewise) $\cont^1$ boundary $\partial S$ such that  $\sigma^{n-2}\rr(C_{\partial S})=0$. Furthermore, let 
 us set
$$\mu(\partial S):=\sum_{i=2}^k  \left( {\bf r}(S)\right)^{i-1} \int_{\partial
S }|\P\ciss\eta|\,\sigma\rr^{n-2}.$$Is there a dimensional constant $C_{dim}<+\infty$ such that $\frac{\mu(\partial S)
}{\nis(\partial S)}\leq C_{dim}$ for every $S\in \Sigma^{n-1}$? 
\end{Problem}
 
 In other words, we are asking if  $\sup_{S\in \Sigma^{n-1}}\frac{\mu(\partial S)
}{\nis(\partial S)}<\infty.$ Notice  that the ratio $\frac{\mu(\partial S)
}{\nis(\partial S)}$ is $0$-homogeneous with respect to Carnot dilations. Furthermore,  $\frac{\mu(\partial S)
}{\nis(\partial S)}$ can always be estimated by\footnote{Roughly speaking,  this assertion can be proved by using the fact that, if $\sigma\rr^{n-2}(C_{\partial S})=0$, then $\sigma\rr^{n-2}\left(\left\{x\in\partial S: |\PH (\nu\wedge\eta)|\leq \epsilon\right\}\right)\longrightarrow 0$ as $\epsilon\rightarrow 0$, where $\nu\wedge\eta$ is any unit normal $2$-vector orienting $\partial S$.} a constant which depends on the characteristic set of $\partial S$ for any  $S\in \Sigma^{n-1}$. As a matter of fact, 
Problem \ref{pope} is equivalent to \underline{understand if such an estimate holds} \underline{with a universal constant}.
Clearly, a positive answer to this problem would  automatically imply the following  inequality:
$$\left(\per({S})\right)^{\frac{Q-2}{Q-1}}\leq
C'_{Isop}\left(\int_S
|\MS|\,\per +\nis(\partial S)\right),$$with $C'_{Isop}=C_{Isop}(1+C_{dim})$. Note also that, a (purely) horizontal Sobolev-type inequality can be proved  \underline{only if} the last inequality holds true.

\begin{oss}\label{ues}An equivalent formulation of Problem \ref{pope} is the following:\begin{itemize}
 \item  are there  dimensional constants $0<C_i<+\infty$, $i\in\{2,...,k\}$, such that $$\left(\per(S)\right)^{\frac{i-1}{Q-1}} \leq C_i\frac{\nis(\partial S)}{ \int_{\partial
S }|\P\ciss\eta|\,\sigma\rr^{n-2}
} $$ whenever $S\in \Sigma^{n-1}$? 
\end{itemize} 
\end{oss}

\begin{es}[The case of the Heisenbeg group $\mathbb H^1$] In the first Heisenberg group $\mathbb H^1$, the problem just formulated in Remark \ref{ues} becomes: is there a constant $0<C<+\infty$ such that $$\left(\sigma\cc^3(S)\right)^{\frac{1}{3}} \leq C \frac{\sigma\cc^1(\partial S)}{ \int_{\partial
S }|\P{_{^{_{\HH_2 S}}}}\eta|\,\sigma\rr^{1}
} $$for any $S\in \Sigma^{n-1}$? Here $\HH_2 S$ corresponds to the tangential direction $\mathbf t=|\PH \nu| T-\langle T, \nu \rangle \nn$. It is not difficult to show that $\sigma\cc^1(\partial S)$ coincides with the integral over $\partial S$ of the contact form $\theta=T^\ast=dz+\frac{ydx-xdy}{2}$ and  hence equals the Euclidean area of the projection of $S$ onto the $xy$-plane. Moreover, the integral at the denominator can be regarded (up to a normalization constant) as the $1$-dimensional intrinsic Hausdorff measure $\mathcal H^1_{\varrho}$ of  $\partial S$. Obviously, in order to prove this inequality, the assumption that  $\sigma^{1}\rr(C_{\partial S})=0$ \underline{cannot} be removed. 
\end{es}
The next sections are devoted to prove Theorem \ref{ghaioio}. Finally,
in Section \ref{sobineqg} we shall discuss some related Sobolev-type
inequalities.

\subsection{Linear isoperimetric inequality and monotonicity
formula}\label{wlineq}

Let $S\subset\GG$ be a compact hypersurface of class $\cont^2$   with
boundary $\partial{S}$. Let $\nu$ denote the outward-pointing unit
normal vector along $S$ and $\varpi=\frac{\P\vv\nu}{|\PH\nu|}$. Furthermore, we shall
set$$\varpi\ci:=\P\ci\varpi=\sum_{\alpha\in I\ci}\varpi_\alpha
X_\alpha$$for $i=2,...,k$. Note that
$\frac{\nu}{|\PH\nu|}=\nn+\sum_{i=2}^k\varpi\ci$.

\begin{no}Let
$\eta$ be the outward-pointing unit normal vector $\eta$ along $\partial
S$. Note that, at each point $x\in\partial S$,
$\eta(x)\in\TT_xS$. In the sequel, we shall set
 $\chi:=\frac{\P\vs\eta}{|\P\ss\eta|}$ and $\chi\ciss:=\P\ciss\chi$ for any
$i=2,...,k$; see Remark \ref{indbun}.\end{no}

We have
$\chi=\sum_{i=2}^k\chi\ciss$ and
$\frac{\eta}{|\P\ss\eta|}=\eta\ss+\chi$; see also Remark
\ref{measonfr}.
\begin{Defi}\label{berlu}Fix a point $x\in\GG$ and consider the \textquotedblleft Carnot homothety\textquotedblright centered
at $x$, that is,  $\delta^x(t,y):=x\bullet\delta_t (x^{-1}\bullet
y)$. The variation vector  of
$\delta^x_t(y):=\delta^x(t,y)$ at $t=1$ is given by
 $Z_x:=\frac{\partial \delta^x_t}{\partial
t}\bigg|_{t=1}.$ \end{Defi}

Let us apply the 1st variation of $\per$, with a  special 
choice of the variation vector. So fix
a point $x\in\GG$ and consider the Carnot homothety
$\delta_t^x(y):=x\bullet\delta_t (x^{-1}\bullet y)$ centered at
$x$.\begin{oss}Without loss of generality, by using group translations, we can choose $x=0\in\GG$. In this case, we have
$$\vartheta^0(t,y)=\delta_ty=\esp \left( ty\cc,t^2y\cd,
t^3y\ctr,...,t^iy\ci,...,t^ky\ck \right) \qquad \forall\,\,t\in
\R,$$where $y\ci=\sum_{j_i\in I\ci} y_{j_i}\ee_{j_i}$ and  $\esp$
is the Carnot exponential mapping; see Section \ref{prelcar}.
Thus the variation vector  related to
$\delta^0_t(y):=\delta^0(t,y)$, at $t=1$, is simply given by
$$Z_0:=\frac{\partial \delta^0_t}{\partial
t}\Big|_{t=1}=\frac{\partial \delta_t}{\partial
t}\Big|_{t=1}=y\cc+ 2y\cd+...+ky\ck.$$
 
\end{oss}
By  invariance of $\per$ under Carnot dilations, one
gets$$\frac{d}{dt}\delta_t^\ast\per\Big|_{t=1}=(Q-1)\,\per({S}).$$Furthermore,
by using the 1st variation formula, it follows that
\begin{equation*}(Q-1)\,\per({S})=-\int_{{S}}\MS\left\langle Z_{x}, \frac{\nu}{|\PH\nu|}\right\rangle\,\per +
\int_{\partial{S}} \left\langle \left( Z_x\ot-\frac{\langle Z_x\op,\nu\rangle}{|\PH\nu|}\nn\ot \right),\frac{\eta}{|\P\ss\eta|}\right\rangle\,\nis.\end{equation*}

\begin{lemma}The following holds  \begin{eqnarray}\label{inA}\frac{1}{\varrho_{x}}\left|\left\langle Z_{x},
\frac{\nu}{|\PH\nu|}\right\rangle\right| \leq 
\left(1+\sum_{i=2}^k i\,{\bf c}_i\,\varrho_{x}^{i-1}|\varpi\ci|\right).\end{eqnarray}Furthermore, we have $\left(1+\sum_{i=2}^k i\,{\bf c}_i\,\varrho_{x}^{i-1}|\varpi\ci|\right)\leq 1+O\left(\frac{\varrho_x}{|\P\cc\nu|}\right) $ as long as $\varrho_x\rightarrow 0^+$.
\end{lemma}
Here and elsewhere, we  use  the \textquotedblleft Big O\textquotedblright notation and the \textquotedblleft little o\textquotedblright notation.
\begin{proof}Without loss of generality, by left-invariance, let $x=0\in\GG$. 
Note
that
$$\left\langle Z_{0}, \frac{\nu}{|\PH\nu|}\right\rangle=
\langle Z_{0}, (\nn+\varpi)\rangle=\langle
y\cc,\nn\rangle+\sum_{i=2}^k\langle
y\ci,\varpi\ci\rangle.$$ 
By Cauchy-Schwartz inequality, we immediately get that
 \[\left|\left\langle Z_{0},
\frac{\nu}{|\PH\nu|}\right\rangle\right|\leq |y\cc|+\sum_{i=2}^k
i\,|y\ci||\varpi\ci|.\] According with Definition
\ref{2iponhomnor1}, let ${\bf c}_i\in\R_+$ be constants such that
$|y\ci|\leq {\bf c}_i \varrho^i(y)$ for $i=2,...,k.$ Using the last
inequality yields 
\[\left|\left\langle Z_{0},
\frac{\nu}{|\PH\nu|}\right\rangle\right| \leq 
\varrho\left(1+\sum_{i=2}^k i\,{\bf c}_i\varrho^{i-1}|\varpi\ci|\right)\leq\varrho\left(1+O\left(\frac{\varrho}{|\P\cc\nu|} \right) \right)\] as long as $\varrho\rightarrow 0^+$.
\end{proof}

\begin{Defi}\label{lsd34}Let $\GG$ be a $k$-step Carnot group and $S\subset\GG$ be a 
hypersurface of class $\cont^2$ with (piecewise)
$\cont^1$ boundary $\partial S$. Moreover, let $S_r:=S\cap B_\varrho(x,r)$, where
$B_\varrho(x,r)$ is the open $\varrho$-ball centered at $x\in \GG$
and of radius $r>0$. We shall set
\begin{eqnarray*}\mathcal{A}(r)&:=&\int_{S_r}|\MS|\left(1+\sum_{i=2}^k
i\,c_i\varrho_x^{i-1}|\varpi\ci|\right)\,\per,\\\mathcal{B}_0(r)&:=&\int_{\partial
S_r}\frac{1}{\varrho_x}\left|\left\langle
\left( Z_x\ot-\frac{\langle Z_x\op,\nu\rangle}{|\PH\nu|}\nn\ot \right),\frac{\eta}{|\P\ss\eta|}\right\rangle\right|\,\nis,\\\mathcal{B}_1(r)&:=&\int_{\partial
B_\varrho(x, r)\cap S}\frac{1}{\varrho_x}\left|\left\langle
\left( Z_x\ot-\frac{\langle Z_x\op,\nu\rangle}{|\PH\nu|}\nn\ot \right),\frac{\eta}{|\P\ss\eta|}\right\rangle\right|\,\nis,\\{\mathcal{B}}_2(r)&:=&\int_{\partial
S\cap B_\varrho(x, r)}\frac{1}{\varrho_x}\left|\left\langle
\left( Z_x\ot-\frac{\langle Z_x\op,\nu\rangle}{|\PH\nu|}\nn\ot \right),\frac{\eta}{|\P\ss\eta|}\right\rangle\right|\,\nis,\end{eqnarray*}where
$\varrho_x(y):=\varrho(x,y)$ for $y\in S$, that is,  $\varrho_x$
denotes the $\varrho$-distance from a fixed point $x\in\GG$.
\end{Defi}
Note that $\mathcal{B}_0(r)=\mathcal{B}_1(r)+\mathcal{B}_2(r)$. We clearly have the following:

\begin{Prop}[Linear Inequality]\label{correctdimin}Let ${S}\subset\GG$ be a compact hypersurface of class $\cont^2$  
 with (piecewise) $\cont^1$ boundary $\partial{S}$. Let
 $r$ be the radius of a $\varrho$-ball centered at $x\in\GG$. Then
\begin{eqnarray*}(Q-1)\,\per({S}_r)\leq
{r}\left(\mathcal{A}(r)+\mathcal{B}_0(r)\right).\end{eqnarray*}
\end{Prop}\begin{proof}Immediate.\end{proof}

\begin{oss}In the sequel we will need
 the following property: \rm there exists
$r_S>0$ such that
\begin{equation}\label{0key} \int_{S_{r+h}\setminus S_r}\frac{1}{\varrho_x}\left|\left\langle
\left( Z_x\ot-\frac{\langle Z_x\op,\nu\rangle}{|\PH\nu|}\nn\ot \right), \grad\ts\varrho_x\right\rangle\right|\,\per \leq  \per(S_{r+h}\setminus S_r)
\end{equation}for $\per$-a.e.  $x\in{\rm Int}\,S$, for
 $\mathcal{L}^1$-a.e. $r,\, h>0$ such that $r+h\leq r_S $. \it 
 In the classical setting the previous inequality easily follows from a key-property of the Euclidean metric $d\Eu$, that is the
  {\rm Ikonal equation} $|\grad\Eu d\Eu|=1$. In fact,  $Z_x(y)=y-x$
and, since $\nn$ coincides with $\nu$, one has $\nn\ot=0$. Thus setting $\varrho_x(y):=d\Eu(x, y)=|x-y|$ yields
\[\frac{\left|\left\langle Z_x(y),
\grad\ts\varrho_x(y)\right\rangle\right|}{\varrho_x(y)}=1-\left\langle\frac{y-x}{|y-x|},\textsl{n}_{\ee}\right\rangle^2\leq
1,\]where $\textsl{n}_{\ee}$ denotes the Euclidean  unit normal of $S$. In particular, we may take $r_S=+\infty$.
A stronger version of \eqref{0key} is a natural assumption in the
Riemannian setting. At this regard we refer the reader to a paper by Chung,
Grigor'jan and Yau  where this hypothesis is the
starting point of a general theory about isoperimetric
inequalities on weighted Riemannian manifolds and graphs; see \cite{CGY}.\end{oss}
 
\noindent It is worth  observing that\footnote{\label{19}We stress that \eqref{0kinawa2} holds true
for every (smooth enough) homogeneous distance
 on any Carnot group $\GG$.
\begin{lemma}\label{cardcaz}Let $\GG$ be a $k$-step Carnot group and let
$\varrho:\GG\times\GG\longrightarrow\R_+$ be any $\cont^1$-smooth
homogeneous norm. Then $\frac{1}{\varrho_x}\langle Z_x,
\grad\,\varrho_x\rangle=1$ for every $x\in\GG$.
\end{lemma}\begin{proof}By homogeneity and left-invariance of $\varrho$. More
precisely,  we have
$t\varrho(z)=\varrho(\delta_t z)$ for all $t>0$ and for every
$z\in\GG$. Setting $z:=x^{-1}\bullet y$, we get that $t\varrho(x,
y)=t\varrho(z)=\varrho(\delta_t z)=\varrho(x, x\bullet\delta_t z)$
for every $x, y \in\GG$ and for all $t>0$. Hence
$\varrho(x,y)=\frac{d}{dt}\varrho(x,
 x\bullet\delta_t x^{-1}\bullet y)\big|_{t=1}=\langle\grad\,\varrho_x(y),Z_x(y)
\rangle,
$ and  the claim follows.\end{proof}}
\begin{equation}\label{0kinawa2}\frac{1}{\varrho_x}\langle Z_x, \grad\,\varrho_x\rangle=1\end{equation}for every $x\in\GG$.
 The last identity  can be used  to rewrite \eqref{0key}. More precisely, we have
\[\frac{\left|\left\langle\left( Z_x\ot(y)-\frac{\langle Z_x\op(y),\nu\rangle}{|\PH\nu|}\nn\ot \right),
\grad\ts\varrho_x(y)\right\rangle\right|}{\varrho_x(y)}=\left|1-\frac{\left\langle
\grad\cc\,\varrho_x(y),\nn(y)\right\rangle \left\langle Z_x(y),
\frac{\nu(y)}{|\PH\nu|}\right\rangle}{\varrho_x(y)}\right|.\]
Hence \eqref{0key} can be formulated as follows:\\

\noindent $(\spadesuit)$\,\it There exists
$r_S>0$ such that:
\begin{eqnarray}\label{0key2}\int_{S_{r+h}\setminus S_r}\left|1-\frac{\left\langle
\grad\cc\,\varrho_x(y),\nn(y)\right\rangle \left\langle Z_x(y),\left(\nn(y)+\varpi(y)\right) \right\rangle}{\varrho_x(y)}\right|\,\per(y)\leq \per \left(S_{r+h}\setminus S_r\right) 
\end{eqnarray} for $\per$-a.e.  $x\in{\rm Int}\,S$, for
 $\mathcal{L}^1$-a.e. $r,\, h>0$ such that $r+h\leq r_S $.\\\rm

\begin{lemma}[Key result]\label{kr}Let $x \in {\rm Int}(S\setminus C_S)$. Set $$\pi(S_t):=\int_{ S_t}\left|1-\frac{\left\langle
\grad\cc\,\varrho_x(y),\nn(y)\right\rangle \left\langle Z_x(y),\left(\nn(y)+\varpi(y)\right) \right\rangle}{\varrho_x(y)}\right|\,\per(y)$$for $t>0$. Then $$\lim
_{t\rightarrow 0^+}\frac{\pi(S_t)}{\per(S_t)}=1.$$ 
\end{lemma}

\begin{oss}\label{kr1}
By using standard results about differentiation of measures, it follows from  Lemma \ref{kr} that $ \pi(S_t)=\per(S_t)$ for $\per$-a.e. $x\in S$, for all $t>0$; see, for instance, Theorem 2.9.7 in \cite{FE}. Thus, we get  that
$\pi\left(S_{t+h}\setminus S_t\right)=\per\left(S_{t+h}\setminus S_t\right)$ for $\per$-a.e. $x\in S$ and for every $t, h\geq 0$. \underline{In particular, we may choose} $r_S=+\infty$.
\end{oss}

\begin{proof}[Proof of Lemma \ref{kr}]Let $x \in {\rm Int}(S\setminus C_S)$ and note that $S_t=\vartheta_{t}^{x}\left( \delta_{\frac{1}{t}}^{x} S\cap B_{\varrho}(x, 1)\right)$ for all $t>0$. So we have
\begin{eqnarray*}\frac{ \pi(S_t)}{\per \left(S_t\right)}&=&\frac{\int_{\vartheta_{t}^{x}\left( \delta_{\frac{1}{t}}^{x} S\cap B_{\varrho}(x, 1)\right) }\left|1-\frac{\left\langle
\grad\cc\,\varrho_x(y),\nn(y)\right\rangle \left\langle Z_x(y),\left(\nn(y)+\varpi(y)\right) \right\rangle}{\varrho_x(y)}\right|\,\per(y)}{\per \left(\vartheta_{t}^{x}\left( \delta_{\frac{1}{t}}^{x} S\cap B_{\varrho}(x, 1)\right)\right)} \\&=&\frac{\int_{  \delta_{\frac{1}{t}}^{x} S\cap B_{\varrho}(x, 1) }\left|1-\frac{\left\langle
\grad\cc\,\varrho_x(\delta_{t}^{x}(z)),\nn(\delta_{t}^{x}(z))\right\rangle \left\langle Z_x(\delta_{t}^{x}(z)),
\left(\nn(\delta_{t}^{x}(z))+\varpi(\delta_{t}^{x}(z))\right) \right\rangle}{\varrho_x(\delta_{t}^{x}(z))}\right|\,\per(z)}{\per \left( \delta_{\frac{1}{t}}^{x} S\cap B_{\varrho}(x, 1) \right)} \\&=&\frac{\int_{  \delta_{\frac{1}{t}}^{x} S\cap B_{\varrho}(x, 1)}\left|1-\frac{\left\langle
\grad\cc\,\varrho_x(\delta_{t}^{x}(z)),\nn(\delta_{t}^{x}(z))\right\rangle \left\langle t\left( [Z_x(z)]\cc+ \overrightarrow{O}(t)\right),
\left(\nn(\delta_{t}^{x}(z))+\varpi(\delta_{t}^{x}(z))\right) \right\rangle}{t\varrho_x(z)}\right|\,\per(z)}{\per \left( \delta_{\frac{1}{t}}^{x} S\cap B_{\varrho}(x, 1)\right)} \\&=&\frac{\int_{ \delta_{\frac{1}{t}}^{x} S\cap B_{\varrho}(x, 1)}\left|1-\frac{\left\langle
\grad\cc\,\varrho_x(\delta_{t}^{x}(z)),\nn(\delta_{t}^{x}(z))\right\rangle \left\langle [Z_x(z)]\cc,
 \, \nn(\delta_{t}^{x}(z))\right\rangle}{\varrho_x(z)}\right|\,\per(z)}{\per \left( \delta_{\frac{1}{t}}^{x} S\cap B_{\varrho}(x, 1)\right)} + O(t),
\end{eqnarray*}as long as $t\rightarrow 0^+$.
It is worth observing that the horizontal gradient of $\varrho_x$ turns out to be homogeneous of degree $0$, and so independent of $t$. Note  that $[Z_x(z)]\cc=\PH(Z_x)(z)=z\cc-x\cc$. Therefore
\[\lim_{t\rightarrow 0^+}\frac{ \pi(S_t)}{\per \left(S_t\right)}=\frac{\int_{S_{\infty}\cap B_{\varrho}(x, 1)  }\left|1-\frac{\left\langle
\grad\cc\,\varrho_x(z),\nn(x)\right\rangle \left\langle \left(z\cc-x\cc\right),
 \, \nn(x)\right\rangle}{\varrho_x(z)}\right|\,\per(z)}{\per \left(S_{\infty}\cap B_{\varrho}(x, 1) \right)} =:L.
\]Recall that, by Theorem \ref{BUP}, we have $S_{\infty}=\mathcal{I}(\nn(x))$. This implies that $\left\langle \left(z\cc-x\cc\right),
 \, \nn(x)\right\rangle=0$ whenever $z\in \mathcal{I}(\nn(x))$ and hence $L=1$, as wished.
\end{proof}

At this point, starting from Proposition \ref{correctdimin}, we may
 prove a monotonicity formula for the $\HH$-perimeter
 $\per$, which is one of our main results. We shall set ${S}_t:={S}\cap
{B_\varrho}(x,t)$, for $t>0$.

\begin{teo}[The monotonicity inequality for the measure $\per$]\label{rmonin}  Let ${S}\subset\GG$ be a
compact hypersurface of class $\cont^2$ with (piecewise) $\cont^1$ boundary $\partial S$. Then for every $x\in {\rm Int}(S\setminus C_S)$ the
following ordinary differential inequality
holds\begin{eqnarray}\label{rmytn}-\frac{d}{dt}\frac{\per({S}_t)}{t^{Q-1}}
\leq \frac{\mathcal{A}(t)+{\mathcal{B}}_2(t)}{t^{Q-1}}
\end{eqnarray}for $\mathcal{L}^1$-a.e. $t>0$.
\end{teo}\begin{proof}By
applying Sard's Theorem we get that ${S}_t$ is a 
manifold of class $\cont^2$ with boundary for $\mathcal{L}^1$-a.e. $t>0$. From the
inequality in
 Proposition \ref{correctdimin}  we have
\[(Q-1)\,\per({S}_t)\leq
t\left(\mathcal{A}(t)+\mathcal{B}_0(t)\right)\]for
$\mathcal{L}^1$-a.e. $t>0$, where $t$ is the radius of a
$\varrho$-ball centered at $x\in {\rm Int}\,S$. Since
$$\partial{S}_t=\{\partial
B_\varrho(x,t)\cap {S}\}\cup\{\partial{S}\cap B_\varrho(x,t)\},$$
we get that
\begin{eqnarray*}(Q-1)\,\per({S}_t)\leq t\,\left(\mathcal{A}(t) +
\mathcal{B}_1(t)+\mathcal{B}_2(t)\right).\end{eqnarray*}We estimate
$\mathcal{B}_1(t)$ by using \eqref{0key} and Coarea Formula. For
every $t, h>0$ one has
\begin{eqnarray*}\int_{t}^{{t+h}}\mathcal{B}_1(s)\,ds&=&\int_{t}^{{t+h}}
\int_{\partial B_\varrho(x,s)\cap
{S}}\frac{1}{\varrho_{x}}\left|\left\langle \left( Z_x\ot-\frac{\langle Z_x\op,\nu\rangle}{|\PH\nu|}\nn\ot \right),
\frac{\eta}{|\P\ss\eta|}\right\rangle\right|\nis \\&=&
\int_{S_{t+h}\setminus S_t}\frac{1}{\varrho_{x}}\left|\left\langle
\left( Z_x\ot-\frac{\langle Z_x\op,\nu\rangle}{|\PH\nu|}\nn\ot \right), \frac{\grad\ts\varrho_{x}}{|\qq\varrho_{x}
|}\right\rangle\right||\qq
\varrho_x|\,\per\\&=&\int_{S_{t+h}\setminus
S_t}\per,\end{eqnarray*}where we have used the following facts:
\begin{itemize}
\item $\eta=\frac{\grad\ts{\varrho_x}}{|\grad\ts{\varrho_x}|}$ and
$\eta\ss=\frac{\qq\varrho_x}{|\qq\varrho_x|}$ along $\partial
B_\varrho(x,s)\cap {S}$  for $\mathcal{L}^1$-a.e. $s\in]t,
t+h[$;\item Coarea
formula \eqref{1coar} together with Lemma \ref{kr} and Remark \ref{kr1}.\end{itemize}Therefore $$\frac{\int_{t}^{{t+h}}\mathcal{B}_1(s)\,ds}{h}=
\frac{\per(S_{t+h}\setminus S_t)}{h} $$as long as $h\rightarrow
0^+$ and hence $\mathcal{B}_1(t)=\frac{d}{dt}\,\per({S}_t) $
for $\mathcal{L}^1$-a.e. $t>0$. So we get that
\[(Q-1)\,\per({S}_t)\leq t\left(\mathcal{A}(t)+{\mathcal{B}}_2(t)+\frac{d}{dt}\,\per({S}_t)\right)\]which
is equivalent to \eqref{rmytn}.
\end{proof}

\subsection{Further estimates}\label{perdindirindina}
In this section we study the integrals
$\mathcal{A}(t)$ and $\mathcal{B}_2(t)$  appearing in the
right-hand side of the monotonicity formula
\eqref{rmytn}.\\

\noindent{\bf \large Estimate of $\mathcal{A}(t)$.}
\begin{lemma}\label{perdopo}Let $S\subset\GG$ be a 
hypersurface of class $\cont^k$, let $x\in {\rm Int}\,S$ and let $S_t=S\cap B_{\varrho}(x,
t)$ for some $t>0$. Then there exists a constant ${\bf b}_\varrho>0$, only
dependent on $\varrho$ and $\GG$, such that
\begin{eqnarray}\label{margheritacarosio}
\lim_{t\rightarrow
0^+}\frac{{\int_{{S}_t}|\varpi\ci|\,\per}}{t^{Q-i}}\leq
\DH_i\,{\bf b}_\varrho\qquad\mbox{for every }\,\, i=2,..., k\end{eqnarray}where $\DH_i=\dim\HH_i$.
\end{lemma}

\begin{proof} For any
$\alpha=\DH+1,...,n$, we have
$\left(X_\alpha\LL\Vol\right)\big|_S= \langle X_\alpha, \nu\rangle\,\sigma^{n-1}\rr \big|_S=
 \ast\omega_\alpha |_S$, where $\ast$ denotes the Hodge star
operator; see \cite{Helgason}. Moreover
$\delta_t^\ast(\ast\omega_\alpha)=t^{Q-{\rm
ord}(\alpha)}(\ast\omega_\alpha)$ for every $t>0$. So we get that
$${\int_{{S}_t}|\varpi\ci|\,\per}=\int_{{S}_t}|\P\ci\nu|\,\sigma^{n-1}\rr \leq \sum_{{\rm ord}(\alpha)=i}\int_{S_t} |X_\alpha\LL\Vol|=
\sum_{{\rm ord}(\alpha)=i}t^{Q-i}\int_{\delta^x_{\frac{1}{t}}S\cap
B_\varrho(x,
1)}\left|(\ast\omega_\alpha)\circ\delta^x_t\right|.$$Since
\[\int_{\delta^x_{\frac{1}{t}}S\cap
B_\varrho(x,
1)}\left|(\ast\omega_\alpha)\circ\delta^x_t\right|\leq
\sigma^{n-1}\rr\left(\delta^x_{\frac{1}{t}}S\cap B_\varrho(x,
1)\right),\]by using Theorem \ref{BUP} we may pass to the limit as
$t\rightarrow 0^+$ the right-hand side. More precisely, if $x\in
{\rm Int}(S\setminus C_S)$ the rescaled hypersurfaces
$\delta^x_{\frac{1}{t}}S$ converge to the vertical hyperplane
$\mathcal{I}(\nn(x))$ as $t\rightarrow 0^+$. Otherwise, $x\in {\rm Int}(S\cap C_S)$ and we can assume that ${\rm ord}(x)=Q-i$, for
some $i=2,..., k$. We also recall that the limit-set $S_\infty$ is a polynomial
hypersurface of homogeneous degree  $i$ passing through $x$; see
Remark \ref{fraccicarla}.  So let
us
set $b_1:=\sup_{X\in\HH,\,|X|=1}\sigma^{n-1}\rr(\mathcal{I}(X)\cap
B_\varrho(0, 1))$, where $\mathcal{I}(X)$ denotes the
vertical hyperplane through $0\in\GG$ and orthogonal to $X$.
In order to study the characteristic case, let $b_2:=\sup_{\Psi\in\mathcal{P}ol^{k}_0}\sigma^{n-1}\rr(\Psi\cap
B_\varrho(0, 1))$, where $\mathcal{P}ol^{k}_0$ denotes
the class of all graphs of polynomial functions of homogeneous degree $\leq k$, passing through
$0\in\GG$.  Using the left-invariance of
$\sigma^{n-1}\rr$ and setting
\begin{equation}\label{tonda3}{b}_\varrho:=\max\{b_1,\,b_2\},\end{equation} yields
$\lim_{t\rightarrow0^+}\sigma^{n-1}\rr\left(\delta^x_{\frac{1}{t}}S\cap B_\varrho(x,
1)\right)\leq {b}_\varrho$. Therefore
\begin{eqnarray*}
\lim_{t\rightarrow0^+}\frac{{\int_{{S}_t}|\varpi\ci|\,\per}}{t^{Q-i}}\leq\lim_{t\rightarrow0^+} \DH_i\,
\sigma^{n-1}\rr\left(\delta^x_{\frac{1}{t}}S\cap B_\varrho(x,
1)\right)\leq \DH_i\, {b}_\varrho\end{eqnarray*} which achieves
the proof of \eqref{margheritacarosio}.
\end{proof}

\begin{oss}If $S$ is
just of class $\cont^2$, then \eqref{margheritacarosio} holds
for every $x\in {\rm Int}(S\setminus C_S)$. The same assertion  holds if $x\in C_S$ has order
${\rm ord}(x)=Q-i$ for some $i=2,...,k$ and $S$ is of class
$\cont^i$.
\end{oss}

\indent Let $S\subset\GG$ be of class $\cont^2$, let $x\in{\rm
Int}(S\setminus C_S)$ and  $S_t=S\cap B_\varrho(x,
t)$. Moreover, let $\mathcal{A}(t)$ be as in Definition
\ref{lsd34}. By  applying Theorem \ref{baloghteo}, we get that
$\dim_{\rm Eu-Hau}(C_{ S})\leq n-2$. In particular $\sigma\rr^{n-1}$-a.e. interior point of $S$ is non-characteristic.

\begin{lemma}\label{crux00}Under the previous assumptions, one has
\begin{equation}\label{cruxii}\mathcal{A}(t)\leq \int_{S_t}|\MS|\,\per.\end{equation}
\end{lemma}
\begin{proof}First, note that $\varrho_x(y)=\varrho(x, y)\rightarrow 0^+$
as $t\rightarrow 0^+$. Hence
\begin{eqnarray*}\mathcal{A}(t)=\int_{S_t}|\MS|\left(1+\sum_{i=2}^k
i\,{\bf c}_i\varrho_x^{i-1}|\varpi\ci|\right)\,\per\leq
\int_{S_t}|\MS|\left(1+\frac{2
{\bf c}_2\varrho_x\left(1 + o(1)\right)}{|\PH\nu|}\right)\,\per
\end{eqnarray*}as long as $t\rightarrow
0^+$. Note that $\frac{1}{|\PH\nu|}$ is continuous near $x\in{\rm
Int}(S\setminus C_S)$. Since $\MS$ turns out to be continuous near
every non-characteristic point,  by using standard differentiation
results in Measure Theory (see Theorem 2.9.7 in \cite{FE}),  we get that
\[\lim_{t\rightarrow
0^+}\frac{ \int_{S_t}|\MS|\left(\frac{2 {\bf c}_2 \varrho_x\left(1 +
o(1)\right)}{|\PH\nu|}\right) \per }{\per(S_t)}=0 \] and \eqref{cruxii}  
follows.\end{proof}

Actually, a similar result holds true even if $x\in {\rm
Int}(S\cap C_S)$, at least whenever $\MS$ is bounded and $S$ is smooth enough near
$C_S$. Below we shall make  use of Theorem
\ref{BUP}, Case (ii).

\begin{lemma}\label{cux0}Let $S$ be a hypersurface of class $\cont^k$ and assume that $\MS$ is bounded on $S$. Let
$x\in {\rm Int}(S\cap C_S)$ be an interior  characteristic point
such that ${\rm ord}(x)=Q-i$, for some $i=2,...,k$. This means that
there exists $\alpha=\DH+1,...,n$, ${\rm ord}(\alpha)=i$, such
that $S$ can be represented, locally around $x$, as a
$X_\alpha$-graph for which
\eqref{0dercond} holds. Then there exists a constant
${\bf d}_\varrho>0$, only dependent on $\varrho$ and $\GG$, such that
$\mathcal{A}(t)\leq
\|\MS\|_{L^\infty(S)}\,\,\left(\kappa_\varrho(C_S(x))+
{\bf d}_\varrho\right)\,t^{Q-1}$ as long as $t\rightarrow
0^+$. In particular, we have$$\mathcal{A}(t)\leq \|\MS\|_{L^\infty(S)}\,\left(\kappa_\varrho(C_S(x))+
{\bf d}_\varrho\right)\mathcal{S}_{\varrho}^{Q-1}(S_t)$$for all $t>0$, where $\mathcal{S}_{\varrho}^{Q-1}$ denotes the  
  spherical Hausdorff measure computed with respect to the  the homogeneous distance $\varrho$.
\end{lemma}
\begin{proof}Using Lemma \ref{perdopo} we obtain \begin{eqnarray*}\frac{\sum_{i=2}^k\int_{S_t}
i\,{\bf c}_i\varrho_x^{i-1}|\varpi\ci|\,\per}{t^{Q-1}}&\leq&{\sum_{i=2}^k
i\,{\bf c}_i\DH_i\,{\bf b}_\varrho}
\end{eqnarray*}as $t\rightarrow 0^+$,
 where ${\bf b}_\varrho$ is the constant defined by \eqref{tonda3}. Finally, the thesis follows by setting ${\bf d}_\varrho:={\sum_{i=2}^k
i\,{\bf c}_i\DH_i {\bf b}_\varrho}$ and  by using a well-known density estimate; see  Theorem 2.10.17 in \cite{FE}.
\end{proof}
 The previous result will be applied only in Section \ref{asintper}.\\

 \noindent{\bf \large Estimate of
${\mathcal{B}}_2(t)$.}\\

We  define a family of (homogeneous) measures on $(n-2)$-dimensional submanifolds of $\GG$.
\begin{Defi}\label{xazax} Let $N\subset\GG$ be a $(n-2)$-dimensional submanifold of class
$\cont^1$. Let $\xi\in\XX^0(\nn N)$ be a unit horizontal normal vector field  to $N$ and let $\alpha\in I\vv=\{h+1,...,n\}$ be such that ${\rm ord}(\alpha)=i$. Then we  define a  $(Q-i-1)$-homogeneous measure  $\mu_\alpha\equiv\mu_\alpha(\xi)$ on $N$ by setting$$\mu_\alpha(N\cap B):=  \int_{N\cap B}|(\xi\wedge X_\alpha)\LL\Vol|\qquad \forall \,\,B\in\mathcal{B}or(\GG).$$
\end{Defi}Note that  $|(\xi\wedge X_\alpha)\LL\Vol|$ is the norm of the differential $(n-2)$-form  $(\xi\wedge X_\alpha)\LL\Vol\big|_N$; see \cite{FE}, pp. 31-32. Now let $\underline{\tau}=\{\tau_1,...,\tau_n\}$ be a \it graded orthonormal frame \rm on an open neighborhood of $N\subset \GG$. This means that $\{\tau_{j_i}: j_i\in I\ci\}$ is a orthonormal basis of the $i$-th layer $\HH_i$ for any $i=1,...,k$.  Furthermore, let us denote by $\underline{\phi}=\{\phi_1,...,\phi_n\}$ the \it dual coframe \rm of  $\underline{\tau}$ defined by duality as   $\phi_i(\tau_j)=\delta_{i}^{j}$, where $\delta_{i}^{j}$ is the 
 {\it Kronecker delta}, that is,  $\delta_{i}^{j}=1$ if $i=j$ and $0$ otherwise. For simplicity, we also assume that $\tau_1=\xi$ where $\xi\in\XX^0(\nn N)$ is a unit horizontal normal vector field  to $N$. Using these new coordinates, it follows that $\xi\wedge X_\alpha=\tau_1\wedge\tau_\alpha$ and  we  get that $(\xi\wedge X_\alpha)\LL\Vol\big|_N=\ast \phi_1\wedge\phi_\alpha\big|_N$, where $\ast$ denotes the \it Hodge star operator \rm on differential forms; see \cite{Lee}, \cite{Helgason}. Clearly, the homogeneity degree (or \it weight\rm ; see \cite{P4}) of the differential $(n-2)$-form $\ast \phi_1\wedge\phi_\alpha\big|_N$ is given by $Q-({\rm ord}(\alpha)+1)=Q-i-1$. 

\begin{lemma}\label{417}
 Let  $x\in {\rm Int}(S\setminus C_S)$. Then \begin{eqnarray*}{\mathcal{B}}_2(t)&\leq&\nis(\partial
S\cap B_\varrho(x, t))+ \sum_{i=2}^k i\,{\bf c}_i\, t^{i-1} \sum_{j_i\in I\ci}\int_{\partial
S\cap B_\varrho(x, t)}|\nn\wedge X_{j_i} \LL\Vol| \end{eqnarray*}for every $t>0$; see Definition \ref{2iponhomnor1}. 
\end{lemma}

\begin{no}\label{nuota1}
For the sake of brevity, hereafter we shall adopt the notation introduced in Definition \ref{xazax}. In particular, we assume that $\xi=\nn$, where $\nn$ is the horizontal unit normal to $\overline{S}=S\cup\partial S$. We shall set  $\mu_{j_i}(\partial S\cap B):=\int_{\partial S\cap B}|\nn\wedge X_{j_i} \LL\Vol|$ for all $B\in\mathcal{B}or(\GG)$ and $$\mu(x, t):=\sum_{i=2}^k i\,{\bf c}_i\, t^{i-1} \sum_{j_i\in I\ci}\mu_{j_i}(\partial S\cap B_\varrho(x, t))\qquad \forall\,\, x\in {\rm Int}(S\setminus C_S)\quad \forall\,\,t>0.$$
\end{no}
\begin{proof}[Proof of Lemma \ref{417}] Let us set
\begin{eqnarray*} \widetilde{{\mathcal B}_2}(t)&:=&\int_{\partial
S\cap B_\varrho(x, t)}
 \left|\left\langle \frac{\left([Z_x]\vv-{\langle [Z_x]\vv ,\varpi\rangle}\nn\right)\ot}{\varrho_x}, \chi\right\rangle\right|\nis.\end{eqnarray*}Furthermore, let $x\in  {\rm Int}(S\setminus C_S)$ be fixed and let $f:\partial S\setminus C_{\partial S}\longrightarrow\R_+$ be defined as$$f(y):=\frac{1}{\varrho_x(y)}\left|\left\langle
\left( Z_x\ot(y)-\frac{\langle Z_x\op(y),\nu(y)\rangle}{|\PH\nu(y)|}\nn\ot(y) \right),\frac{\eta(y)}{|\P\ss\eta(y)|}\right\rangle\right|.$$Then \begin{eqnarray*}f  &\equiv&\frac{1}{\varrho_x}\left|\left\langle
 \left( Z_x -\frac{\langle Z_x\op,\nu\rangle}{|\PH\nu|}\nn\right)\ot  ,\left(\eta\ss+\chi\right)\right\rangle\right|\\&=&\frac{1}{\varrho_x}\left|\left\langle
 \left( [Z_x]\cc -\frac{\langle [Z_x]\cc\op,\nu\rangle}{|\PH\nu|}\nn+ \left([Z_x]\vv-\frac{\langle [Z_x]\vv ,\nu \rangle}{|\PH\nu |}\nn\right)\right)\ot  ,\left(\eta\ss+\chi\right)\right\rangle\right|\\&=&\frac{1}{\varrho_x}\left|\left\langle
 \left( [Z_x]\cc -\frac{\langle [Z_x]\cc\op,\nu\rangle}{|\PH\nu|}\nn+ \left([Z_x]\vv- {\langle [Z_x]\vv ,\varpi\rangle} \nn\right)\right)\ot  ,\left(\eta\ss+\chi\right)\right\rangle\right|\\&=&\frac{1}{\varrho_x}\left|\left\langle
 \left( [Z_x]\cc - {\langle [Z_x]\cc\op,\nn\rangle}\nn+ \left([Z_x]\vv-{\langle [Z_x]\vv ,\varpi\rangle}\nn\right)\right)\ot  ,\left(\eta\ss+\chi\right)\right\rangle\right|\\&=&\frac{1}{\varrho_x}\left|\left\langle
 \left( [Z\ss]_x +\left([Z_x]\vv-{\langle [Z_x]\vv ,\varpi\rangle}\nn\right)\right)\ot  ,\left(\eta\ss+\chi\right)\right\rangle\right|\\ &=&\frac{|\langle[Z\ss]_x,\eta\ss\rangle|}{\varrho_x} +  \left|
 \left\langle \frac{\left([Z_x]\vv-{\langle [Z_x]\vv ,\varpi\rangle}\nn\right)\ot}{\varrho_x}, \chi\right\rangle\right|\\&\leq& 1+\left|\left\langle \frac{\left([Z_x]\vv-{\langle [Z_x]\vv ,\varpi\rangle}\nn\right)\ot}{\varrho_x}, \chi\right\rangle\right|.\end{eqnarray*}So we have shown that 
$${\mathcal{B}}_2(t) \leq \nis(\partial
S\cap B_\varrho(x, t))+\widetilde{{\mathcal B}_2}(t)\qquad\forall\,\,t>0.$$
Note that $[Z_x]\vv=\sum_{i=2}^k  \sum_{j_i\in I\ci}\langle[Z_x]\vv, X_{j_i}\rangle X_{j_i}$. Using the assumptions on  $\varrho$  we   get that $\left|\frac{\langle[Z_x]\vv, X_{j_i}\rangle}{\varrho_x}\right|\leq \textbf{c}_i\varrho_x^{i-1}$; see Definition \ref{2iponhomnor1}. Therefore
\begin{eqnarray*} \widetilde{{\mathcal B}_2}(t)&\leq&\sum_{i=2}^k i\,{\bf c}_i\,t^{i-1}\sum_{j_i\in I\ci}\int_{\partial
S\cap B_\varrho(x, t)}
 \left|\left\langle {\left(X_{j_i}-{\langle X_{j_i} ,\varpi\rangle}\nn\right)\ot}, \chi\right\rangle\right|\nis\\&=&\sum_{i=2}^k i\,{\bf c}_i\,t^{i-1}\sum_{j_i\in I\ci}\int_{\partial
S\cap B_\varrho(x, t)}
 \left|\left\langle \left( |\PH \nu|X_{j_i}-\nu_{j_i}\nn\right) , \P\vs\eta\right\rangle\right|\sigma\rr^{n-2}. \end{eqnarray*} Now let $\underline{\tau}$ be an adapted moving frame to $S$; see  Definition \ref{movadafr}.  
Using this frame we  get that
\begin{eqnarray*}\int_{\partial
S\cap B_\varrho(x, t)}
 \left|\left\langle \left( |\PH \nu|X_{j_i}-\nu_{j_i}\nn\right) , \P\vs\eta\right\rangle\right|\sigma\rr^{n-2}&=&\int_{\partial
S\cap B_\varrho(x, t)}
 \left|\nu_1\eta_{j_i}-\nu_{j_i}\eta_1\right|\sigma\rr^{n-2}\\&=&\int_{\partial
S\cap B_\varrho(x, t)}
 \left|\tau_1\wedge \tau_{j_i}\LL\Vol\right|\\&=&\int_{\partial
S\cap B_\varrho(x, t)}|\nn\wedge X_{j_i} \LL\Vol|.\end{eqnarray*}The thesis easily follows.
\end{proof}

\subsection{Proof of  the Isoperimetric Inequality}\label{isopineq1}

By applying the results of Section \ref{perdindirindina} together with
Theorem \ref{rmonin} we get the following  version of the
monotonicity inequality:
\begin{corollario} \label{rmonin2}Let ${S}\subset\GG$ be a  
hypersurface of class $\cont^2$ with (piecewise) $\cont^1$ boundary $\partial S$.
Then, for every
$x\in {\rm Int}(S\setminus C_S)$  we have
\begin{eqnarray}\label{rmytn2}-\frac{d}{dt}\frac{\per({S}_t)}{t^{Q-1}}
\leq\frac{1}{t^{Q-1}} \left(\int_{S_t}|\MS|\,\per
+\nis(\partial S\cap B_\varrho(x, t))+ \mu(x, t) \right)
\end{eqnarray}for $\mathcal{L}^1$-a.e. $t>0$; see Notation \ref{nuota1}.
\end{corollario} 
\begin{proof}The proof  follows by applying Theorem
\ref{rmonin}, Lemma \ref{crux00} and Lemma \ref{417}. 
\end{proof} 

\begin{no}Let $x\in {\rm Int}(S\setminus C_S)$. Henceforth, we shall set
\[\mathcal{D}(t):= \int_{S_t}|\MS|\,\per
+\nis(\partial S\cap B_\varrho(x, t))+ \mu(x, t) \qquad \forall\,\,t>0.\]\end{no}

\begin{lemma}\label{lem}Let ${S}\subset\GG$ be a  
hypersurface of class $\cont^2$ with (piecewise) $\cont^1$ boundary $\partial S$.  Let $x\in {\rm Int}\,(S\setminus C_S)$ and let \begin{equation}r_0(x):= 
2\left(\frac{\per({S})}{k_\varrho(\nn(x))}\right)^{{1}/{Q-1}}.\end{equation} 
For every $\lambda\geq 2$ there exists $r\in ]0, r_0(x)]$
such that\begin{eqnarray*}\label{condlem}\per({S}_{\lambda r})\leq
\lambda^{Q-1}\,r_0(x)\,\mathcal{D}(r).\end{eqnarray*}\end{lemma}

Due to Remark \ref{boundonmetricfactor}, the number $r_0(x)>0$ can be globally estimated from above and below.

\begin{no}\label{klkl}We  set   ${\bf r}(S):=\sup_{x\in {\rm Int}(S\setminus C_S)} r_0(x)$.
\end{no}

\begin{proof}[Proof of Lemma \ref{lem}]Fix $r\in]0,r_0(x)]$ and note that $\per({S}_t)$
is a monotone non-decreasing function of $t$ on $]r,r_0(x)]$. We start from the identity
$$\per({S}_t)/t^{Q-1}=\left(\per({S}_t)-\per\left({S}_{r_0(x)}\right)\right)/t^{Q-1}+\per\left({S}_{r_0(x)}\right)/t^{Q-1}.$$The first addend is an  increasing  function of $t$, while
the second one is an absolutely continuous function of $t$.
Therefore, by integrating the differential inequality
\eqref{rmytn}, we get that

\begin{equation}\label{ppp1}\frac{\per({S}_{r})}{r^{Q-1}}
\leq\frac{\per\left({S}_{r_0(x)}\right)}{\left(r_0(x)\right)^{Q-1}}+\int_r^{r_0(x)}\mathcal{D}(t)\,{t^{-(Q-1)}}dt.\end{equation}Therefore

\begin{eqnarray*}\beta&:=&\sup_{r\in]0,r_0(x)]}\frac{\per({S}_{r})}{r^{Q-1}}
\leq\frac{\per\left({S}_{r_0(x)}\right)}{\left(r_0(x)\right)^{Q-1}}+
\int_0^{r_0(x)}\mathcal{D}(t)\,{t^{-(Q-1)}}dt.\end{eqnarray*} Now
we argue by contradiction. If the lemma is false, it follows that
for every $r\in]0,r_0(x)]$\begin{equation*}\per({S}_{\lambda
r})>\lambda^{Q-1}r_0(x)\,\mathcal{D}(t).\end{equation*}From the
last inequality we infer that
\begin{eqnarray*}\int_0^{r_0(x)}\mathcal{D}(t)
\,{t^{-(Q-1)}}dt&\leq&\frac{1}{\lambda^{Q-1}r_0(x)}\int_{0}^{r_0(x)}\per({S}_{\lambda
t})\,t^{-(Q-1)}dt\\&=&\frac{1}{\lambda\,r_0(x)}\int_{0}^{\lambda
r_0(x)}\per({S}_{s})\,s^{-(Q-1)}ds\\&=&\frac{1}{\lambda\,
r_0(x)}\int_{0}^{r_0(x)}\per({S}_{s})\,s^{-(Q-1)}ds+
\frac{1}{\lambda r_0(x)}\int_{r_0(x)}^{\lambda
r_0(x)}\per({S}_{s})\,s^{-(Q-1)}ds\\&\leq
&\frac{\beta}{\lambda}+\frac{\lambda-1}{\lambda}\frac{\per({S})}{\left(r_0(x)\right)^{Q-1}}.\end{eqnarray*}Therefore,
using \eqref{ppp1} yields
$$\beta\leq\frac{\per\left({S}_{r_0(x)}\right)}{\left(r_0(x)\right)^{Q-1}}+
 \frac{\beta}{\lambda}+\frac{\lambda-1}{\lambda}\frac{\per({S})}{\left(r_0(x)\right)^{Q-1}}$$and so
$$\frac{\lambda-1}{\lambda}\beta\leq\frac{2\lambda-1}{\lambda}\left(\frac{\per({S})}{\left(r_0(x)\right)^{Q-1}}\right)
\leq\frac{2\lambda-1}
{\lambda}\left(\frac{k_\varrho(\nn(x))}{2^{Q-1}}\right).$$By its
own definition, one has
$$k_\varrho(\nn(x))=\lim_{r\searrow
0^+}\frac{\per({S}_r)}{r^{Q-1}}\leq\beta.$$Furthermore,
since\footnote{Indeed, the first non-abelian Carnot group is the
Heisenberg group $\mathbb{H}^1$ for which $Q=4$.} $Q-1\geq 3$, we get
that $\lambda-1\leq \frac{2\lambda-1}{8},$ or equivalently
$\lambda\leq\frac{7}{6}$, which contradicts the hypothesis
$\lambda\geq 2$.

\end{proof}

 The next covering lemma is well-known and can be found in \cite{BuZa}; see also \cite{FE}.
\begin{lemma}[Vitali's Covering Lemma]\label{cov}Let $(X,\varrho)$ be a compact metric space and
let $A\subseteq X$. Moreover, let ${\textsf{C}}$ be a covering of
$A$ by closed $\varrho$-balls with centers in $A$. We also assume
that each point $x$ of $A$ is the center of at least one closed
$\varrho$-ball belonging to $\textsf{C}$ and that the radii of the
balls of the covering ${\textsf{C}}$ are uniformly bounded by some
positive constant. Then, for every $\lambda> 2$ there exists a no
more than countable subset
${\textsf{C}}_\lambda\subsetneq{\textsf{C}}$ of pairwise
non-intersecting closed balls
$\overline{B}_\varrho(x_l,r_l),\,l\in \mathbb{N},$ such that
$$A\subset\bigcup_{l\in \mathbb{N}}{B}_\varrho(x_l,\lambda\,r_l).$$\end{lemma}

We are now in a position to prove our main result.

\begin{proof}[Proof of Theorem \ref{ghaioio}]We  apply Lemma \ref{lem}. So  let
$\lambda>2$ and, for every $x\in{\rm Int}({S}\setminus
C_{{S}})$, let $r(x)\in]0, {\bf r}(S)]$ be such
that\begin{equation}\label{fifiin}\per({S}_{r(x)})\leq\lambda^{Q-1}\,{\bf r}(S)\,
 \mathcal{D}(r(x)).\end{equation}Let
 $\textsf{C}=\left\{\overline{B_\varrho}(x,r(x)): x\in
 {\rm Int}(S\setminus C_S)\right\}$ be a covering of  
 ${S}$. By Lemma \ref{cov}, there
 exists a non more than countable subset $\textsf{C}_\lambda\subsetneq\textsf{C}$ of pairwise
non-intersecting closed balls $\overline{B}_\varrho(x_k,r_k)$, $k\in \mathbb{N}$, such that
$${S}\setminus
C_{S}\subset\bigcup_{l\in\mathbb{N}}{B}_\varrho(x_l,\lambda\,r_l),$$
where we have set $r_l:=r(x_l)$. We
therefore get
\begin{eqnarray*}\per({S})&\leq&\sum_{l\in\mathbb{N}}\per({S}\cap
B_\varrho(x_l,\lambda\,r_l))\\&\leq&
\lambda^{Q-1}\,{\bf r}(S)\sum_{l\in\mathbb{N}}\,\mathcal{D}(r_l)\qquad\,\,\,
\mbox{(by \eqref{fifiin})}
\\&=&\lambda^{Q-1}\,{\bf r}(S)\,\sum_{l\in\mathbb{N}}  \left(\int_{S_{r_l}}|\MS|\,\per
+\nis(\partial S\cap B_\varrho(x_l, {r_l})) + \mu(x_l, r_l)  \right) \\
\\&\leq&\lambda^{Q-1}\,{\bf r}(S)\, \left( \int_S |\MS|\,\per +\nis(\partial S) +\sum_{l\in\mathbb{N}}  \mu(x_l, r_l)\right). 
\end{eqnarray*} 
By letting $\lambda\searrow 2$, we get that
\begin{eqnarray*}\per({S})\leq
2^{Q-1}{\bf r}(S)\,\left(\int_S |\MS|\,\per + \nis(\partial S)+ \sum_{l\in\mathbb{N}}  \mu(x_l, r_l)\right).\end{eqnarray*}
Since$$2^{Q-1}{\bf r}(S)\leq 2^{Q-1}\sup_{x\in {\rm Int}(S\setminus C_S)}2
\left(\frac{\per({S})}{k_\varrho(\nn(x))}\right)^{\frac{1}{Q-1}}=2^Q\sup_{x\in
{\rm Int}(S\setminus C_S)}\frac{\left(\per({S})\right)^{\frac{1}{Q-1}}}{\left(k_\varrho(\nn(x))\right)^{\frac{1}{Q-1}}},$$
using  \eqref{emfac} yields
$$2^{Q-1}{\bf r}(S)\leq 2^Q\,
\frac{\left(\per({S})\right)^{\frac{1}{Q-1}}}{K_1^{\frac{1}{Q-1}}};$$ see Remark \ref{boundonmetricfactor}.  Therefore\begin{eqnarray}\label{finineqf}\left(\per({S})\right)^{\frac{Q-2}{Q-1}}\leq
C_1\left(\int_S
|\MS|\,\per +\nis(\partial S)+\sum_{l\in\mathbb{N}}  \mu(x_l, r_l)\right) \end{eqnarray}where we have set $C_{1}:={2^{Q}}/{K_1^{\frac{1}{Q-1}}}.$ 
Furthermore, we have
 \begin{eqnarray*}\sum_{l\in\mathbb{N}}  \mu(x_l, r_l)&=&\sum_{l\in\mathbb{N}}\sum_{i=2}^k i\,{\bf c}_i\,r_l^{i-1}\sum_{j_i\in I\ci}\int_{\partial
S\cap B_\varrho(x_l, r_l)}|\nn\wedge X_{j_i} \LL\Vol|\\&=&\sum_{l\in\mathbb{N}}\sum_{i=2}^k i\,{\bf c}_i\,r_l^{i-1}\sum_{j_i\in I\ci}\mu_{j_i}(\partial
S\cap B_\varrho(x_l, r_l))\end{eqnarray*}where we have
used Definition \ref{xazax} with $\xi=\nn$; see Notation \ref{nuota1}. Then $$\sum_{l\in\mathbb{N}}  \mu(x_l, r_l)\leq  \sum_{i=2}^k i\,{\bf c}_i\, \left( {\bf r}(S) \right)^{i-1}\sum_{j_i\in I\ci}\mu_{j_i}(\partial
S);$$see Notation \ref{klkl}. Furthermore, since $$\mu_{j_i}(\partial S)=\int_{\partial
S }|\nn\wedge X_{j_i} \LL\Vol|\leq \int_{\partial
S }|\P\ciss\eta|\sigma\rr^{n-2},$$ we have
$$\sum_{j_i\in I\ci}\mu_{j_i}(\partial S)\leq h_i \int_{\partial
S }|\P\ciss\eta|\,\sigma\rr^{n-2},$$where $h_i=\dim \HH_i$. 
Therefore, we conclude that
 \begin{eqnarray*}\sum_{l\in\mathbb{N}}  \mu(x_l, r_l)&\leq& \sum_{i=2}^k i\,{\bf c}_i\left( {\bf r}(S) \right)^{i-1}\sum_{j_i\in I\ci}\mu_{j_i}(\partial S)\\&\leq & \sum_{i=2}^k i\,{\bf c}_i\,h_i\left( {\bf r}(S)\right)^{i-1} \int_{\partial
S }|\P\ciss\eta|\,\sigma\rr^{n-2}\\&\leq & C_{2}\sum_{i=2}^k  \left( {\bf r}(S)\right)^{i-1} \int_{\partial
S }|\P\ciss\eta|\,\sigma\rr^{n-2},\end{eqnarray*}where we have set $C_2:=\max_{i=2}^k i\,{\bf c}_i\,h_i$.
Using the last inequality it follows that\begin{eqnarray}\label{finineqf222}\left(\per({S})\right)^{\frac{Q-2}{Q-1}}\leq
C_{Isop}\left(\int_S
|\MS|\,\per +\nis(\partial S)+ \sum_{i=2}^k  \left( {\bf r}(S)\right)^{i-1} \int_{\partial
S }|\P\ciss\eta|\,\sigma\rr^{n-2} \right) \end{eqnarray}where we have set $C_{Isop}:=\max\left\lbrace C_1, C_2 \right\rbrace$. This achieves the proof.
\end{proof}

\subsection{An application of the monotonicity formula: asymptotic
behavior of $\per$}\label{asintper}
 
 The monotonicity formula
\eqref{rmytn} (see Theorem \ref{rmonin}) can be formulated as
follows:
\begin{eqnarray}\label{zoccowa}\frac{d}{dt}\left(\frac{\per({S}_t)}{t^{Q-1}}\,\exsp
\left(\int_0^t\frac{\mathcal{A}(s)+\mathcal{B}_2(s)}{\per({S}_s)}ds\right)\right)\geq
0\end{eqnarray} for $\mathcal{L}^1$-a.e. $t\in[0, r_S]$ and for
every $x\in {\rm Int}\,(S\setminus C_S)$. For the sake of simplicity, let $\partial S=\emptyset$ (and hence
 $\mathcal{B}_2(s)=0$). By Theorem
\ref{BUP}, Case (i), we may pass to the limit as $t\searrow 0^+$
in the previous inequality; see Section \ref{blow-up}. Hence
\begin{eqnarray}\label{zoccona}\per({S}_t)\geq
\kappa_{\varrho}(\nn(x))\,t^{Q-1}\exsp
\left(-\int_0^t\frac{\mathcal{A}(s)}{\per({S}_s)}ds\right),\end{eqnarray}for
every $x\in{\rm Int}({S}\setminus C_S)$.

\begin{corollario}\label{asynt}Let $\GG$ be a $k$-step Carnot group and let
${S}\subset\GG$ be a  hypersurface of class $\cont^2$ without boundary. Assume  
that $|\MS|\leq\MS^0<+\infty$. Then, for every $x\in{\rm
Int}({S}\setminus C_S)$, one has
\begin{eqnarray}\label{ak}
\per({S}_t)\geq \kappa_{\varrho}(\nn(x))\,t^{Q-1} e^{-t\,
 \MS^0}\end{eqnarray}as long as $t
\rightarrow 0^+$.
\end{corollario}
\begin{proof}We just have to bound $\int_0^t
\frac{\mathcal{A}(s)}{\per({S}_s)}\,ds$ from above. Using Lemma
\ref{crux00} yields\[\int_0^t
\frac{\mathcal{A}(s)}{\per({S}_s)}\,ds\leq
\MS^0\left(1+o(1)\right)\] as long as $t\rightarrow 0^+$ and
\eqref{ak} follows from \eqref{zoccona}.
 \end{proof}

\begin{war}
If $S$ is smooth enough near its characteristic set $C_S$ and $\MS$ is globally bounded, the
previous asymptotic estimate can be generalized by applying the results of
Section \ref{perdindirindina}. In the following corollaries, however, \underline{we need to assume that condition \eqref{0key2} holds} at the point where  the monotonicity inequality has to be proved.
\end{war}

\begin{corollario}\label{2asynt}Let $\GG$ be a $k$-step Carnot group. Let
${S}\subset\GG$ be a  hypersurface without boundary and such that $|\MS|\leq\MS^0<+\infty$. Let
$x\in {\rm Int}(S\cap C_S)$ and    assume that ${\rm ord}(x)=Q-i$, for
some $i=2,..., k$. 
 With no loss of generality, we suppose that
 there
exists $\alpha\in\{\DH+1,...,n\}$, ${\rm ord}(\alpha)=i$, such that $S$
can be represented, locally around $x$, as  $X_\alpha$-graph of
a $\cont^i$ function satisfying \eqref{0dercond}. We further assume that condition \eqref{0key2} holds at the point $x$. Then
\begin{eqnarray}\label{2asintotica}
\per({S}_t)\geq \kappa_{\varrho}(C_S(x))\,t^{Q-1} e^{-t\, \MS^0\left(\kappa_\varrho(C_S(x))+
{\bf d}_\varrho\right)}\end{eqnarray}as long as $t \rightarrow
0^+$.
\end{corollario}We recall that the function $\kappa_{\varrho}(C_S(x))$ has been defined in
Theorem \ref{BUP}; see Case (ii). Notice that
${\bf d}_\varrho={\sum_{i=2}^k i\,{\bf c}_i\DH_i\,{\bf b}_\varrho}$, where
${\bf b}_\varrho$ is the constant
defined by \eqref{tonda3}.
\begin{proof}By arguing as above, we may pass to the
limit in \eqref{zoccowa}
 as $t\searrow 0^+$  and we  get that
\begin{eqnarray*}\per({S}_t)\geq
\kappa_{\varrho}(C_S(x))\,t^{Q-1}\exsp
\left(-\int_0^t\frac{\mathcal{A}(s)}{\per({S}_s)}ds\right).\end{eqnarray*}
By using Lemma \ref{cux0}, we get that
\begin{eqnarray*}\int_0^t\frac{\mathcal{A}(s)}{\perh({S}_s)}ds\leq
 \|\MS\|_{L^\infty(S)}\,\left(\kappa_\varrho(C_S(x))+
{\bf d}_\varrho\right)\,t\leq \MS^0\left(\kappa_\varrho(C_S(x))+
{\bf d}_\varrho\right)\,t
\end{eqnarray*}as long as $t\rightarrow
0^+$. This achieves the proof. \end{proof}

 In particular, in the case of   Heisenberg
 groups
$\mathbb{H}^r$, the following holds:
\begin{corollario}\label{hasynt}Let
$(\mathbb{H}^r, \varrho)$ be the  Heisenberg group endowed with
the Koranyi distance; see Example \ref{Kor}. Let
$S\subset\mathbb{H}^r$ be a hypersurface of class $\cont^2$ without boundary and assume
that $|\MS|\leq\MS^0<+\infty$. Furthermore, let $x\in{\rm Int} (S \cap C_S)$ be such that condition \eqref{0key2} holds. Then 
\begin{eqnarray}\label{2asintotica}
\perh({S}_t)\geq \kappa_{\varrho}(C_S(x))\,t^{Q-1} e^{-t\, \MS^0
\left(\kappa_\varrho(C_S(x))+{\bf b}_\varrho\right)}\end{eqnarray}as long as $t \rightarrow
0^+$.\end{corollario}The density-function $\kappa_{\varrho}(C_S(x))$ has
been defined in Theorem \ref{BUP}; see Case (ii).
Moreover,  ${\bf b}_\varrho$ is the constant defined by
\eqref{tonda3}.
\begin{proof}By arguing as for the non-characteristic case, we may pass to the
limit in \eqref{zoccowa}
 as $t\searrow 0^+$. As above, we have
\begin{eqnarray*}\perh({S}_t)\geq
\kappa_{\varrho}(C_S(x))\,t^{Q-1}\exsp
\left(-\int_0^t\frac{\mathcal{A}(s)}{\perh({S}_s)}ds\right),\end{eqnarray*}as
$t\searrow 0^+$, for every $x\in {S}\cap C_S$. By applying Lemma
\ref{perdopo} 
\begin{eqnarray*}\frac{\mathcal{A}(s)}{\perh({S}_s)}\leq
\MS^0\left(\kappa_\varrho(C_S(x))+ 2\,{\bf c}_2\,{\bf b}_\varrho \right)=\MS^0\left(\kappa_\varrho(C_S(x))+
{\bf b}_\varrho\right),
\end{eqnarray*}for every (small
enough) $s>0$, since in this case ${\bf c}_2=\frac{1}{2}$.
\end{proof}
\begin{es} Let $(\mathbb{H}^r, \varrho)$, where
$\varrho$ is the Koranyi distance and $Q=2r +2$. Let
$$S=\{\exp(x\cc, t)\in \mathbb{H}^r : t=0\}.$$We have
$C_S=0\in\mathbb{H}^r $ and
$\nn=-\frac{1}{2}C^{2r+1}\cc x\cc$. Furthermore
\[-\MS=\div\cc \nn=\frac{1}{2}\div_{\R^{2r}}(-x_2, x_1,...,-x_{2r},x_{2r-1})=0.\]Note that
$\kappa_\varrho(C_S)=\frac{O_{2r}}{4r}$, where $\mathit{O}_{2r-1}$
is the surface measure of the unit sphere
$\mathbb{S}^{2r-1}\subset\R^{2r}$. Thus \eqref{2asintotica} says
that $\perh({S}_t)\geq \frac{O_{2r}}{4r}\,t^{Q-1}$. This inequality can also be proved by using the formula $\perh=\frac{|x\cc|}{2}\,d
\mathcal{L}^{2r}$ and then by introducing spherical coordinates on $\R^{2r}$.
\end{es}

\section{ Sobolev-type inequalities on hypersurfaces}\label{sobineqg}

The isoperimetric inequality \eqref{2gha} turns out to be equivalent to
a Sobolev-type inequality. The proof is analogous to that of the
equivalence between the (Euclidean) Isoperimetric Inequality and the
Sobolev one; see \cite{BuZa}. Below we shall assume that $S$ is a compact $\cont^2$ hypersurface without boundary. 
\begin{teo}\label{sobolev2}Let $\GG$ be a $k$-step Carnot group endowed with a homogeneous metric $\varrho$ as in
Definition \ref{2iponhomnor1}. Let
$S\subset\GG$ be a  compact hypersurface of class $\cont^2$  without
boundary. Let $\MS$ be the horizontal mean curvature of
$S$ and assume that $\MS\in L^1(S; \sigma\rr^{n-1})$.
Then\begin{eqnarray}\label{sersobolev1}\left(\int_S|\psi|^{\frac{Q-1}{Q-2}}\,\per\right)^{\frac{Q-2}{Q-1}}\leq
C_{Isop}\left(\int_{S}\left(|\psi|\,|\MS|
+|\qq\psi|\right)\,\per+ \sum_{i=2}^k  \left( {\bf r}(S)\right)^{i-1}   \int_S |\grad\ciss\psi|\,\sigma\rr^{n-1}\right) \end{eqnarray}for every
$\psi\in\cont^1(S)$, where $C_{Isop}$ is the same constant
appearing in Theorem \ref{ghaioio}.
\end{teo}
\begin{proof}The proof follows a classical argument;
see \cite{FedererFleming}, \cite{MAZ}. Since
$|\qq\psi|\leq|\qq|\psi||$, without loss of generality we 
assume $\psi\geq 0$. Set
$S_t:=\{x\in S: \psi(x)>t\}$. The set $S_t$ is a bounded open subset of $S$ and, by applying
Sard's Lemma, we see that its boundary $\partial S_t$ is $\cont^1$ 
for $\mathcal{L}^1$-a.e. $t\geq 0$. Furthermore, $S_t=\emptyset$
for each (large enough) $t>0$. The main tools  
are  {\it Cavalieri's
principle}\footnote{\label{CavPrin}The following lemma, also known
as {\it Cavalieri's principle}, is a simple consequence of Fubini's
Theorem:
\begin{lemma}\label{CPrin}Let X be an abstract space, $\mu$ a measure on $X$,
$\alpha>0$, $\varphi\geq 0$ and $A_t=\{x\in X: \varphi>t\}$. Then
$$\int_0^{+\infty}t^{\alpha-1}\mu(A_t)\,dt=\frac{1}{\alpha}\int_{A_0}\varphi^\alpha\,d\mu.$$\end{lemma}}
 and the Riemannian Coarea Formula; see \cite{BuZa}, \cite{Ch2}, \cite{FE}. We start by the
 identity
\begin{equation}\label{concavp}\int_S|\psi|^{\frac{Q-1}{Q-2}}
\per=\frac{Q-1}{Q-2}\int_0^{+\infty}t^{\frac{1}{Q-2}}\,\per(S_t)\,dt\end{equation}
which follows from Lemma \ref{CPrin} with
$\alpha=\frac{Q-1}{Q-2}$. We also recall that, if
$\varphi:\R_+\longrightarrow\R_+$ is a positive {\it decreasing}
function and $\alpha\geq 1$,
then$$\alpha\int_0^{+\infty}t^{\alpha-1}\varphi^\alpha\,dt\leq\left(\int_0^{+\infty}\varphi(t)\,dt\right)^\alpha.$$
Using \eqref{concavp} and the last inequality yields
\begin{eqnarray*}&&\int_S\psi^{\frac{Q-1}{Q-2}}\,\per\\&=&\frac{Q-1}{Q-2}
\int_0^{+\infty}t^{\frac{1}{Q-2}}\,\per(S_t)\,dt \\&\leq&
\left[\int_0^{+\infty}\left(\per(S_t)\right)^{\frac{Q-2}{Q-1}}\,dt\right]^{\frac{Q-1}{Q-2}}\\&\leq&
\left[\int_0^{+\infty}C_{Isop}\,\left(\int_{S_t} |\MS|\,\per
+\nis(\partial S_t)+ \sum_{i=2}^k  \left( {\bf r}(S_t)\right)^{i-1} \int_{\partial
S_t}|\P\ciss \eta|\,\sigma\rr^{n-2} 
\right)\,dt\right]^{\frac{Q-1}{Q-2}}\\\\&\leq& \left[C_{Isop}\left(\int_{S}\left(|\psi|\,|\MS|
+|\qq\psi|\right)\,\per+ \sum_{i=2}^k  \left( {\bf r}(S)\right)^{i-1}   \int_S |\grad\ciss\psi|\,\sigma\rr^{n-1}\right)\right]^{\frac{Q-1}{Q-2}},\end{eqnarray*}where
we have used  \eqref{2gha}
with $S=S_t$ and the (Riemannian) Coarea formula together with the obvious estimate ${\bf r}(S_t)\leq {\bf r}(S)$.
 \end{proof}

\begin{no}For any $p>0$,  set
 $\frac{1}{p^\ast}=\frac{1}{p}-\frac{1}{Q-1}$. Furthermore, we
 denote by $p'$ the H\"{o}lder conjugate of $p$, that is,  $\frac{1}{p}+\frac{1}{p'}=1$.  \end{no}
 Henceforth, we shall assume that $\MS$ is globally
bounded on $S$ and  set
$\MS^0:=\|\MS\|_{L^\infty(S)}$.

\begin{corollario}\label{corsob1}\small Under the  assumptions of Theorem \ref{sobolev2},
one has
$$\|\psi\|_{L^{p^\ast}(S)}\leq C_{Isop}\left[\MS^0 \|\psi\|_{L^p\left(S, \per\right)}+ c_{p^{\ast}}\left( \|\qq\psi\|_{L^p\left(S, \per\right)}  +\sum_{i=2}^k  \left( {\bf r}(S)\right)^{i-1} \|\grad\ciss\psi\|_{L^p\left(S, \sigma\rr^{n-1}\right)}\right) \right]$$ for every
$\psi\in  \cont^1(S)$, where
$c_{p^{\ast}}:=p^\ast\frac{Q-2}{Q-1}$.  Thus, there exists
$C_{p^{\ast}}=C_{p^{\ast}}(\MS^0, {\bf r}(S), \varrho, \GG)$ such
that$$\|\psi\|_{L^{p^\ast}\left(S, \per\right)}\leq C_{p^{\ast}}\left(
\|\psi\|_{L^p\left(S, \per\right)}+ \|\qq\psi\|_{L^p\left(S, \per\right)}+  \|\grad\vs\psi\|_{L^p\left(S, \sigma\rr^{n-1}\right)}\right)$$ for every
$\psi\in  \cont^1(S)$.
\end{corollario}

\begin{proof}Let us apply \eqref{sersobolev1} with
$\psi$ replaced by $\psi|\psi|^{t-1}$, for some $t>0$. It follows
that
\begin{eqnarray}\nonumber\left(\int_S|\psi|^{t\,\frac{Q-1}{Q-2}}\,\per\right)^{\frac{Q-2}{Q-1}}\leq
C_{Isop}\left[ \int_{S}\left(\MS^0\,|\psi|^t+ t|\psi|^{t-1}
|\qq\psi|\right)\per\right.\\\label{wfiha}\left.  +\sum_{i=2}^k  \left( {\bf r}(S)\right)^{i-1} \int_{S} t|\psi|^{t-1}
|\grad\ciss\psi|\,\sigma\rr^{n-1} \right] .
\end{eqnarray}If we put $(t-1)p'=p^\ast$,  one gets $p^\ast=t\,\frac{Q-1}{Q-2}$. Using H\"{o}lder inequality
yields
\begin{eqnarray*}\left(\int_S|\psi|^{p^\ast}\per\right)^{\frac{Q-2}{Q-1}}\leq
C_{Isop}\left(\int_S|\psi|^{p\ast}\per\right)^{\frac{1}{p'}}\times\\\times\left(\MS^0\,\|\psi\|_{L^p\left(S, \per\right)}
+ t\,\|\qq\psi\|_{L^p\left(S, \per\right)} + \sum_{i=2}^k  \left( {\bf r}(S)\right)^{i-1}  t\,\|\grad\ciss\psi\|_{L^p\left(S, \sigma\rr^{n-1}\right)}\right).
\end{eqnarray*}
\end{proof}

\begin{corollario}\label{corsob12}Under the  assumptions of Theorem \ref{sobolev2}, let
$p\in[1, Q-1[$. For all $q\in[p, p^\ast]$ one has
\begin{eqnarray*}\|\psi\|_{L^{q}\left(S, \per\right)}\leq \left(1+\MS^0\,C_{Isop}\right)
\|\psi\|_{L^p\left(S, \per\right)}\\+c_{p^\ast}\,C_{Isop}\left( \|\qq\psi\|_{L^p\left(S, \per\right)}+\sum_{i=2}^k  \left( {\bf r}(S)\right)^{i-1} \|\grad\ciss\psi\|_{L^p\left(S, \sigma\rr^{n-1}\right)}\right)\end{eqnarray*}for
every $\psi\in  \cont^1(S)$. In particular, there exists
$C_q={C}_q(\MS^0,  {\bf r}(S), \varrho, \GG)$ such that
$$\|\psi\|_{L^{q}\left(S, \per\right)}\leq C_q\left(
\|\psi\|_{L^p\left(S, \per\right)}+ \|\qq\psi\|_{L^p\left(S, \per\right)}+  \|\grad\vs\psi\|_{L^p\left(S, \sigma\rr^{n-1}\right)}\right)$$ for every
$\psi\in  \cont^1(S)$.
\end{corollario}

\begin{proof}For any given $q\in[p, p^\ast]$ there exists $\alpha\in[0,1]$ such that
 $\frac{1}{q}=\frac{\alpha}{p}+\frac{1-\alpha}{p^\ast}.$ Hence$$\|\psi\|_{L^q\left(S, \per\right)}\leq
 \|\psi\|^\alpha_{L^p\left(S, \per\right)}\|\psi\|^{1-\alpha}_{L^{p^\ast}\left(S, \per\right)}
 \leq\|\psi\|_{L^p\left(S, \per\right)}+\|\psi\|_{L^{p^\ast}\left(S, \per\right)},$$where we have
 used  {\it interpolation inequality} and Young's inequality.
 The thesis follows from Corollary \ref{corsob1}.\end{proof}

\begin{corollario}[Limit case: $p=Q-1$]\label{corsob13}Under the  assumptions of Theorem \ref{sobolev2}, let
$p=Q-1$. For every $q\in[Q-1, +\infty[$ there exists
$C_q={C}_q(\MS^0, {\bf r}(S), \varrho, \GG)$ such that
$$\|\psi\|_{L^{q}\left(S, \per\right)}\leq C_q
\left(\|\psi\|_{L^p\left(S, \per\right)}+\|\qq\psi\|_{L^p\left(S, \per\right)}+  \|\grad\vs\psi\|_{L^p\left(S, \sigma\rr^{n-1}\right)}\right)$$ for every
$\psi\in  \cont^1(S)$.
\end{corollario}

\begin{proof}By using  \eqref{wfiha} we easily get that there exists
${C}_1={C}_1(\MS^0, {\bf r}(S), t, \varrho, \GG)>0$ such that
\begin{eqnarray*}\left(\int_S|\psi|^{t\,\frac{Q-1}{Q-2}}\per\right)^{\frac{Q-2}{Q-1}}&\leq&
{C}_1\left[ \int_{S}\left(|\psi|^t+ |\psi|^{t-1} |\qq\psi|\right)\,\per+ \int_S|\psi|^{t-1} |\grad\vs\psi|\,\sigma\rr^{n-1}\right] 
\end{eqnarray*}for every $\psi\in  \cont^1(S)$. From now on we assume that $t\geq 1$. Using H\"{o}lder
inequality with $p=Q-1$, yields
\begin{eqnarray*}\|\psi\|_{L^{t\,\frac{Q-1}{Q-2}}\left(S, \per\right)}^t \leq 
{C}_1\left[\|\psi\|_{{L^t}\left(S, \per\right)}^t +
\|\psi\|_{L^{\frac{(t-1)(Q-1)}{Q-2}}\!\left(S, \per\right)}^{t-1}\right.\times \\\left. \times\left( \|\qq\psi\|_{L^{Q-1}\left(S, \per\right)}+ \|\grad\vs\psi\|_{L^{Q-1}\left(S, \sigma\rr^{n-1}\right)}\right) \right]
\end{eqnarray*}for every $\psi\in  \cont^1(S)$ and $t\geq 1$. By means of Young's 
inequality, we get that there exists another constant
${C_2}={C_2}(\MS^0, {\bf r}(S), t, \varrho, \GG)$ such that
\begin{eqnarray*}\|\psi\|_{L^{t\,\frac{Q-1}{Q-2}}\left(S, \per\right)} \leq 
{C_2}\left(\|\psi\|_{L^t\left(S, \per\right)}+
\|\psi\|_{L^{\frac{(t-1)(Q-1)}{Q-2}}\!\left(S, \per\right)}+\right.\\\left.+
\|\qq\psi\|_{L^{Q-1}\left(S, \per\right)}+ \|\grad\vs\psi\|_{L^{Q-1}\left(S,\sigma\rr^{n-1}\right)}\right).
\end{eqnarray*}By setting $t=Q-1$ in the last inequality we get
that
\begin{eqnarray*}\|\psi\|_{L^{\frac{(Q-1)^2}{Q-2}}\!\left(S, \per\right)}\leq
{C_2}\left(2\|\psi\|_{L^{Q-1}\left(S, \per\right)}+ \|\qq\psi\|_{L^{Q-1}\left(S, \per\right)}+ \|\grad\vs\psi\|_{L^{Q-1}\left(S, \sigma\rr^{n-1}\right)}\right).
\end{eqnarray*}By reiterating this procedure for $t=Q, Q+1,...$
one can show that for all $q\geq Q-1$ there exists
$C_q={C}_q(\MS^0, {\bf r}(S), \varrho, \GG)$ such
that$$\|\psi\|_{L^q\left(S, \per\right)}\leq C_q\left(\|\psi\|_{L^{Q-1}\left(S, \per\right)}+
\|\qq\psi\|_{L^{Q-1}\left(S, \per\right)}+ \|\grad\vs\psi\|_{L^{Q-1}\left(S, \sigma\rr^{n-1}\right)} \right)$$for every
$\psi\in  \cont^1(S)$, as wished.
\end{proof}

{\footnotesize \noindent Francescopaolo Montefalcone:\\
Dipartimento di Matematica\\
Universit\`{a} degli Studi di Padova\\
Address: Via Trieste, 63,\,\,
35121 Padova (Italy)
 \\ {\it E-mail}: {\textsf montefal@math.unipd.it}}}


\begin{thebibliography}{99}{\tiny
\bibitem {Allard}{\sc W.K.~Allard},
{\em On the first variation of a varifold}, Ann. Math.  {95},
417-491, 1972.






\bibitem {A2}{\sc L.~Ambrosio},
{\em Some fine properties of sets of finite perimeter in Ahlfors
regular metric measure spaces}, Adv. in Math., 159 (1) (2001),
51-67.



\bibitem{AK1}{\sc L.~Ambrosio,\,B.~Kircheim}, {\em Current in metric spaces}, Acta Math.{ 185} (2000) 1-80.



\bibitem{Anze} {\sc G.~Anzellotti},
{\em Pairings between measures and bounded functions and compensated compactness},  Ann. Mat. Pura Appl., IV. Ser. 135, 293-318 (1983). 


\bibitem{balogh} {\sc Z.M.~Balogh}, {\em Size of characteristic sets and functions
with prescribed gradient,} J. Reine Angew. Math.  564
 (2003) 63--83.


\bibitem{BTW} {\sc Z.M.~Balogh,\, J.T.~Tyson,\, B.~Warhurst},
{\em Sub-Riemannian vs. Euclidean dimension comparison and 
fractal geometry on Carnot groups,} Adv. Math.  220, no. 2 (2009)
560--619.



\bibitem{Bal3} {\sc Z.M.~Balogh,\, C.~Pintea,\, H.~Rohner},
{\em Size of tangencies to non-involutive distributions,} Preprint
2010.


\bibitem{BuZa} {\sc Yu.D.~Burago,\,V.A.~Zalgaller}, {\em Geometric Inequalities,}
Grundlenheren der mathematischen Wissenshaften {285}, Springer
Verlag (1988).


\bibitem{CDG}{\sc L.~Capogna,\,D.~Danielli,\,N.~Garofalo},
{\em The geometric Sobolev embedding for vector fields and the
isoperimetric inequality},  Comm. Anal. Geom. {{12}} (1994).


\bibitem{CCM}{\sc L.~Capogna,\,G.~Citti,\,M.~Manfredini},
{\em Smoothness  of  Lipschitz  minimal intrinsic graphs in
Heisenberg groups $H^n$, $n>1$}, Indiana University Mathematics
Journal (2008).



\bibitem{Ch1}{\sc I.~Chavel},
{\em ``Riemannian Geometry: a modern introduction''}, Cambridge
University Press (1994).\bibitem{Ch2}
\leavevmode\vrule height 2pt depth -1.6pt width 23pt, {\em
Isoperimetric Inequalities}, Cambridge University Press, 2001.


\bibitem {Che} {\sc J.~Cheeger}, {\em
Differentiability of Lipschitz functions on metric measure
spaces}, Geom.Funct.An., {{9}}, (1999) 428-517.



\bibitem {Cheeger1} {\sc J.~Cheeger,\, B.~Kleiner}, {\em Differentiating maps into L1 and the geometry of BV functions}, Ann. of Math., Accepted Paper (2008).


\bibitem{variZ}{\sc G.Q~Chen,\, M.~Torres, \, W.P.~Ziemer}, {\em Gauss-Green theorem for weakly differentiable vector fields, sets of finite perimeter, and balance laws,} Commun. Pure Appl. Math. 62, No. 2, 242-304 (2009).




\bibitem {vari} {\sc J.J~Cheng,\, J.F.~Hwang,\, A.~Malchiodi,\, P.~Yang}, {\em
Minimal surfaces in pseudohermitian geometry,} {Ann. Sc. Norm.
Sup. Pisa Cl. Sci.}, { 5}, IV (2005) 129-179.


\bibitem {CGY} {\sc F.~Chung,\,A.~Grigor'yan,\, S.T.~Yau}, {\em
Higher eigenvalues and isoperimetric inequalities on Riemannian
manifolds and graphs,} {Comm. Anal. Geom.}, {8}, no.5 (2000)
969-1026.




\bibitem {CMS} {\sc G.~Citti,\, M.~Manfredini,\,A.~Sarti}, {\em
Neuronal oscillation in the visual cortex: Gamma-convergence to
the  Riemannian Mumford-Shah functional,} SIAM Jornal of
Mathematical Analysis Volume 35, Number 6, 1394 - 1419.



\bibitem{Corvin} {\sc L.J.~Corvin,\,F.P.~Greenleaf}, {\em
Representations of nilpotent Lie groups and their applications},
Cambridge University Press (1984).



\bibitem{DanGarN80}{\sc  D.~Danielli,\,N.~Garofalo,\,D.M.~Nhieu}, {\em
Minimal surfaces, surfaces of constant mean curvature and
isoperimetry in Carnot groups}, Preprint (2001).



\bibitem{DanGarN8}\leavevmode\vrule height
2pt depth -1.6pt width 23pt , {\em Sub-Riemannian Calculus on
Hypersurfaces in Carnot groups},  Adv. Math.  215,  no. 1 (2007)
292--378.



\bibitem{DGN3} \leavevmode\vrule height
2pt depth -1.6pt width 23pt, {\em Sub-Riemannian calculus and
monotonicity of the perimeter for graphical strips}, Math. Zeith.
Vol. 264, n.1 (2010)

\bibitem{DaSe} {\sc G.~David,\,S.~Semmes},
{\em ``Fractured Fractals and Broken Dreams. Self-Similar Geometry
through Metric and Measure''}, Oxford University Press (1997).

\bibitem{DG} {\sc E.De~Giorgi},{\em Un progetto di teoria delle correnti,
forme differenziali e variet\`a non orientate in spazi metrici},
in {\em Variational Methods, Non Linear Analysis and Differential
Equations in Honour of J.P. Cecconi}, M.Chicco  et al. Eds. ECIG,
Genova (1993) 67-71.


\bibitem{FE} {\sc H.~Federer}, {\em ``Geometric Measure Theory''}, Springer Verlag (1969).


\bibitem{FedererFleming} {\sc H.~Federer,\,W.H.~Fleming},
 {\em ``Normal and Integral Currents''}, Ann. Math. { 72} (1960)  458-520.




\bibitem{FePh} {\sc C.~Fefferman,\,D.H.~Phong},
 {\em Subelliptic eigenvalue problems}, Conference in harmonic analysis in honor of Antoni Zygmund, Vol. I, II (Chicago, Ill., 1981),
 590-606, Wadsworth Math. Ser., Wadsworth, Belmont, CA, 1983.


\bibitem{FraBal} {\sc A.~Baldi,\,B.~Franchi},
 {\em Differential forms in Carnot groups: a $\Gamma $-convergence approach}, Calc. Var. Partial Differ. Equ. 43, No. 1-2, 211-229 (2012).


\bibitem{FGW} {\sc B.~Franchi, S.~Gallot, \& R.L.~Wheeden}, {\em
Sobolev and isoperimetric inequalities for degenerate metrics},
Math. Ann., {{300}} (1994) 557-571.

\bibitem{FLanc} {\sc B.~Franchi \& E.~Lanconelli}, {\em
H$\ddot{o}$lder regularity theorem for a class of non uniformly
elliptic operators with measurable coefficients}, Ann. Scuola
Norm. Sup. Pisa, {10} (1983) 523-541.












\bibitem{FSSC3}
{\sc B.~Franchi,\, R.~Serapioni,\, F.S.~Cassano}, {\em
Rectifiability and Perimeter in the Heisenberg Group}, Math. Ann.,
{ 321} (2001) 479-531.

\bibitem{FSSC5} \leavevmode\vrule height
2pt depth -1.6pt width 23pt ,  {\em On the structure of finite
perimeter sets in step 2 Carnot groups}, J. Geom. Anal., { 13},
no. 3 (2003) 421-466.






\bibitem{GN} {\sc N.~Garofalo,\, D.M.~Nhieu}, {\em Isoperimetric and
Sobolev inequalities for Carnot-Carath\'eodory spaces and the
existence of minimal surfaces}, Comm. Pure Appl. Math., {{49}}
(1996) 1081-1144.

\bibitem{G} {\sc N.~Garofalo,\,S.~Pauls}, {\em The Bernstein problem in the Heisenberg group,}
arXiv:math.DG/0209065.




\bibitem{Gr1} {\sc M.~Gromov},{\em Carnot-Carath\'eodory spaces seen from
within}, in {\em ``Subriemannian Geometry''}, Progress in
Mathematics, { 144}. ed. by A.Bellaiche and J.Risler, Birkhauser
Verlag, Basel (1996).


\bibitem{HaKo} {\sc P.Haj\l asz, \& P.Koskela}, {\em Sobolev met
Poincare}, Mem. Am. Math. Soc., {688} (2000).

\bibitem{HeSi}{\sc W.~Hebisch,\, A.Sikora},{\em A smooth subadditive homogeneous norm on a homogeneous group
}, Studia Mathematica, T. XCVI (1990).




\bibitem{Helgason}{\sc S.~Helgason}, {\em ``Differential geometry, Lie groups,
and symmetric spaces''}, Academic Press, New York (1978).




\bibitem{HP}{\sc R.H~Hladky,\, S.D.~Pauls},
{\em Variation of Perimeter Measure in sub-Riemannian geometry},
Preprint (2007).




\bibitem{Jost}{\sc J.~Jost,\,X.Li~Jost},
{\em Calculus of Variations},  Cambridge Studies in Advanced Mathematics,
64,  Cambridge University Press, Cambridge, 1998.




\bibitem{KR}
{\sc A.~Kor\'anyi \& H.M.~Reimann}, {\em Foundation for the Theory
of Quasiconformal Mapping on the Heisenberg Group,} Advances in
Mathematics, {{111}} (1995) 1-85.



\bibitem{Lee}
{\sc J.M.~Lee}, {\em ``Introduction to Smooth Manifolds ''}, Springer
Verlag, 2003.
 
 


\bibitem{Mag}
{\sc V.~Magnani}, {\em ``Elements of Geometric Measure Theory on
sub-Riemannian groups''}, {PHD Thesis}, Scuola Normale Superiore
di Pisa, (2002).

\bibitem{Mag2}\leavevmode\vrule height
2pt depth -1.6pt width 23pt, {\em Characteristic points,
Rectifiability and Perimeter measure on stratified groups}, J.
Eur. Math. Soc. (JEMS), { 8}, no. 5 (2006)  {585-609}.

\bibitem{Mag8}
\leavevmode\vrule height 2pt depth -1.6pt width 23pt, {\em Blow-up
estimates at horizontal points and applications}, J. Geom. Anal.
vol. 20, {n. 3}, 705-722 (2010).

\bibitem{Mag4}\leavevmode\vrule height 2pt depth -1.6pt width
23pt, {\em Non-horizontal submanifolds and Coarea Formula}, J.
Anal. Math.  106  (2008), 95--127.

\bibitem{Mag3}{\sc V.~Magnani, \,D.~Vittone},
{\em An intrinsic measure for submanifolds in stratified groups},
J. Reine Angew. Math.  619  (2008) 203--232.









\bibitem{MAZ}
{\sc V.G.~Maz'ya}, {\em Classes of domains and embedding theorem
for functional spaces}, (in Russian) Dokl. Acad. Nauk SSSR, { 133}
(1960) no.3, 527-530. Engl. transl. {\it Soviet Math. Dokl.,} { 1}
(1961) 882-885.

\bibitem{MS}
{\sc J.H.~Michael,\,L.M.~Simon}, {\em Sobolev and Mean-value
Inequalities on Generalized Submanifolds of $\Rn$}, Comm. on Pure
and Applied Math., Vol. XXVI (1973)361-379 .


\bibitem{3}
{\sc J.~Milnor}, {\em Curvatures of left invariant Riemannian
metrics}, {Adv. Math.}, { 21} (1976) {293--329}.







\bibitem{Mi}
{\sc J.~Mitchell}, {\em On Carnot-Carath\'eodory metrics},
J.Differ. Geom. {{21}} (1985) 35-45.


\bibitem{Monte}
{\sc F.~Montefalcone}, {\em Some Remarks in Differential and
 Integral  Geometry of Carnot Groups} {Tesi di Dottorato}--Universit\`{a}
 degli Studi di Bologna-- (2004).
\bibitem{Montea}
{\leavevmode\vrule height 2pt depth -1.6pt width 23pt}, {\em Some
relations among volume, intrinsic perimeter and one-dimensional
restrictions of $\mathit{BV}$ functions in Carnot groups} Annali
della Scuola Normale Superiore di Pisa, Classe di Scienze, { 5},
IV, Fasc.1 (2005) 79-128.

\bibitem{Monteb}{\leavevmode\vrule
height 2pt depth -1.6pt width 23pt}, {\em Hypersurfaces and
variational formulas in sub-Riemannian Carnot groups}, Journal de
Math\'ematiques Pures et Appliqu\'ees,  {87} (2007) 453-494.


\bibitem{MonteAr}
\leavevmode\vrule height 2pt depth -1.6pt width 23pt,
{Isoperimetric, Sobolev and Poincar\'e inequalities on
hypersurfaces in sub-Riemannian Carnot groups}, preprint ArXiv
(2009).


\bibitem{MonteStab}
\leavevmode\vrule height 2pt depth -1.6pt width 23pt,
{Stable  $\HH$-minimal hypersurfaces}, preprint ArXiv
(2012).







\bibitem{Montgomery} {\sc R.~Montgomery}, {\em
``A Tour of Subriemannian Geometries, Their Geodesics and
Applications''}, AMS, Math. Surveys and Monographs, 91 (2002).

\bibitem{P1}{\sc P.~Pansu}, {\em ``G\`eometrie du Group d'Heisenberg''}, Th\`ese
pour le titre de Docteur, 3\`eme cycle, Universite Paris {\rm VII}
(1982). \bibitem{P2} \leavevmode\vrule height 2pt depth -1.6pt
width 23pt, {\em M\'etriques de Carnot-Carath\'eodory et
quasi-isom\'etries des espaces symm\'etriques de rang un,} Ann. of
Math. 2, {{129}} (1989) 1-60.


\bibitem{P4}\leavevmode\vrule
height 2pt depth -1.6pt width 23pt, {\em Submanifolds and
differential forms in Carnot manifolds, after M. Gromov et M.
Rumin}, preprint (2005).


\bibitem{Pauls}
{\sc S.D.~Pauls}, {\em Minimal surfaces in the Heisenberg group},
Geom. Dedicata, { 104} (2004) 201-231.


\bibitem{Pfeffer} {\sc W.~Pfeffer}, {\em
``Derivation and integration''}, Cambridge Tracts in Mathematics. 140. Cambridge: Cambridge University Press. xvi, 266 p.  (2001).


\bibitem{RR}
{\sc M.~Ritor\'e, C.~Rosales}, {\em Area stationary surfaces in
the Heisenberg group $\mathbb{H}^1$}, Adv. Math. 219, no.2
 (2009) 179-192.


\bibitem{SCM} {\sc A.~Sarti, G.~ Citti, M.~Manfredini},
{\em `From neural oscillations to variational problems in the
visual cortex}, special issue of the {Journal of Physiology},
Volume 97, Issues 2-3 (2003) 87-385.



\bibitem{Simon}
{\sc L.M.~Simon}, {\em Lectures on geometric measure theory}, Proceedings of the Centre for Mathematical Analysis, Australian National University, Vol. 3 (1983).




\bibitem{Spiv} {\sc M.~Spivak}, {\em ``Differential Geometry''},
Volume IV, Publisch or Perish (1964).


\bibitem{Stein} {\sc E.M.~Stein}, {\em ``Harmonic Analysis''}, Princeton
University Press (1993).

\bibitem{Stric} {\sc R.S.~Strichartz}, {\em Sub-Riemannian geometry},
J. Diff. Geom., {24} (1986)  221-263;  {\em Corrections,}, J. Diff.
Geom., {30} (1989) 595-596.


\bibitem{Taylor} {\sc M.E.~Taylor}, {\em ``Measure Theory and Integration''}, Graduate Studies in Mathematics, Vol. 76, A.M.S. (2006).


\bibitem{Vara}{\sc V.S.~Varadarajan}, {\em Lie groups, Lie algebras and their representations}, Reprint of the 1974 edition. Graduate Texts in Mathematics. 102. New York, NY: Springer. xiii, 430 p. \$ 29.80 (1984).


\bibitem{Varo}{\sc N.Th.~Varopoulos,\, L.~Saloff-Coste, \,T.~Coulhon}, {\em
Analysis and Geometry on Groups}, Cambridge University Press
(1992).


\bibitem{Ver}{\sc A.M.~Vershik, \,V.Ya~Gershkovich}, {\em
Nonholonomic Dynamical systems, Geometry of Distributions and
Variationals Problems}, in V.I. Arnold, S.P. Novikov (eds.),
Dynamical systems VII, Springer-Verlag (1996). 

}






\end{thebibliography}
\end{document}